\documentclass[12pt]{article}

\usepackage[margin=1.25in]{geometry}

\usepackage{amssymb,amsmath}
\usepackage{amsthm}
\usepackage{hyperref}
\usepackage{mathrsfs}
\usepackage{textcomp}
\usepackage{enumerate}
\usepackage{framed}
\usepackage{ulem}
\usepackage{color,xcolor}

\usepackage{soul}
\soulregister\cite7 
\soulregister\citep7 
\soulregister\citet7 
\soulregister\ref7 
\soulregister\pageref7 

\hypersetup{
    colorlinks,%
    citecolor=blue,%
    filecolor=blue,%
    linkcolor=blue,%
    urlcolor=black
}

\newtheorem{theorem}{Theorem}[section]
\newtheorem{definition}{Definition}[section]
\newtheorem{lemma}{Lemma}[section]
\newtheorem{remark}{Remark}[section]
\newtheorem{proposition}{Proposition}[section]
\newtheorem{corollary}{Corollary}[section]

\numberwithin{equation}{section}

\allowdisplaybreaks[3]
\makeatletter

\newdimen\bibspace
\renewenvironment{thebibliography}[1]{%
 \section*{\refname 
       \@mkboth{\MakeUppercase\refname}{\MakeUppercase\refname}}%
     \list{\@biblabel{\@arabic\c@enumiv}}%
          {\settowidth\labelwidth{\@biblabel{#1}}%
           \leftmargin\labelwidth
           \advance\leftmargin\labelsep
           \itemsep\bibspace
           \parsep\z@skip     %
           \@openbib@code
           \usecounter{enumiv}%
           \let\p@enumiv\@empty
           \renewcommand\theenumiv{\@arabic\c@enumiv}}%
     \sloppy\clubpenalty4000\widowpenalty4000%
     \sfcode`\.\@m}
    {\def\@noitemerr
      {\@latex@warning{Empty `thebibliography' environment}}%
     \endlist}

\makeatother

           \newcommand{\ud}{\mathrm{d}}
\newcommand{\be}{\begin{equation}}      \newcommand{\ee}{\end{equation}}

              \newcommand{\sn}{\mathbb{S}^n}

\newcommand{\alp}{\ensuremath{\alpha}}
\newcommand{\bt}{\ensuremath{\beta}}
\newcommand{\wdt}{\ensuremath{\widetilde}}
\newcommand{\ub}{\ensuremath{\mathrm{B}}}

\begin{document}

\title{Compactness and existence results of the prescribing fractional $Q$-curvatures problem
 on $\mathbb{S}^n$}

\author{ {\sc Yan Li}\,, {\sc Zhongwei Tang}\thanks{The research was supported by National Science Foundation of China(12071036,12126306)}\, and {\sc Ning Zhou} \\
\small School of Mathematical Sciences, \\
\small Laboratory of Mathematics and Complex Systems, MOE,\\
\small Beijing Normal University, Beijing, 100875, P.R. of China}

\date{}

\maketitle

\begin{abstract}
This paper is devoted to establishing the compactness and existence results
of the solutions to the prescribing fractional $Q$-curvatures problem of order $2\sigma$
on $n$-dimensional standard sphere
when  $ n-2\sigma=2$, $\sigma=1+m/2,$ $m\in \mathbb{N}_{+}.$
The compactness results are novel and optimal.
In addition, we prove a degree-counting formula of all solutions
to achieve the existence.
From our results, we can know where blow up occur.
Furthermore, the sequence of solutions
that blow up precisely at any finite distinct
location can be constructed.
It is worth noting that our results include the case of multiple harmonic.
\end{abstract}


\section{Introduction}
The study of the prescribing scalar curvature problem on
Riemannian manifolds, which dates back to \cite{kwj,kwa,kwa1},
has received a lot of attention. In the case of $n$-dimensional standard
sphere $(\mathbb{S}^{n},g_{0}),$ this is known as Nirenberg problem.
The classical Nirenberg problem is as follows:
which function $K$ on $(\mathbb{S}^{n},g_{0})$
is the scalar curvature
(Gauss curvature in dimension $n=2$)
of a metric $g$ that is conformal to $g_{0}$?
If we denote $g=e^{2v} g_{0}$
in the two dimensional case and
$g=v^{\frac{4}{n-2}}g_{0}$
in the $n$ $(n\geq 3)$ dimensional case,
 this problem is equivalent to solving the following nonlinear elliptic equations:
\be\label{1.69}
-\Delta_{g_0}v+1 =K e^{2 v} \quad \text { on }\, \mathbb{S}^{2},\\
\ee
and
\be\label{1.70}
-\Delta_{g_0}v+c(n)R_0 v=c(n) K v^{\frac{n+2}{n-2}} \quad \text { on }\, \mathbb{S}^{n},\quad n \geq 3,
\ee
where $\Delta_{g_{0}}$ is the Laplace-Beltrami operator, $c(n)=\frac{n-2}{4(n-1)},$
 $R_0=n(n-1)$ is the scalar curvature associated to $g_0.$

A first answer to   the Nirenberg problem was given by
Koutroufiotis \cite{Kou}, which established  the existence
of the solutions to \eqref{1.69}
 by assuming that $K$ is
an antipodally symmetric function which close to 1.
Morse \cite{MOn1973} proved the existence of antipodally
symmetric solutions to \eqref{1.69}  for all antipodally
symmetric functions $K$ which are positive somewhere.
Chang and Yang \cite{CYPrescribing1987}
further extended this existence result to the case of
$K$ without making any symmetry assumptions.
Moreover, Bahri and Coron \cite{BC} presented a sufficient condition
for the existence of solutions to  \eqref{1.70} in
dimension $n=3$.
As for  the compactness of all solutions in dimensions $n=2,3,$
 Chang et al. \cite{CGY}, Han \cite{Han}, and
Schoen and Zhang \cite{SZPrescribed1996} proved
that  a sequence of solutions cannot blow up at
more than one point.
Li \cite{LYYJ, LYY}  established the compactness
and existence results for \eqref{1.70}.
In these two papers, the compactness result is very
different from the previous low-dimensional case.
In fact, when $n=2$ or $n=3,$ a sequence of solutions
cannot blow up at more than one point.
However, if $n>3$, there could be blow up at many points,
which considerably complicates the study of the problem.
There have been many papers on the problem and related ones, see e.g.,
\cite{cl,cd,em,hl,ji,msa,msj,wy}.

The linear operators defined on left-hand  side
of \eqref{1.69} and \eqref{1.70} are called the conformal Laplacian
associated to the metric $g_{0}$ and are denoted as $P_{1}^{g_{0}}.$
For any Riemannian manifold $(M,g),$
let $R_{g}$ be the scalar curvature of $(M,g),$
and the conformal Laplacian be defined as $P_{1}^{g}=-\Delta_{g}+\frac{n-2}{4(n-1)}R_{g}.$
The Paneitz operator $P_{2}^{g}$ is
another conformal invariant operator, which was discovered by Paneitz \cite{Pa}.
Graham et al. \cite{GJMS} constructed a sequence of conformally
covariant elliptic operators $\{P_{k}^{g}\}$
 on Riemannian manifolds for all positive integers $k$ if
$n$ is odd, and for $k\in\{ 1,\cdots,n/2\}$  if $n$ is even,
which are called GJMS operators.
Juhl \cite{ju1,ju2}  found  an explicit formula and a recursive formula
for GJMS operators and $Q$-curvatures (see also Fefferman and
 Graham \cite{FGJuhl2013}).
 Graham and Zworski \cite{GZ} presented a family
  of fractional order conformally invariant operators
  $P_{\sigma}^{g}$  of non-integer order
 $\sigma\in (0,n/2)$
  on the conformal infinity of asymptotically hyperbolic manifolds.
 In addition, Chang and Gonz\'alez \cite{CG}
 showed that the operator $P_{\sigma}^{g},$
 $\sigma\in (0,n/2)$ can  be defined as a Dirichlet-to-Neumann operator
  of a conformally compact
Einstein manifold by using localization method in \cite{CS},
they also provided some new interpretations and properties of those fractional
operators and their associated fractional $Q$-curvatures.

Regarded as a generalization of Nirenberg problem,
the prescribing fractional $Q$-curvature problem
 of order $2\sigma$ on $\mathbb{S}^{n}$
 can be described as:
which function $K$ on $\mathbb{S}^n$ is the fractional $Q$-curvature of a
metric $g$ on $\mathbb{S}^n$  conformally equivalent to $g_0?$
If we denote $g=v^{4/(n-2\sigma)}g_{0},$
this problem can be represented  as finding the solution of the following
nonlinear equation with critical exponent:
\be\label{1.1}
P_{\sigma}^{g_{0}}(v)=c(n, \sigma)
K v^{\frac{n+2 \sigma}{n-2 \sigma}} \quad \text { on }\, \mathbb{S}^{n},
\ee
where $n\geq 2,$ $0<\sigma<n/2,$
$c(n,\sigma)=\Gamma(\frac{n}{2}+\sigma)/\Gamma(\frac{n}{2}-\sigma),$
 $\Gamma$ is the Gamma function,  $K$ is a function defined on $\mathbb{S}^n,$
 $P_{\sigma}^{g_{0}}$ is an intertwining operator of   $2\sigma$-order:
$$
P_{\sigma}^{g_{0}}=\frac{\Gamma(B+\frac{1}{2}+\sigma)}{\Gamma(B+\frac{1}{2}-\sigma)}, \quad B=\sqrt{-\Delta_{g_{0}}
+\Big(\frac{n-1}{2}\Big)^{2}}.
$$
 In what follows, $P_{\sigma}^{g_{0}}$ is simply written  as $P_{\sigma}.$
It can be viewed as the pull back operator of the $\sigma$
power of the Laplacian $(-\Delta)^{\sigma}$ on $\mathbb{R}^{n}$
via the stereographic projection:
$$
(P_{\sigma}(v)) \circ F=|J_{F}|^{-\frac{n+2 \sigma}{2 n}}(-\Delta)^{\sigma}(|J_{F}|^{\frac{n-2 \sigma}{2 n}}(v \circ F)) \quad \text { for } \, v \in C^{2\sigma}(\mathbb{S}^{n}),
$$
where $F$ is the inverse of the stereographic projection and $|J_{F}|$
is the determinant of the Jacobian of $F$.
In addition, the Green function of $P_{\sigma}$ is the spherical
Riesz potential, i.e.,
\be\label{2.27}
P_{\sigma}^{-1}f(\xi)=c_{n,\sigma}
\int_{\sn}\frac{f(\zeta)}{|\xi-\zeta|^{n-2\sigma}}\ud vol_{g_{\sn}}(\zeta)
\quad \text{ for }\, f\in L^{p}(\sn),
\ee
where $c_{n,\sigma}
=\frac{\Gamma(\frac{n-2\sigma}{2})}{2^{2\sigma}\pi^{n/2}\Gamma(\sigma)},$
$p>1,$ and $|\cdot|$ is the Euclidean distance in $\mathbb{R}^{n+1}.$

Many research have been conducted on the
fractional operators $P_{\sigma}^{g}$ and their
associated fractional $Q$-curvature,
 for instance, see \cite{AC,Chti1,clz,Chti2,CROn2011,
DMAPrescribingI2002,DMAPrescribingII2002,ERMountain2002,
JLXOn2014,JLXOn2015,jlxm,Liu_cpaa}.
The flatness of the prescribing fractional $Q$-curvature function $K$
plays a crucial role in the study of this problem.
We begin with the definition of the $\beta$-flatness condition that characterizes flatness.

{\bf $\beta$-flatness condition:}
Let $ K\in C^{1}(\mathbb{S}^n)$
($K\in C^{1,1}(\mathbb{S}^{n})$ if $0<\sigma\leq1/2$)
 be a positive function and $\beta $ is
a positive constant,
we say that $K$ satisfies the $\beta$-flatness condition if for
every critical point $\xi_{0}$ of $K,$ in some geodesic normal
coordinates $\{y_{1}, \cdots, y_{n}\}$ centered at $\xi_{0}$,
there exists a small neighborhood $\mathscr{O}$ of $0$  and
$a_{j}(\xi_{0})\ne 0,$ $\sum_{j=1}^{n}a_j(\xi_{0})\neq 0,$
such that
$$
K(y)=K(0)+\sum_{j=1}^{n} a_{j}(\xi_{0})|y_{j}|^{\beta}+R(y) \quad \text { in } \mathscr{O},
$$
where $\sum_{s=0}^{[\beta]}
|\nabla^s R(y)||y|^{-\bt+s}\rightarrow 0$ as $y\rightarrow 0,$
here $\nabla^{s}$ denotes all possible derivatives of order $s$ and
$[\beta]$ is the integer part of $\beta.$
 We call $\beta$ the  flatness order.

For $\sigma\in(0,1)$ and $\beta\in(n-2\sigma, n),$ Jin et al. \cite{JLXOn2014, JLXOn2015}
proved the existence of the solutions to \eqref{1.1}
and derived some compactness properties  when $K$
satisfies the $\beta$-flatness condition by using
the approach based on approximation of the solutions
to \eqref{1.1} by a blow up
subcritical method.
For $\sigma\in(0,n/2)$ and $\beta\in(n-2\sigma,n),$
Jin et al. \cite{jlxm} developed a unified approach to establish blow up profiles,
compactness and existence of positive solutions to \eqref{1.1} when $K$ satisfies
$\beta$-flatness condition by making use of integral representations.
Since their conclusions are valid
only when the flatness order
$n-2\sigma<\beta<n,$ some very interesting functions $K$ are excluded.
In fact, note that an important class of functions,
which is worth including in the results of existence and compactness
for \eqref{1.1},
are the Morse functions with only non-degenerate critical points.
Such functions  satisfy the $2$-flatness condition.

Existence results of the solutions to \eqref{1.1}
were given  when $\beta\in (1,n-2\sigma]$ by Abdelhedi et al. \cite{Chti1},
 and  when $\beta\in [n-2\sigma, n)$ by Chtioui  and Abdelhedi  \cite{Chti2}.
Under a so-called ``non-degenerate condition'',
Khadijah and Chtioui \cite{kc} studied the lack of compactness and provided
the existence results for \eqref{1.1} when $\beta=n-2\sigma=2,$ $\sigma\in (0,n/2).$

However, under the  assumption of the flatness order $\beta=n-2\sigma$
of prescribing curvature function $K,$
the precise  compactness results of the solutions to \eqref{1.1} are unknown.
When $\sigma=1$ and $n=2\sigma+2=4,$
 Li \cite{LYY} obtained  the optimal compactness and
 a degree-counting formula of the solutions to \eqref{1.70}
when $K$ is some special class of functions satisfying condition $2=n-2\sigma$-flatness condition.
Therefore, a quite natural question arises: can we establish
the optimal compactness results and  provide
a degree-counting formula of the solutions to \eqref{1.1}
when the curvature function $K$ is specified as a
special function satisfying the $\beta$-flatness condition with $\beta=n-2\sigma=2$?
The main target of this article is to give an affirmative answer to this question.

In the present paper, we are interested to the prescribing fractional $Q$-curvature problem
\eqref{1.1}, $n=2\sigma+2,$ $\sigma=1+m/2,$ $m\in \mathbb{N}_{+}.$
Our aim is to establish the optimal compactness and existence results
 of the solutions, when the
prescribing curvature function $K$ is some special function
satisfying $2=n-2\sigma$-flatness condition.
In order to obtain an existence result, we will prove a degree-counting
formula of the solutions to \eqref{1.1}.
This counting formula, together with the compactness results completely describes
where blow up occur. Especially, from our results,
we can construct a sequence of solutions to \eqref{1.1}
 that blow up precisely at these points for any finite distinct points
on $\mathbb{S}^n.$

First of all, Eq. \eqref{1.1} is not always solvable.
Indeed, we have the Kazdan-Warner type obstruction:
for any conformal Killing vector field $X$ on $\mathbb{S}^{n}$, there holds
$$
\int_{\mathbb{S}^{n}}(\nabla_{X} K) v^{\frac{2 n}{n-2 \sigma}} \mathrm{d} v o l_{g_{\mathbb{S}^{n}}}=0
$$
for any solution $v$ of \eqref{1.1},
see \cite{be,xu}.

Before state our results, we introduce some definitions and notations.

For $K \in C^{2}(\mathbb{S}^{n}),$
we introduce the following notation:
\be\label{1.2}
\begin{aligned}
\mathscr{K}&=\{q \in \mathbb{S}^{n}: \nabla_{g_{0}} K(q)=0\}, \\
\mathscr{K}^{+}&=\{q \in \mathbb{S}^{n}: \nabla_{g_0} K(q)=0,\,
\Delta_{g_0} K(q)>0\}, \\
\mathscr{K}^{-}&=\{q \in \mathbb{S}^{n}: \nabla_{g_{0}} K(q)=0,\, \Delta_{g_{0}} K(q)<0\},\\ \mathscr{M}_{K}&=\{v \in C^{2\sigma}(\mathbb{S}^{n}): v \text { satisfies \eqref{1.1}} \}.
\end{aligned}
\ee

For any $k$ $(k \geq 1)$ distinct points $q^{(1)}, \cdots, q^{(k)}
\in \mathscr{K} \backslash \mathscr{K}^{+},$ the
 $k \times k$ symmetric matrix $M=(M(q^{(1)}, \cdots, q^{(k)}))$ is defined by
\be\label{M}
\begin{aligned}
M_{ii}&=-
 \frac{\Delta_{g_{0}}K(q^{(i)})}
{K(q^{(i)})^{{n}/{2\sigma}}},
\\
M_{i j}&
=-n(n-1)\frac{G_{q^{(i)}}(q^{(j)})}
{(K(q^{(i)}) K(q^{(j)}))^{{1}/{2\sigma}}},\quad i\ne j,
\end{aligned}
\ee
where
\be\label{1.8}
G_{q^{(i)}}(q^{(j)})
=\frac{1}{1-\cos d(q^{(i)},q^{(j)})}
\ee
 is the Green's function of $P_{\sigma}$ on $\mathbb{S}^n,$
 and  $d(\cdot\,,\,\cdot)$ denotes the geodesic distance.
Let $\mu(M)$ denote the smallest eigenvalue of $M$, and
when $k=1,$
 $$
 \mu(M)=M=-\frac{\Delta_{g_0}K(q^{(1)})}{K(q^{(1)})^{{n}/{2\sigma}}}.
 $$

In what follows, we define
\begin{align}
\begin{aligned}\label{1.80}
C^{2}(\mathbb{S}^{n})^{*}:=
&\{ K \in C^{2}(\mathbb{S}^{n}): K>0\, \text{ on } \,\mathbb{S}^{n}, \text{ and }\\
&\quad K \text{ has only non-degenerate critical points}\},
\end{aligned}
\end{align}
and
\begin{align}\label{1.79}
\begin{aligned}
\mathscr{A}=&\left\{K\in C^{2}(\mathbb{S}^{n})^{*}:
\Delta_{g_{0}} K \neq 0 \text { on } \mathscr{K}, \text{ and }\right.\\
&\left.\quad \mu(M(q^{(1)}, \cdots, q^{(k)})) \neq 0, \forall\, q^{(1)},
\cdots, q^{(k)} \in \mathscr{K}^{-}, k \geq 2\right\}.
\end{aligned}
\end{align}
We can observe  that $\mathscr{A}$ is open in $C^{2}(\mathbb{S}^{n})$ and
   $\mathscr{A}$ is dense in $C^2(\mathbb{S}^n)^*$
with respect to the $C^{2}(\mathbb{S}^{n})$ norm.

\begin{remark}
In this paper, we mainly establish the compactness and existence results for
\eqref{1.1} when $K\in \mathscr{A}.$
It is worth noting that the sign of the smallest eigenvalue of $M(q^{(1)},\cdots, q^{(k)})$
plays a key  role in counting formula of all sloutions and compactness results.
\end{remark}

We will introduce an integer-valued continuous function Index: $\mathscr{A}\rightarrow\mathbb{Z},$
which has an explicit formula for $K\in \mathscr{A}$ being a Morse function.

\begin{definition}\label{defn1.2}
We define $\mathrm{Index}$: $\mathscr{A}\rightarrow \mathbb{Z}$ by the following properties:
\begin{enumerate}
\item [(i)] For any Morse function $K\in \mathscr{A}$ with $\mathscr{K}^{-}=\{ q^{(1)},\cdots, q^{(s)}\},$
we define
\be\label{index}
\mathrm{Index}(K)=
-1+\sum_{k=1}^{s}
\sum_{\substack{\mu(M(q^{(i_1)},\cdots, q^{(i_{k})}))>0,\\
1\leq i_{1}< \cdots< i_k \leq s}}(-1)^{k-1+\sum_{j=1}^{k}i(q^{(i_{j})})},
\ee
where $i(q^{(i_{j})})$ denotes the Morse index of $K$ at $q^{(i_{j})}.$

\item [(ii)] $\mathrm{Index}:$ $\mathscr{A}\rightarrow \mathbb{Z}$ is continuous with respect to the $C^{2}(\mathbb{S}^n)$ norm of $\mathscr{A}$ and hence is locally constant.
\end{enumerate}
\end{definition}

\begin{remark}
The existence and  uniqueness of the $\mathrm{Index}$ mapping  follows
from Theorem \ref{thm2.1} and the proof of Theorem \ref{thm4} below.
\end{remark}

Our first result is about the compactness of the solutions
when $K\in \mathscr{A},$
 which is:
\begin{theorem}\label{thm2.1}
Let $\sigma=1+m/2,$ $m\in \mathbb{N}_{+}$ and
$n=2\sigma+2.$ Let $\mathscr{A}$ be as in \eqref{1.79} and $K\in \mathscr{A}.$  Then
for any $\alpha\in(0,1),$
there exists constants
$\delta=\delta(K)>0$ and $C=C(K)>0,$ such that for any $\mathcal{K}\in C^{2}(\mathbb{S}^{n})$
satisfying $\|\mathcal{K}-K\|_{C^{2}(\mathbb{S}^{n})}<\delta,$
 and any $v\in \mathscr{M}_{\mathcal{K}},$
 we have
\begin{align}
v\in C^{2\sigma,\alpha}(\mathbb{S}^n):
1/C<v<C,\, \|v\|_{C^{2\sigma,\alpha}(\mathbb{S}^n)}<C,
\end{align}
 where $\mathscr{M}_{\mathcal{K}}$ is as in \eqref{1.2}.
\end{theorem}

For any given $\sigma=\frac{n-2}{2},$ $0<\alpha<1,$ $R>0,$   we define
\be\label{1.82}
\mathscr{O}_{R}:=\{v\in C^{2\sigma,\alpha}(\mathbb{S}^n):
1/R<v<R,\, \|v\|_{C^{2\sigma,\alpha}(\mathbb{S}^n)}<R\}.
\ee

Our second  result is about  degree-counting  formula and the existence of the solutions to \eqref{1.1},
which is:
\begin{theorem}\label{thm4}
Let $\sigma=1+m/2,$ $m\in \mathbb{N}_{+}$ and $n=2\sigma+2.$
Let $\mathscr{A}$ be as in \eqref{1.79}, $K\in \mathscr{A}$
and $\mathrm{Index}(K)$ be as in Definition \ref{defn1.2}.
Then for any $\alpha\in (0,1),$ there exists a constant $R_{0}=R_{0}(K,\alpha),$ such that for all $R>R_{0},$
we have
\be\label{1.38}
\deg_{C^{2\sigma,\alpha}}(v-P_{\sigma}^{-1}(c(n,\sigma)Kv^{\frac{n+2\sigma}{n-2\sigma}}), \mathscr{O}_{R}, 0)=\mathrm{Index}(K),
\ee
where  $\deg_{C^{2\sigma,\alpha}}$ denotes the
Leray-Schauder degree in $C^{2\sigma,\alpha}(\mathbb{S}^n).$

Furthermore, if $\mathrm{Index}(K)\ne 0,$ then
\eqref{1.1} has at least one solution.
\end{theorem}

\begin{remark}\label{re1}
It follows from Theorem \ref{thm1} that when $K\in\mathscr{A},$ the solutions to \eqref{1.1}  belong to $\mathscr{O}_{R}$ for some
$R>0.$ We call the left-hand side of \eqref{1.38}  the total degree of the solutions to the
conformally invariant equation.
 From Theorem \ref{thm4}, the total degree is $\mathrm{Index}(K).$
\end{remark}

For any finite subset $\mathcal{R}\subset \mathbb{S}^{n},$
we use $\sharp\mathcal{R}$ to denote the number of elements in the set $\mathcal{R}$.
Let us now state a corollary of Theorem \ref{thm4}, which is:
\begin{corollary}\label{cor1}
Let $\sigma=1+m/2,$ $m\in \mathbb{N}_{+}$ and $n=2\sigma+2.$
Let $\mathscr{A}$ be as in \eqref{1.79} and $K\in \mathscr{A}$ be a
 Morse function satisfying $\sharp\mathscr{K}^{-}\leq 1$ or
for any distinct $P,$ $Q\in \mathscr{K}^{-},$
\be\label{1.97}
\Delta_{g_{0}}K(P)\Delta_{g_{0}}K(Q)<\frac{n^2(n-1)^2}{4}K(P)K(Q).
\ee
Then for any $\alpha\in(0,1),$ there exists a constant $C=C(K,\alpha)>0,$
such that for all solutions $v$ to \eqref{1.1}, we have
$v\in \mathscr{O}_{C},$
and for all $R\geq C,$
$$
\deg_{C^{2\sigma,\alpha}}
(v-P_{\sigma}^{-1}(c(n,\sigma)Kv^{\frac{n+2\sigma}{n-2\sigma}}), \mathscr{O}_{R}, 0)
=-1+\sum_{\substack{\nabla_{g_0}K(q_{0})=0,
\\ \Delta_{g_{0}}K(q_0)<0}}(-1)^{i(q_{0})},
$$
where $\mathscr{O}_{C}$ is as in \eqref{1.82} and
$i(q_0)$ denotes the Morse index of $K$ at $q_{0}.$

Furthermore, if
$$
\sum_{\substack{\nabla_{g_0}K(q_{0})=0,
\\ \Delta_{g_{0}}K(q_0)<0}}(-1)^{i(q_{0})}\ne 1,
$$
then \eqref{1.1} has at least one solution.
\end{corollary}

Our third result is about the blow up behavior of the solutions when the
prescribing fractional $Q$-curvature function
$K\in C^{2}(\mathbb{S}^{n})^{*}\backslash\mathscr{A}=\partial \mathscr{A},$ which is:
\begin{theorem}\label{thm2}
Let $\sigma=1+m/2,$ $m\in \mathbb{N}_{+}$ and $n=2\sigma+2.$
Let $\mathscr{A}$ be as in \eqref{1.79} and $C^{2}(\mathbb{S}^n)^{*}$ be as in \eqref{1.80}.
Then for any $K \in C^{2}(\mathbb{S}^{n})^{*} \backslash \mathscr{A}=\partial \mathscr{A},$
there exists $K_{i} \rightarrow K$ in $C^{2}(\mathbb{S}^{n})$ and $v_{i}
\in \mathscr{M}_{K_{i}},$ such that
\be\label{1.72}
\lim _{i \rightarrow \infty}(\max _{\mathbb{S}^{n}} v_{i})=\infty,
\quad \lim _{i \rightarrow \infty}(\min _{\mathbb{S}^{n}} v_{i})=0,
\ee
where  $\mathscr{M}_{K_{i}}$ is as in \eqref{1.2}.
\end{theorem}

From Theorems \ref{thm2.1}, \ref{thm4}, and \ref{thm2}, we can know that
the total degree of solutions to \eqref{1.1} strongly depend
on the sign of the  smallest eigenvalue  of $M(q^{(1)},\cdots, q^{(k)}).$
 In fact, the points $q^{(1)}, \cdots, q^{(k)}$ for which
 $\mu(M(q^{(1)}, \cdots, q^{(k)}))$ is positive characterize the so-called asymptotic in the theory of critical points at infinity developed by Bahri \cite{Bahri,
 BC}. For instance, considering a continuous family of functions $K_{t}$ $(0\leq t\leq 1),$ the total degree changes when the smallest eigenvalue of $M(K_{t};(q^{(1)}, \cdots, q^{(k)}))
 $ crosses zero while it remains unchanged when other eigenvalues cross zero.

It follows from Theorem \ref{thm2} that  when $K\in C^{2}(\mathbb{S}^{n})^{*}
\backslash\mathscr{A},$ the solutions
to \eqref{1.1} may blow  up. A natural question is where the  blow up occur?
The following results present the accurate location of the blow up.

For any $K\in C^{2}(\mathbb{S}^{n}),$  we first define
\be\label{1.81}
\begin{gathered}
\mathscr{H}(K)=\{(q^{(1)}, \cdots, q^{(k)}):\,k \geq 1,\,q^{(j)} \in \mathscr{K} \backslash \mathscr{K}^{+},\, \forall\, j: 1 \leq j \leq k, \\
\quad q^{(j)} \neq q^{(\ell)},\, \forall\, j \neq \ell,\, \mu(M(q^{(1)}, \cdots, q^{(k)}))=0\}.
\end{gathered}
\ee

Combined with Theorem \ref{thm1},
we give the fourth result in this paper, which is about the location of blowing up when
$K\in C^{2}(\mathbb{S}^{n})^{*}\backslash\mathscr{A}$ :
\begin{theorem}\label{thm5}
Let $\sigma=1+m/2,$ $m\in \mathbb{N}_{+}$ and $n=2\sigma+2.$
Let $\mathscr{A}$ be as in \eqref{1.79} and $C^{2}(\mathbb{S}^n)^{*}$ be as in \eqref{1.80}.
For a given function $K\in C^{2}(\mathbb{S}^n)^*\backslash\mathscr{A},$ we have  the following results:
\begin{enumerate}
  \item[(i)]
  For any $K_{i}\rightarrow K$ in $C^{2}(\mathbb{S}^n),$
and $v_{i}\in \mathscr{M}_{K_{i}}$ with $\max_{\mathbb{S}^n} v_{i}\rightarrow \infty,$
then for some $ (q^{(1)},\cdots, q^{(k)}) \in \mathscr{H}(K),$ $\{v_{i}\}$ (after passing to
a subsequence) blows up at precisely the $k$ points.
  \item[(ii)]
 For any  $(q^{(1)},\cdots, q^{(k)}) \in\mathscr{H}(K),$ there exists $K_{i}\rightarrow K$ in $C^{2}(\mathbb{S}^n),$
$v_{i}\in \mathscr{M}_{K_{i}},$ such that $\{v_{i}\}$ blows up at precisely the $k$ points.
\end{enumerate}
\end{theorem}

\begin{corollary}
For any $k\in \mathbb{N}_{+}$  distinct points $q^{(1)}, \cdots, q^{(k)}\in \mathbb{S}^n,$
  there exists a sequence of Morse functions $\{K_{i}\}\subset \mathscr{A},$ such that
  for some $v_{i}\in \mathscr{M}_{K_{i}},$ $\{v_{i}\}$ blows up at precisely the $k$ points.
\end{corollary}

In order to obtain the compactness results,
we need to further characterize the behavior of the blow up point of the solutions
to \eqref{1.1} (see Theorem \ref{thm1} below).
More precisely, we will use the Pohozaev type identity
(see Proposition \ref{prop1.1} below)
to  judge the sign of the Laplacian of the prescribing curvature function at these
isolated simple blow up point (see Definition \ref{defn1.3} below). Due to the limit of the form of
the Pohozaev type identity, the proof method is only effective for the case $n-2\sigma=2.$
In addition, when proving the existence results,
we transform the conclusion to be proved into solving the
Brouwer degree of the operator on finite dimensional manifolds
through the homotopy invariance
of the Leray-Schauder degree.
In the process of solving,  we need to get a strictly convex function according to
the form of the operator,
and the condition ``$n-2\sigma=2$'' just ensures the
existence of the form of strictly convex function.
For $n=2\sigma+2,$ $0<\sigma<1,$ we obtain the corresponding
compactness and existence results with $n=3,\sigma=1/2,$
see \cite{ltz}.

The paper is organized as follows:

 In Section \ref{sec3}, our main task is to prove Theorem \ref{thm2.1}.
 Before that, we should  further characterizes  the behavior of blow up points for solutions  to \eqref{1.1}
 (see Theorem \ref{thm1} below), we mainly consider the subcritical equation with $\tau>0$ small:
\be\label{1.92}
P_{\sigma}v_{i}=c(n,\sigma)Kv_{i}^{n-1-\tau}, \quad v_{i}>0 \quad \text{ on } \, \mathbb{S}^n.
\ee
In proving Theorem \ref{thm1}, we first use
 the Green's representation \eqref{2.27}
to transform \eqref{1.92}  into
$$
v_{i}(\xi)=\frac{\Gamma(\frac{n+2 \sigma}{2})}{2^{2 \sigma} \pi^{n / 2} \Gamma(\sigma)}
\int_{\mathbb{S}^{n}} \frac{K_{i}(\eta) v_{i}(\eta)^{n-1-\tau_{i}}}{|\xi-\eta|^{2}}\,
 \ud \eta \quad \text { on }\, \mathbb{S}^{n},
$$  and
then use some results of blow up analysis given in Appendix \ref{sec2}
to complete the proof.
By using Theorem \ref{thm1},  integral representation,  Harnack inequality
and Schauder type estimates,  we have completed the proof of Theorem \ref{thm2.1}.

Section \ref{sec4} is devoted to proving the
Theorems \ref{thm4}, \ref{thm2}, and \ref{thm5}.
Firstly, recall the classification  of solutions for integral
equation (\cite{clo}) and  optimal representation in small tubular neighborhood
(\cite{BC}), we give the definition of
 $\Sigma_{\tau}=\Sigma_{\tau}(\overline{P}_{1},\cdots, \overline{P}_{k})$
  for $\overline{P}_{1},\cdots, \overline{P}_{k}\in \mathscr{K}^{-}$
   with $\mu(M(\overline{P}_{1},\cdots, \overline{P}_{k}))>0.$
   Then by using Theorem \ref{thm1}
 and some results in \cite{jlxm}, we obtain that for $\tau>0$ very small,
  the solutions to \eqref{1.92}
 either stay bounded or stay  in one of the $\Sigma_{\tau}$
 (see Proposition \ref{prop3} below).
 Furthermore,  we obtain the $H^{\sigma}$
  topological degree of the solutions to \eqref{1.92} on
$\Sigma_{\tau}\,($see Theorem \ref{thm3} below$).$
It follows from the above results that  for all $0<\tau<2,$
 the $H^{\sigma}$ total degree of the solutions to
 \eqref{1.92} is equal
 to $-1$ (see Proposition \ref{prop4} below).
Then we can conclude that $H^{\sigma}$ topological
degree of those solutions to \eqref{1.92} which remain
bounded as $\tau$ tends to zero is equal to $\mathrm{Index}(K).$
Some well-known results in degree theory
imply that the $H^{\sigma}$ degree contribution above is equal
 to the $C^{2\sigma,\alpha}$ topological degree
 of those bounded solutions to \eqref{1.92}.
Thus,  we proved Theorem \ref{thm4}.
Furthermore, we complete the proof of
Theorem \ref{thm2} by using the degree-counting
formula and perturbing the function $K$ near its critical point.
In the end, using Theorem \ref{thm1} and the idea
of the proof of Theorem \ref{thm2}, we prove  Theorem \ref{thm5}.

In Appendix \ref{sec2},  since the fact that
by the Green's representation \eqref{2.27}
and the stereographic projection, we can write
Eq. \eqref{1.1} as the form
\begin{align}\label{1.90}
u(x)=\int_{\mathbb{R}^{n}}
\frac{K(y) u(y)^{\frac{n+2\sigma}{n-2\sigma}}}
{|x-y|^{n-2\sigma}}\,
 \ud y \quad \text { on }\, \mathbb{R}^{n},
\end{align}
we first review the H\"older estimates, Schauder type estimates,
blow up profile for nonlinear integral equations \eqref{1.90}
 established by Jin-Li-Xiong \cite{jlxm}.

In  Appendix \ref{sec5},
we provide some useful technical
results and elementary estimates.

\section{The characterization of blow up behavior and  compactness result}\label{sec3}
In this section, our main task is to prove Theorem \ref{thm2.1}.
Before that, we need further characterizes the blow up points for solutions to \eqref{1.1}
by using integral representation
and some estimates in the Appendix \ref{sec2} (see Theorem \ref{thm1} below),
 which  plays a key role in proving main result concerning compactness and existence.
We first review some definitions of blow up points.

Let $\Omega$ be a domain in $\mathbb{R}^{n}$ and $K_{i}$ are
 nonnegative bounded functions in $\mathbb{R}^{n}.$
 Let $\{\tau_{i}\}_{i=1}^{\infty}$ be a sequence of
 nonnegative constants satisfying $\lim _{i \rightarrow \infty} \tau_{i}=0$, and set
$$
p_{i}=\frac{n+2 \sigma}{n-2 \sigma}-\tau_{i}.
$$
Suppose that $0 \leq u_{i} \in L_{{loc}}^{\infty}(\mathbb{R}^{n})$
satisfies the nonlinear integral equation
\be\label{2.1}
u_{i}(x)=\int_{\mathbb{R}^{n}} \frac{K_{i}(y) u_{i}(y)^{p_{i}}}{|x-y|^{n-2 \sigma}}\, \ud y
\quad \text { in } \,\Omega.
\ee
We assume that $K_{i} \in C^{1}(\Omega)$ $(K_{i}\in C^{1,1}(\mathbb{S}^{n})$ if
$\sigma\leq 1/2$) and, for some positive constants $A_{1}$ and $A_{2}$,
\be\label{2.2111}
1 / A_{1} \leq K_{i}, \quad \text { and } \quad
\|K_{i}\|_{C^{1}(\Omega)} \leq A_{2},\,( \|K_{i}\|_{C^{1,1}(\Omega)} \leq A_{2}\, \text { if }\,\sigma \leq \frac{1}{2}).
\ee

\begin{definition}
  Suppose that $\{K_{i}\}$ satisfies \eqref{2.2111} and $\{u_{i}\}$ satisfies \eqref{2.1}.
  A point $\overline{y} \in \Omega$ is called a blow up point of $\{u_{i}\}$ if there exists a sequence
  $y_{i}$ tending to $\overline{y}$ such that $u_{i}(y_{i}) \rightarrow \infty$.
\end{definition}

\begin{definition}\label{defn1.1}
A blow up point $\overline{y} \in \Omega$ is called an isolated blow up point of $\{u_{i}\}$ if there exists $0<\overline{r}<\operatorname{dist}(\overline{y}, \Omega)$, $\overline{C}>0$, and $a$
sequence $y_{i}$ tending to $\overline{y}$, such that $y_{i}$ is a local maximum point of $u_{i},$ $u_{i}(y_{i}) \rightarrow \infty$ and
\be\label{1.57}
u_{i}(y) \leq \overline{C}|y-y_{i}|^{-2 \sigma /(p_{i}-1)} \quad \text { for all }\, y \in B_{\overline{r}}(y_{i}).
\ee
\end{definition}
Let $y_{i} \rightarrow \overline{y}$ be an isolated blow up point of $\{u_{i}\}$, and define,
for $r>0$,
$$
\overline{u}_{i}(r):=\frac{1}{|\partial B_{r}(y_{i})|} \int_{\partial B_{r}(y_{i})} u_{i}
\quad \text{and}\quad
\overline{w}_{i}(r):=r^{2 \sigma /(p_{i}-1)} \overline{u}_{i}(r).
$$

\begin{definition}\label{defn1.3}
A point  $y_{i} \rightarrow \overline{y} \in \Omega$ is called an isolated simple blow up point if $y_{i} \rightarrow \overline{y}$ is an isolated blow up point such that for some $\rho>0$ (independent of i), $\overline{w}_{i}$ has precisely one critical point in $(0, \rho)$ for large $i$.
\end{definition}

\subsection{ Characterization of blow up behavior }
Recall the definitions of the matrix $M$ given in \eqref{M} and its smallest eigenvalue $\mu(M)$.
The result about characterization of blow up behavior
of the solutions to \eqref{1.1} is:
\begin{theorem}\label{thm1}
Let $\sigma=1+m/2,$ $m\in \mathbb{N}_{+}$ and $n=2\sigma+2.$
Let $K \in C^{2}(\mathbb{S}^{n})$ be a positive function and $\mathscr{K},
\mathscr{K}^{-},\mathscr{K}^{+}$ be as in \eqref{1.2}.
Let $p_{i}$ satisfy $p_{i} \leq \frac{n+2\sigma}{n-2\sigma}=\frac{n+2\sigma}{2}=n-1,$ $ p_{i} \rightarrow
n-1,$
$K_{i} \in C^{2}(\mathbb{S}^{n})$ satisfy $K_{i} \rightarrow K$
in $C^{2}(\mathbb{S}^{n}),$ and $v_{i}\in C^{2\sigma}(\mathbb{S}^n)$ satisfy
\be\label{1.7}
P_{\sigma}v_{i}=c(n,\sigma)K_iv_{i}^{p_i}
\ee
and
$$
\lim_{i\rightarrow\infty}\max_{\mathbb{S}^n}v_{i}=\infty.
$$
Then there exists a constant $\delta^{*}>0$
depending only on $ \min_{\mathbb{S}^{n}} K,$ $\|K\|_{C^{2}(\mathbb{S}^{n})},$ and
the modulus of the continuity of $\nabla_{g_{0}} K$ if $\sigma>1/2$
such that after passing to a subsequence, we have:
\begin{enumerate}
  \item[(i)]  $\{v_i\}$ (still denote the subsequence by $\{v_i\}$)  has only
  isolated simple blow up points $q^{(1)}, \cdots, q^{(k)} \in \mathscr{K} \backslash \mathscr{K}^{+}$
   $(k \geq 1)$ with $|q^{(j)}-q^{(\ell)}| \geq \delta^{*},$  $\forall\, j \neq \ell,$
  and $\mu(M(q^{(1)}, \cdots, q^{(k)})) \geq 0.$
  Furthermore, $q^{(1)}, \cdots, q^{(k)} \in \mathscr{K}^{-}$ if $k \geq 2.$
  \item[(ii)]
    Let $q^{(1)},\cdots,q^{(k)}$ be as in (i), and
     $q_{i}^{(j)}$ be the local maximum of $v_i$
     with $q_{i}^{(j)}\rightarrow q^{(j)},$ we have
  \begin{align}
  \lambda_{j}&:=K(q^{(j)})^{-1/2\sigma}
  \lim_{i\rightarrow\infty}v_{i}(q_{i}^{(1)})(v_{i}(q_{i}^{(j)}))^{-1}\in (0,\infty),
  \label{1.3}\\
  \mu^{(j)}&:=\lim_{i\rightarrow \infty} \tau_{i} v_{i}(q_{i}^{(j)})^{2} \in[0, \infty).\label{1.4}
 \end{align}
  \item[(iii)]
  Let $\lambda_{j},\mu^{(j)}, j=1,\cdots,k$ be as in (ii), then when $k=1,$
    \be\label{1.5}
    \mu^{(1)}=-\frac{2}{\sigma}\frac{\Delta_{g_0} K(q^{(1)})}{K(q^{(1)})^{n/2\sigma}},
     \ee
when $k\geq 2$,
\be\label{1.6}
\sum_{\ell=1}^{k} M_{\ell j}(q^{(1)}, \cdots, q^{(k)})
\lambda_{\ell}= \frac{\sigma}{2}\lambda_{j} \mu^{(j)}, \quad \forall\, j: 1 \leq j \leq k.
\ee

\item [(iv)]
$\mu^{(j)}\in(0,\infty),$ $\forall\,j=1, \cdots, k,$ if and only if $\mu(M(q^{(1)},\cdots,q^{(k)})) >0.$
\end{enumerate}
\end{theorem}

We first give the following proposition:
\begin{proposition}\label{prop6}
Let $K \in C^{2}(\mathbb{S}^{n})$, $n\geq 2$, be a positive function and $\mathscr{K},
\mathscr{K}^{-},\mathscr{K}^{+}$ be as in \eqref{1.2}.
Let $p_{i}$ satisfy $p_{i} \leq \frac{n+2\sigma}{n-2\sigma}$ $ p_{i} \rightarrow
\frac{n+2\sigma}{n-2\sigma},$
$K_{i} \in C^{2}(\mathbb{S}^{n})$ satisfy $K_{i} \rightarrow K$
in $C^{2}(\mathbb{S}^{n}),$ and $v_{i}$ satisfy
$$
P_{\sigma}v_{i}=c(n,\sigma)K_iv_{i}^{p_i}.
$$
Then exists a constant $\delta^{*}>0$ depending only on
$ \min_{\mathbb{S}^{n}} K,$ $\|K\|_{C^{2}(\mathbb{S}^{n})},$
and the modulus of the continuity of $\nabla_{g_{0}} K$ if $\sigma>1/2$
 such that, after passing to a subsequence,
either $\{v_{i}\}$ stays bounded in $L^{\infty}(\mathbb{S}^{n})$
or $\{v_{i}\}$ has only isolated simple blow up points and the distance between any two blow up points is bounded blow by $\delta^{*}$.
\end{proposition}
\begin{proof}
The proof follows from the same arguments used to prove Theorem 3.3 in
\cite{jlxm}, so we omit it.
\end{proof}

\begin{proof}[Proof of Theorem \ref{thm1}]
From  Proposition \ref{prop6} and
$\lim_{i\rightarrow\infty}\max_{\mathbb{S}^n}v_{i}=\infty,$
there exists a constant $\delta^{*}>0$ depending only on
$ \min_{\mathbb{S}^{n}} K,$ $\|K\|_{C^{2}(\mathbb{S}^{n})},$
and the modulus of the continuity of $\nabla_{g_{0}} K$ if $\sigma>1/2$
such that $\{v_{i}\}$ has only isolated simple blow up points $q^{(1)}, \cdots, q^{(k)} \in \mathscr{K}$ $(k \geq 1)$ with $|q^{(j)}-q^{(\ell)}| \geq \delta^{*}$
$(j \neq \ell).$

By \eqref{2.27}, \eqref{1.7} is equivalent to
\begin{align}\label{1.101}
v_{i}(\xi)=\frac{\Gamma(\frac{n+2 \sigma}{2})}{2^{2 \sigma} \pi^{n / 2} \Gamma(\sigma)}
\int_{\mathbb{S}^{n}} \frac{K_{i}(\eta) v_{i}(\eta)^{p_{i}}}{|\xi-\eta|^{2}}\,
 \ud \eta \quad \text { on }\, \mathbb{S}^{n}.
\end{align}
Let $F$ be the stereographic projection with
with $q^{(j)}$ being the south pole:
$$
\begin{aligned}
F:\mathbb{R}^n&\rightarrow \mathbb{S}^n \backslash \{-q^{(j)}\},\\
 x&\mapsto\Big(\frac{2x}{1+|x|^2},\frac{|x|^2-1}{|x|^2+1}\Big).
\end{aligned}
$$
Let $\tau_{i}= n-1-p_{i},$
via the stereographic projection, the equation \eqref{1.101}
is translated to
$$
u_{i}(x)=\frac{\Gamma(\frac{n+2 \sigma}{2})}{2^{2 \sigma} \pi^{n / 2} \Gamma(\sigma)}
\int_{\mathbb{R}^n}\frac{\wdt{K}_{i}(y)H(y)^{\tau_{i}}u_{i}(y)^{p_{i}}}{|x-y|^{2}}\,\ud y
\quad
\text{ on }\, \mathbb{R}^{n},
$$
where
\be\label{1.9}
H(x)=\frac{2}{1+|x|^2},\quad
u_i(x)=H(x)v_i(F(x)), \quad \widetilde{K}_{i}(x)=K_{i}(F(x)).
\ee

Let $x_{i}^{(j)}$ be the local maximum of $u_{i}$ and $x_{i}^{(j)}\rightarrow 0.$
It follows from Propositions \ref{prop4} and \ref{prop7} that
\begin{align}
\begin{aligned}\label{1.10}
u_{i}(x_{i}^{(j)})u_{i}(x)\rightarrow h^{(j)}(x) :=a
K(q^{(j)})^{-1/\sigma}|x|^{-2}&+b^{(j)}(x)\\
\quad& \text { in }\, C_{{loc}}^{2}(\mathbb{R}^{n}
\backslash\{\cup_{\ell=1}^{k}x^{(\ell)}\}),
\end{aligned}
\end{align}
where
\be\label{2.16}
a=4c_{n,\sigma}c(n,\sigma)\int_{\mathbb{R}^n}\Big(\frac{1}{1+|y|^{2}}\Big)^{n-1}\,\ud y
=2c_{n,\sigma}c(n,\sigma)|\mathbb{S}^{n-1}|\mathrm{B}(\sigma,n/2),
\ee
{$\mathrm{B}(\sigma,n/2)$} is the Beta function, and $c_{n,\sigma}$ is as in \eqref{2.27}.
From the maximum principle, $b^{(j)}(x)$ satisfies
\be\label{2.17}
 b^{(j)}(x)\equiv 0\quad \text{ if } \, k=1,\quad
 b^{(j)}(x)>0\quad\text{ if }\, k\geq 2.
 \ee

By \eqref{1.9} and $y_{i}^{(j)}\rightarrow 0$ as $i\rightarrow \infty,$ we have
$$
\lim_{i\rightarrow \infty}  v_{i}(q_{i}^{(j)}) v_{i}(q)
=\frac{1}{4}\lim_{i\rightarrow \infty}(1+|x|^2)u_{i}(x_{i}^{(j)}) u_{i}(x),
$$
combining with \eqref{1.10}, it easy to see that for $q \neq q^{(j)}$ and close to $q^{(j)},$
\be\label{1.11}
\lim _{i \rightarrow \infty} v_{i}(q_{i}^{(j)}) v_{i}(q)
=\frac{aG_{q^{(j)}}(q)}{2K(q^{(j)})^{{1}/{\sigma}}}
+\wdt{b}^{(j)}(q)\quad \text{ in } \,C^{2}_{{loc}}
(\mathbb{S}^n\backslash \{\cup_{\ell=1}^{k}q^{(\ell)}\}),
\ee
 where $a$ is as in \eqref{2.16}, and
 $\wdt{b}^{(j)}(q)$ is some regular function  on $\mathbb{S}^{n} \backslash$
 $\cup_{\ell \neq j}\{q^{(\ell)}\}$
satisfying $P_{\sigma}\widetilde{b}^{(j)}=0,$
and  $G_{q^{(j)}}(q)$ is the Green function defined as in \eqref{1.8}.

When $k\geq 2,$  taking into account the contribution of all
the poles, we deduce
\be\label{1.121}
\begin{aligned}
\lim _{i \rightarrow \infty} v_{i}(q_{i}^{(j)}) v_{i}(q)
=
\frac{aG_{q^{(j)}}(q)}{2K(q^{(j)})^{{1}/{\sigma}}}
+\frac{a}{2}\sum_{\ell \neq j}
&\lim _{i \rightarrow \infty}
\frac{v_{i}(q_{i}^{(j)})}
{v_{i}(q_{i}^{(\ell)})}
\frac{G_{q^{(\ell)}}(q)}{K(q^{(\ell)})^{{1}/{\sigma}}}\\
&\quad \text{ in }\, C^{2}_{{loc}}
(\mathbb{S}^n\backslash \{\cup_{\ell=1}^{k}q^{(\ell)}\}).
\end{aligned}
\ee
In fact, subtracting all the poles from the limit function, we obtain
a regular function  $\wdt{b}_{0}: \mathbb{S}^n\rightarrow \mathbb{R}$
such that $P_{\sigma}\wdt{b}_{0}=0$ on $\mathbb{S}^n,$
so it must be $\wdt{b}_{0}\equiv 0.$
Using \eqref{1.121}, we have,  for $|y|> 0$ small,
\be\label{2.5}
h^{(j)}(y)=
\frac{a}{K(q^{(j)})^{1/\sigma}|y|^{2}}+2a
\sum_{\ell \neq j} \lim _{i \rightarrow \infty}
 \frac{v_{i}(q_{i}^{(j)})}{v_{i}(q_{i}^{(\ell)})}
\frac{G_{q^{(\ell)}}(q^{(j)})}{K(q^{(\ell)})^{{1}/{\sigma}}}+O(|y|),
\ee
where $a$ is as in \eqref{2.16}.
The conclusion obtained from the above is easy to see that \eqref{1.3} is true.

Before stating the result to be proved, we give the following estimates \eqref{2.3} and \eqref{2.4}.
Using Proposition \ref{prop8}, we obtain
\be\label{2.3}
|\nabla K_{i}(y_{i}^{(j)})|
=O(u_{i}(y_{i}^{(j)})^{-1}), \quad \tau_{i}
=O(u_{i}(y_{i}^{(j)})^{-2}).
\ee
It is obvious that \eqref{1.4} can be proved by  \eqref{2.3}.
 We have proved Part (ii).

Let $y=(y_{(1)},\cdots,y_{(n)})\in \mathbb{R}^{n}.$
It follows from Propositions \ref{prop3.1}, \ref{prop4}, and \ref{prop5}, that
for sufficiently small $\delta>0,$
\be\label{2.4}
\begin{aligned}
&\sum\limits_{j=1}^{n}\Big|\int_{B_{\delta}} y_{(j)} u_{i}(y+x_{i}^{(j)})^{p_{i}+1}\Big|
=o(u_{i}(x_{i}^{(j)})^{-1}), \\
&\sum\limits_{j \neq \ell}\Big|\int_{B_{\delta}} y_{(j)} y_{(\ell)} u_{i}(y+x_{i}^{(j)})^{p_{i}+1}\Big|
=o(u_{i}(x_{i}^{(j)})^{-2}), \\
&\int_{\partial B_{\delta}} u_{i}(y+x_{i}^{(j)})^{p_{i}+1}
=O(u_{i}(x_{i}^{(j)})^{-p_{i}-1}), \\
&\lim \limits_{i \rightarrow \infty} u_{i}(x_{i}^{(j)})^{2}
\int_{B_{\delta}}|y|^{2} u_{i}(y+x_{i}^{(j)})^{p_{i}+1}
=\frac{n2^{1+n}|\mathbb{S}^{n-1}|}{n+2\sigma}
\frac{\ub(\sigma,n/2)}{K(q^{(j)})^{1+{2}/{\sigma}}}.
\end{aligned}
\ee
In fact, the first three formulas in \eqref{2.4} can be easily obtained
 from Proposition \ref{prop5}. For the last formula in \eqref{2.4},
 let $R_{i}$ be as in Proposition \ref{prop3.1} and
\be\label{2.11}
m_{ij}:=u_{i}(x_{i}^{(j)}),\quad
r_{ij}:=R_{i}m_{ij}^{-(p_i-1)/2\sigma},
\quad k_{ij}:=2^{-2}\wdt{K}_{i}(x_{i}^{(j)})^{1/\sigma}.
\ee
Using  Proposition \ref{prop3.1} again, we have
\begin{align*}
 &m_{ij}^{2}\int_{|y|\leq r_{ij}}|y|^{2} u_{i}(y+x_{i}^{(j)})^{p_{i}+1}\,\ud y\\
 =&m_{ij}^{2}\int_{|x|\leq R_{i}}m_{ij}^{\frac{-(2+n)(p_i-1)}{2\sigma}+p_{i}+1}|x|^{2}
 (m_{ij}^{-1}u_{i}(m_{ij}^{-(p_i-1)/2\sigma}x+x_{i}^{(j)}))^{p_{i}+1}\,\ud x\\
 =&m_{ij}^{\frac{-(2+n)(p_i-1)}{2\sigma}+p_{i}+3}
 \int_{|x|\leq R_{i}}|x|^2\Big(\frac{1}{1+k_{ij}|x|^2}\Big)^{p_{i}+1}\,
 \ud x+o(1)\\
 =&\int_{\mathbb{R}^{n}}\frac{|x|^2}{(1+k_{ij}|x|^2)^{n}}\,\ud x+o(1)\\
 =&\frac{n2^{1+n}|\mathbb{S}^{n-1}|}{n+2\sigma}K(q^{(j)})^{-1-{2}/{\sigma}}
 \ub(\sigma,n/2)+o(1).
\end{align*}
We have completed the proof of \eqref{2.4}.

By $n-2\sigma=2$ and $\tau_{i}=(n+2\sigma)/(n-2\sigma)-p_{i},$ it is easy to see that
\be\label{2.6}
\frac{1}{p_{i}+1}=\frac{1}{2\sigma+2-\tau_{i}}
=\frac{1}{n}\Big(1+\frac{\tau_{i}}{n}+O(\tau_{i}^{2})\Big).
\ee
For sufficiently small $\delta>0,$  $u_{i}$ satisfy
$$
u_{i}(x)=\frac{c_{n,\sigma}c(n,\sigma)}{2^{2 \sigma}}
\int_{B_{\delta}(x_{i}^{(j)})}
\frac{\wdt{K}_{i}(y)H(y)^{\tau_{i}}u_{i}(y)^{p_{i}}}{|x-y|^{2}}\,\ud y
+h_{\delta}(x),
$$
where
\be\label{2.7}
h_{\delta}(x)=\frac{c_{n,\sigma}c(n,\sigma)}{2^{2 \sigma}}
\int_{\mathbb{R}^{n}\backslash B_{\delta}(x_{i}^{(j)})}
\frac{\wdt{K}_{i}(y)H(y)^{\tau_{i}}u_{i}(y)^{p_{i}}}{|x-y|^{2}}\,\ud y.
\ee
By Proposition \ref{prop1.1}, we have
\begin{align}
\begin{aligned}\label{1.102}
&\Big(\frac{n-2 \sigma}{2}-\frac{n}{p_{i}+1}\Big)
\int_{B_{\delta}(x_{i}^{(j)})} \wdt{K}_{i}(x)
H(x)^{\tau_{i}} u_{i}(x)^{p_{i}+1} \,\ud x\\
&\quad-\frac{1}{p_{i}+1} \int_{B_{\delta}(x_{i}^{(j)})}(x-x_{i}^{(j)}) \cdot
\nabla (\wdt{K}_{i}(x)H(x)^{\tau_{i}}) u(x)^{p_{i}+1} \,\ud x \\
=& \frac{n-2 \sigma}{2} \int_{B_{\delta}(x_{i}^{(j)})}
\wdt{K}_{i}(x)
H(x)^{\tau_{i}} u_{i}(x)^{p_{i}} h_{\delta}(x) \,\ud x\\
&\quad
+\int_{B_{\delta}(x_{i}^{(j)})} (x-x_{i}^{(j)})\cdot \nabla h_{\delta}(x) \wdt{K}_{i}(x)H(x)^{\tau_{i}}
u_{i}(x)^{p_{i}} \,\ud x \\
&\quad-\frac{\delta}{p_{i}+1} \int_{\partial B_{\delta}(x_{i}^{(j)})}
\wdt{K}_{i}(x) H(x)^{\tau_{i}} u_{i}(x)^{p_{i}+1} \, \ud s.
\end{aligned}
\end{align}

Let  $x=(x_{(1)},\cdots,x_{(n)})\in \mathbb{R}^n,$  by \eqref{2.6} we have
\begin{align}\label{2.8}
&-\frac{1}{p_{i}+1}\int_{B_{\delta}}
x\cdot\nabla
(\wdt{K}_{i}(x+x_{i}^{(j)})H(x+x_{i}^{(j)})^{\tau_{i}})
u_{i}(x+x_{i}^{(j)})^{p_{i}+1}\,\ud x\notag\\
=&-\frac{1}{n}
\sum_{\ell=1}^{n}\int_{B_{\delta}}x_{(\ell)}
\frac{\partial \wdt{K}_{i}}{\partial x_{(\ell)}}(x+x_{i}^{(j)})
u_{i}(x+x_{i}^{(j)})^{p_{i}+1}\,\ud x
+o(m_{ij}^{-2})\notag\\
=&-\frac{1}{n}\int_{B_{\delta}}x\cdot\nabla\wdt{K}(x_{i}^{(j)})u_{i}(x+x_{i}^{(j)})^{p_{i}+1}
\,\ud x\notag\\
& -\frac{1}{n}
\sum_{\ell,m}\int_{B_{\delta}}x_{(\ell)}x_{(m)}\frac{\partial^2 \wdt{K}}{\partial x_{(\ell)}\partial x_{(m)}}(x_{i}^{(j)})
u_{i}(x+x_{i}^{(j)})^{p_{i}+1}\,\ud x
+o(m_{ij}^{-2})\notag\\
=&-\frac{1}{n^{2}}\Delta \wdt{K}(0) \int_{B_{\delta}}|x|^{2} u_{i}(x+x_{i}^{(j)})^{p_{i}+1} \,\ud x
+o(m_{ij}^{-2})\notag\\
=&-\frac{4}{n^2}\Delta_{g_{0}}K(q^{(j)})
\int_{B_{\delta}}
|x|^{2} u_{i}(x+x_{i}^{(j)})^{p_{i}+1} \,\ud x
+o(m_{ij}^{-2}).
\end{align}
Then, by \eqref{2.4} and \eqref{2.8},
\be\label{2.13}
\begin{aligned}
&\lim_{i\rightarrow\infty}
-\frac{m_{ij}^{2}}{p_{i}+1}\int_{B_{\delta}}
x\cdot\nabla \wdt{K}_{i}(x+x_{i}^{(j)})u_{i}(x+x_{i}^{(j)})^{p_{i}+1}\,\ud x\\
=&-\frac{2^{3+n}|\mathbb{S}^{n-1}|\ub(\sigma,n/2)}{n(n+2\sigma)}
\frac{\Delta_{g_{0}} K(q^{(j)})}{K(q^{(j)})^{1+{2}/{\sigma}}}.
\end{aligned}
\ee

Let $r_{i}$ be as in \eqref{2.11} and by \eqref{2.6}, we have
\begin{align}\label{2.9}
&\Big(\frac{n-2\sigma}{2}-\frac{n}{p_{i}+1}\Big)
\int_{B_{\delta}}\wdt{K}_{i}(x+x_{i}^{(j)})H(x+x_{i}^{(j)})^{\tau_i}
u_{i}(x+x_{i}^{(j)})^{p_i+1}\,\ud x\notag\\
=& -\frac{\tau_{i}}{n}\int_{B_{\delta}}
\wdt{K}_{i}(x+x_{i}^{(j)})u_{i}(x+x_{i}^{(j)})^{p_i+1}\,\ud x
+o(m_{ij}^{-2})\notag\\
=&-\frac{\tau_{i}}{n}\int_{B_{\delta}}\wdt{K}_{i}(x_{i}^{(j)})
u_{i}(x+x_{i}^{(j)})^{p_i+1}\,\ud x\notag\\
\quad&+O\Big(\Big|\int_{B_{\delta}}x\cdot \nabla \wdt{K}_{i}(x_{i}^{(j)})
u_{i}(x+x_{i}^{(j)})^{p_i+1}\,\ud x\Big|\Big)\notag\\
\quad&+O\Big(  \int_{B_{\delta}} |x|^{2}
u_{i}(x+x_{i}^{(j)})^{p_i+1}\,\ud x\Big)
+o(m_{ij}^{-2})\notag\\
=&-\frac{\tau_{i}}{n}\wdt{K}(x_{i}^{(j)})
\int_{|x|<r_{ij}} u_{i}(x+x_{i}^{(j)})^{p_i+1}\,\ud x
+o(m_{ij}^{-2})\notag\\
=&-\frac{\tau_{i}2^{n}}{n}K(q^{(j)})^{-1/\sigma}
\int_{\mathbb{R}^{n}}\frac{1}{(1+|x|^2)^n}\,\ud x
+o(m_{ij}^{-2})\notag\\
=&-\frac{\tau_{i}2^{n}|\mathbb{S}^{n-1}|}{n}\Big(\frac{\sigma}{n+2\sigma}\Big)
\mathrm{B}(n/2,\sigma)K(q^{(j)})^{-{1}/{\sigma}}
+o(m_{ij}^{-2}).
\end{align}
It follows from \eqref{1.4} and \eqref{2.9} that
\begin{align}\label{2.14}
&\lim_{i\rightarrow \infty}-m_{ij}^{2}\Big(1-\frac{n}{p_{i}+1}\Big)
\int_{B_{\delta}}\wdt{K}_{i}(x+x_{i}^{(j)})H(x+x_{i}^{(j)})^{\tau_i}
u_{i}(x+x_{i}^{(j)})^{p_i+1}\,\ud x\notag\\
\quad=&-\frac{2^{n+2}|\mathbb{S}^{n-1}|\sigma\ub(\sigma,n/2)}{n(n+2\sigma)}
\frac{\mu^{(j)}}{K(q^{(j)})^{{1}/{\sigma}}}.
\end{align}
In view of  \eqref{2.4}, we obtain
\begin{align}\label{2.12}
&\lim_{i\rightarrow \infty}-m_{ij}^2\frac{\delta}{p_{i}+1}\int_{\partial B_{\delta}}
\wdt{K}_{i}(x+x_{i}^{(j)})H(x+x_{i}^{(j)})^{\tau_{i}}
u_{i}(x+x_{i}^{(j)})^{p_{i}+1}\,\ud x\notag\\
=&\lim_{i\rightarrow \infty}-m_{ij}^2
\frac{\delta}{n}\Big(1+\frac{\tau_{i}}{n}\Big)
\int_{\partial B_{\delta}}\wdt{K}_{i}(x+x_{i}^{(j)})u_{i}(x+x_{i}^{(j)})^{p_{i}+1}\,\ud x\notag\\
=&0.
\end{align}
Using \eqref{2.7} and  Proposition \ref{prop3.1}, we have
\begin{align*}
&m_{ij}^{2}\frac{n-2\sigma}{2}
\int_{B_{\delta}(x_{i}^{(j)})}
\wdt{K}_{i}(x)
u_{i}(x)^{p_{i}}H(x)^{\tau_{i}}
h_{\delta}(x)\,\ud x\\
=&m_{ij}^{2}\int_{B_\delta(x_{i}^{(j)})}
(\wdt{K}_{i}(x_{i}^{(j)})+(x-x_{i}^{(j)})\cdot \nabla\wdt{K}(x_{i}^{(j)})+O(|x-x_{i}^{(j)}|^2))
u_{i}(x)^{p_{i}}h_{\delta}(x)\,\ud x\\
=&m_{ij}^{2-\frac{n(p_{i}-1)}{2\sigma}+p_{i}}
\wdt{K}_{i}(x_{i}^{(j)})\int_{|y|<R_{i}} (m_{ij}^{-1}u_{i}(m_{ij}^{-\frac{p_i-1}{2\sigma}}y+x_{i}^{(j)}))^{p_{i}}
h_{\delta}(m_{ij}^{-\frac{p_{i}-1}{2\sigma}}y+x_{i}^{(j)})\,\ud y+o(1)\\
=&\wdt{K}_{i}(x_{i}^{(j)})\int_{\mathbb{R}^{n}}
\frac{1}{(1+k_{ij}|y|^{2})^{p_{i}}}
h_{\delta}(m_{ij}^{-\frac{p_{i}-1}{2\sigma}}y+x_{i}^{(j)})\,\ud y+o(1)\\
=&2^{n-1}K(q^{(j)})^{-1/\sigma}|\mathbb{S}^{n-1}|\ub(\sigma,n/2)
b^{(j)}(0)+o(1),
\end{align*}
it follows that
\begin{align}\label{2.15}
&\lim_{i\rightarrow\infty}
m_{ij}^{2}\frac{n-2\sigma}{2}
\int_{B_{\delta}(x_{i}^{(j)})}
\wdt{K}_{i}(x)
u_{i}(x)^{p_{i}}H(x)^{\tau_{i}}
h_{\delta}(x)\,\ud x\notag\\
=&2^{n-1}|\mathbb{S}^{n-1}|\ub(\sigma,n/2)\frac{1}{K(q^{(j)})^{{1}/{\sigma}}}
b^{(j)}(0).
\end{align}

When $|x-x_{i}^{(j)}|<\delta,$ a direct calculation gives
\begin{align}\label{2.26}
|\nabla h_{\delta}(x)|\leq
\begin{cases}
\displaystyle C\frac{|\delta^{2\sigma-1}-(\delta-|x-x_{i}^{(j)}|)^{2\sigma-1}|}
 {2\sigma-1}m_{ij}^{-1} & \text{ if }\, \sigma\ne{1}/{2}, \\
\displaystyle C|\log \delta -\log(\delta-|x-x_{i}^{(j)}|)|m_{ij}^{-1} & \text{ if } \, \sigma={1}/{2}.
\end{cases}
\end{align}
The detailed proof of \eqref{2.26} can refer to \cite{jlxm}.
 Using Proposition \ref{prop3.1} and \eqref{2.26}, we can obtain
\begin{align}\label{2.21}
&\Big|
\int_{B_{\delta}(x_{i}^{(j)})}(x-x_{i}^{(j)})\nabla h_{\delta}(x)
\wdt{K}_{i}(x)H(x)^{\tau_{i}}u_{i}(x)^{p_{i}}\,\ud x \Big|\notag\\
\leq &
Cm_{ij}^{-1}\int_{|x-x_{i}^{(j)}|<\delta}|x-x_{i}^{(j)}|u_{i}(x)^{p_{i}}\,\ud x
\notag\\
\leq& Cm_{ij}^{-1-\frac{(n+1)(p_{i}-1)}{2\sigma}+p_{i}}
\int_{|y|<R_{i}}|y|
(m_{ij}^{-1}u_{i}(m_{ij}^{-\frac{(p_{i}-1)}{2\sigma}}y+x_{i}^{(j)})
)^{p_{i}}\,\ud y\notag\\
=&o(m_{ij}^{-2}).
\end{align}

By \eqref{1.102}, \eqref{2.13}, \eqref{2.14}, \eqref{2.12}, \eqref{2.15},
and \eqref{2.21}, we have
\be\label{2.18}
\frac{8\sigma\mu^{(j)}}{n(n+2\sigma)}
\frac{1}{K(q^{(j)})^{{1}/{\sigma}}}
+\frac{16}{n(n+2\sigma)}
\frac{\Delta_{g_0} K(q^{(j)})}{K(q^{(j)})^{1+{2}/{\sigma}}}
=-\frac{1}{K(q^{(j)})^{{1}/{\sigma}}}
b^{(j)}(0).
\ee
Consequently, $q^{(j)} \in \mathscr{K} \backslash \mathscr{K}^{+}$, $1 \leq j \leq k,$
 and when $k \geq 2$, $q^{(j)} \in \mathscr{K}^{-},$  $ 1 \leq j \leq k.$
It is easy to see that \eqref{1.5} follows from \eqref{2.17} and \eqref{2.18} when $k=1$.

From \eqref{2.5}, \eqref{2.16}, \eqref{1.9}, and \eqref{1.3}, we can obtain
\be\label{2.22}
\begin{aligned}
b^{(j)}(0)=&\frac{2^2\Gamma({n}/{2})|\mathbb{S}^{n-1}|}{\pi^{n/2}}
\sum_{\ell \neq j} \frac{\lambda_{\ell}}{\lambda_{j}}
\frac{G_{q^{(\ell)}}(q^{(j)})}{(K(q^{(j)})K(q^{(\ell)}))^{{1}/{2\sigma}}}\\
=&8\sum_{\ell \neq j} \frac{\lambda_{\ell}}{\lambda_{j}}
\frac{G_{q^{(\ell)}}(q^{(j)})}{(K(q^{(j)})K(q^{(\ell)}))^{{1}/{2\sigma}}}.
\end{aligned}
\ee
Substituting \eqref{2.22} into \eqref{2.18} to get
$$
-n(n-1)\sum_{\ell \neq j}
\frac{G_{q^{(\ell)}}(q^{(j)})}{(K(q^{(j)})K(q^{(\ell)}))^{{1}/{2\sigma}}} \lambda_{\ell}
-\frac{\Delta_{g_0} K(q^{(j)})}{K(q^{(j)})^{{n}/{2\sigma}}}\lambda_{j}
=\frac{\sigma}{2}\lambda_{j}\mu^{(j)}.
$$
We have established \eqref{1.6} and thus verified Part (iii).

We claim that there exists some
\be\label{eig}
\eta=(\eta_{1},\cdots,\eta_{k})\neq 0\quad
\hbox{with}\quad\eta_{\ell}\geq 0, \, \forall\, \ell=1,\cdots,k,
 \ee
such that
$$
\sum_{\ell=1}^{k} M_{\ell j}(q^{(1)},
 \cdots, q^{(k)}) \eta_{\ell}=\mu(M) \eta_{j}, \quad \forall\, j=1,\cdots,k.
$$
Indeed, choose $\Lambda>\max_{i}M_{ii},$ then the matrix $\Lambda I-M$ is a positive
matrix (see \cite{HJ} for the definition), where $I$ denotes the unit matrix.
The claim can follows from \cite[Theorem 8.2.2]{HJ}.

Multiplying \eqref{1.6} by $\eta_{j}$
and summing over $j,$  then using Part (ii) and  \eqref{eig}, we have
\be\label{1.84}
\mu(M) \sum_{j} \lambda_{j} \eta_{j}
=\sum_{\ell, j} M_{\ell j} \lambda_{\ell} \eta_{j}
=\frac{1}{4} \sum_{j} \lambda_{j} \eta_{j} \mu^{(j)}\geq 0.
\ee
It follows that $\mu(M)\geq 0$. We have verified part (i) of Theorem \ref{thm1}.
 Part (iv) follows from (i)--(iii). The proof  is completed.
\end{proof}

\subsection{Proof of Theorem \ref{thm2.1}}
Using some results of blow up analysis in Appendix \ref{sec2} and
Theorem \ref{thm1},  we are going to prove Theorem \ref{thm2.1}.

\begin{proof}[Proof of Theorem \ref{thm2.1}]
We first prove the existence of upper bounds.
Suppose the assertion of the theorem  is false. Then we can find
 that there exists $K_{i} \rightarrow K$ in $C^{2}(\mathbb{S}^{n})$ such that
$\max _{\mathbb{S}^{n}} v_{i} \rightarrow \infty$ for some $v_{i} \in \mathscr{M}_{K_{i}}.$
Theorem \ref{thm1} shows that $\{v_{i}\}$ has only isolated simple
blow up points $\{q^{(1)}, \cdots, q^{(k)}\}\subset \mathscr{K}
\backslash\mathscr{K}^{+}.$

Next, we prove that $k>1.$
Let $q_{0}$ be the isolated simple blow up point of $v_{i}.$
It follows from Proposition \ref{prop8}
and $K\in \mathscr{A}$ that  $q_{0}$ is a
non-degenerate critical point
of $K$.
Let $F$ be the stereographic projection
with $q_{0}$ being the south pole,
and $\wdt{K}:=K(F(y)).$

We assert that for any $\widehat{y}\in \mathbb{R}^{n},$
\begin{align}\label{3.13}
\left(
\begin{array}{c}
\int_{\mathbb{R}^{n}}\nabla^{2}\wdt{K}(0)(y+\widehat{y}) (1+|y|^{2})^{-n}\\\\
\int_{\mathbb{R}^{n}}
 \frac{1}{2}\langle (y+\widehat{y}), \nabla^{2}\wdt{K}(0)(y+\widehat{y})\rangle
  (1+|y|^{2})^{-n}
\end{array}
\right)
\ne 0.
\end{align}
In fact,  if there exists some $\widehat{y}\in \mathbb{R}^{n}$ such that
$$
\int_{\mathbb{R}^{n}}\nabla^{2}\wdt{K}(0)(y+\widehat{y}) (1+|y|^{2})^{-n}=0,
$$
then by  the property of odd function, the non degeneracy of $\nabla^{2}\wdt{K}(0)$,
 and $\Delta\wdt{K}(0)\ne 0,$ we can obtain that
\begin{align*}
 \int_{\mathbb{R}^{n}} \frac{1}{2}\langle (y+\widehat{y}),
 \nabla^{2}\wdt{K}(0)(y+\widehat{y})\rangle
  (1+|y|^{2})^{-n}\ne 0.
\end{align*}
Thus \eqref{3.13} is proved.

Suppose the contrary that $q_{0}$ is the only blow up of $v_i.$
We are going to find some $\widehat{y}$ such that \eqref{3.13} fails.
By \eqref{2.27}, we know that \eqref{1.1} is equivalent to
\begin{align}\label{3.1}
v_{i}(\xi)=\frac{\Gamma(\frac{n+2 \sigma}{2})}{2^{2 \sigma} \pi^{n / 2} \Gamma(\sigma)}
\int_{\mathbb{S}^{n}} \frac{K_{i}(\eta) v_{i}(\eta)^{n-1}}{|\xi-\eta|^{2}}\,
 \ud \eta \quad \text { on }\, \mathbb{S}^{n}.
\end{align}
Under the  stereographic projection $F,$ the equation \eqref{3.1} is transformed to
$$
u_{i}(x)=\frac{\Gamma(\frac{n+2 \sigma}{2})}{2^{2 \sigma} \pi^{n / 2} \Gamma(\sigma)}
\int_{\mathbb{R}^n}\frac{\wdt{K}_{i}(y)u_{i}(y)^{n-1}}{|x-y|^{2}}\,\ud y
\quad
\text{ on }\, \mathbb{R}^{n},
$$
where
\be\label{1.9}
H(x)=\frac{2}{1+|x|^2},\quad
u_i(x)=H(x)v_i(F(x)), \quad \widetilde{K}_{i}(x)=K_{i}(F(x)).
\ee
Let $y_{i}$ be the local maximum point of $u_{i}(y)$
and $m_{i}=:u_{i}(y_{i}).$
First, we establish
\begin{align}\label{3.2}
|y_{i}|=O(m_{i}^{-1}) .
\end{align}
Since we have assumed that $v_{i}$ has no blow up
point other than $q_{0}$, it follows from Proposition
\ref{prop4} and the Harnack inequality that
$u_{i}(y)\leq C(\varepsilon)|y|^{-2}m_{i}^{-1}$
for $|y|\geq \varepsilon>0.$

By the Kazdan-Warner condition, we have
\begin{align}\label{3.3}
\int_{\mathbb{R}^{n}} \nabla \wdt{K}_{i} u_{i}^{n}=0.
\end{align}
It follows that for $\varepsilon>0$ small we have
\begin{align}\label{3.5}
\Big|\int_{B_{\varepsilon}} \nabla \wdt{K}_{i}
(y+y_{i}) u_{i}(y+y_{i})^{n}\Big|
\leq C(\varepsilon) m_{i}^{-n}.
\end{align}
For $|y|\leq \varepsilon,$
\begin{align}\label{3.6}
\wdt{K}_{i}(y)=\wdt{K}_{i}(0)+\frac{1}{2}
\langle y, \nabla^{2}\wdt{K}_{i}(0)y\rangle+o(|y|^{2}),
\end{align}
it follows that
\begin{align}\label{3.7}
\lim_{|y|\to 0}
\nabla \big(\wdt{K}_{i}(y)-
\langle y, \nabla^{2}\wdt{K}_{i}(0)y\rangle \big)|y|^{-1}=0,
\end{align}
where $\langle\cdot,\cdot\rangle$ denotes the inner product in $\mathbb{R}^{n}.$
Since $\det(\nabla^{2}\wdt{K}(0))\ne 0$ and $\wdt{K}_{i}\to \wdt{K},$
there exists a constant $C>0$ such that
\begin{align}\label{3.8}
\Big|\frac{1}{2}\nabla \langle y, \nabla^{2}\wdt{K}_{i}(0)y\rangle \Big|=
|\nabla^{2}\wdt{K}_{i}(0)y|\geq C|y|,\quad \forall\, |y|\leq \varepsilon.
\end{align}

By \eqref{3.5}, \eqref{3.7} and \eqref{3.8}, we can obtain
\begin{align*}
&\Big|\int_{B_{\varepsilon}}(1+o_{\varepsilon}(1))
  \nabla^{2}\wdt{K}_i(0)(y+y_{i})
u_{i}(y+y_{i})^{n}
\Big|\leq  C(\varepsilon)m_{i}^{-n}.
\end{align*}
Multiplying the above by $m_{i},$ and let $ \wdt{y}_{i}:=m_{i}y_{i},$ we have
\begin{align*}
&\Big|\int_{B_{\varepsilon}}(1+o_{\varepsilon}(1))
\nabla^{2}\wdt{K}_i(0)(m_{i}y+\wdt{y}_{i})
u_{i}(y+y_{i})^{n}
\Big|\leq  C(\varepsilon)m_{i}^{1-n}.
\end{align*}
Suppose \eqref{3.2} is false, namely
$\wdt{y}_{i}\rightarrow \infty$ along a subsequence.
From Proposition \ref{prop3.1},
we can choose $R_{i}\leq |\wdt{y}_{i}|/4$ such that
\begin{align*}
&\Big|\int_{|y|\leq R_{i}m_{i}^{-1}}(1+o_{\varepsilon}(1))
\nabla^{2}\wdt{K}_i(0) (m_{i}y+\wdt{y}_{i})
u_{i}(y+y_{i})^{n}
\Big|
\\
=&
\Big|\int_{|z|\leq R_{i}}(1+o_{\varepsilon}(1))
\nabla^{2}\wdt{K}_i(0)(z+\wdt{y}_{i})
(m_{i}^{-1}u_{i}(m_{i}^{-1}z+y_{i}))^{n}\Big| \sim |\wdt{y}_{i}|.
\end{align*}
On the other hand, it follows from Proposition \ref{prop5} that
\begin{align*}
&\Big|\int_{R_{i}m_{i}^{-1}\leq |y|\leq \varepsilon}(1+o_{\varepsilon}(1))
\nabla^{2}\wdt{K}_i(0) (m_{i}y+\wdt{y}_{i})
u_{i}(y+y_{i})^{n}
\Big|
\\
\leq &
C\Big|
\int_{R_{i}m_{i}^{-1}\leq |y|\leq \varepsilon}
(|m_{i}y|+|\wdt{y}_{i}|)u_{i}(y+y_{i})^{n}
\Big|\leq o(1)|\wdt{y}_{i}|.
\end{align*}
It follows that $|\wdt{y}_{i}|\leq C(\varepsilon) m_{i}^{1-n}.$
This contradicts to $\wdt{y}_{i}\to \infty.$
Thus \eqref{3.2} is proved.

It follows from the Kazdan-Warner condition that
$$
\int_{\mathbb{R}^{n}}\langle y, \nabla \wdt{K}_{i} (y+y_{i})\rangle
 u_{i}(y+y_{i})^{n}=0 .
$$
Similar to \eqref{3.5}, we have for any $\varepsilon>0,$
$$
\Big|
\int_{B_{\varepsilon}}\langle y, \nabla \wdt{K}_{i} (y+y_{i})
\rangle  u_{i}(y+y_{i})^{n}
\Big|
\leq C(\varepsilon)m_{i}^{-n}.
$$
By \eqref{3.6}, \eqref{3.7}, \eqref{3.8}, and Proposition \ref{prop5},  we have
\begin{align*}
&\Big|
\int_{B_{\varepsilon}}
\langle y,\nabla^{2}\wdt{K}_{i}(0)(y+y_{i}) \rangle
u_{i}(y+y_{i})^{n}
\Big|\\
\leq& C(\varepsilon) m_{i}^{-n}
+o_{\varepsilon}(1)\int_{B_{\varepsilon}}(|y|^{2}+|y||y_{i}|)u_{i}(y+y_{i})^{n}\\
\leq &C(\varepsilon)m_{i}^{-n}+o_{\varepsilon}(1)m_{i}^{-2}.
\end{align*}
Multiplying the above by $m_{i}^{2},$ due to $n-2=2\sigma,$ we have
\begin{align*}
\lim_{i\to\infty}m_{i}^{2}
\Big|
\int_{B_{\varepsilon}}
\langle y,  \nabla^{2}\wdt{K}_{i}(0)(y+y_{i})\rangle
u_{i}(y+y_{i})^{n}
\Big|=o_{\varepsilon}(1).
\end{align*}
Let $R_{i}\to \infty$ as $i\to \infty,$ and $r_{i}:=R_{i}m_{i}^{-1}$.
By Proposition \ref{prop5}, we have
\begin{align*}
&m_{i}^{2}
\Big|\int_{r_{i}\leq |y|\leq \varepsilon}
\langle y, \nabla^{2}\wdt{K}_{i}(y+y_{i})\rangle
u_{i}(y+y_{i})^{n}
\Big|\notag\\
\leq&C m_{i}^{2}
\Big|
\int_{r_{i}\leq |y|\leq \varepsilon}
(|y|^{2}+|y||y_{i}|)u_{i}(y+y_{i})^{n}
\Big|\rightarrow 0\quad \text{ as }\, i\to \infty.
\end{align*}
Using Proposition \ref{prop3.1}, making a change of
variable $z=m_{i}y,$ and then letting $\varepsilon\to 0,$ we have,
\begin{align}\label{3.11}
\int_{\mathbb{R}^{n}}
\langle z, \nabla^{2}\wdt{K}(0)(z+z_{0})\rangle (1+k|z|^{2})^{-n}=0,
\end{align}
where $z_{0}=\lim_{i\to\infty} m_{i}y_{i}$
and $k=\lim_{i\to\infty}\wdt{K}_{i}(y_{i})^{1/\sigma}/4.$

It follows from \eqref{3.3} that
$$
\int_{\mathbb{R}^{n}}\nabla\wdt{K}_{i}(y+y_{i})u_{i}(y+y_{i})^{n}=0.
$$
Using the same method above, we obtain
\begin{align}\label{3.12}
\int_{\mathbb{R}^{n}}
\nabla^{2}\wdt{K}(0)(z+z_{0}) (1+k|z|^{2})^{-n}=0.
\end{align}
It follows from \eqref{3.11} and \eqref{3.12} that
\begin{align}\label{3.17}
\int_{\mathbb{R}^{n}}\frac{1}{2}
\langle z+z_{0}, \nabla^{2}\wdt{K}(0) (z+z_{0})\rangle
 (1+k|z|^{2})^{-n}=0.
\end{align}
From \eqref{3.12} and \eqref{3.17} we can see that \eqref{3.13} does not hold for $\widehat{y}=k^{1/2}z_{0}.$
Therefore, we proved that $k>1.$

By  Part (i) of Theorem \ref{thm1}, we have $\{q^{(1)},\cdots, q^{(k)}\}\subset \mathscr{K}^{-}$
and $\mu(M(q^{(1)},\cdots,q^{(k)}))\geq 0.$
It follows from $v_{i}\in \mathscr{M}_{K_{i}}$ that $\tau_{i}=0.$
Applying Part (iv) of Theorem \ref{thm1}, we deduce  that $\mu(M(q^{(1)}, \cdots, q^{(k)}))=0.$
 This leads to a contradiction with $K \in \mathscr{A}.$
From the Harnack inequality  and Schauder  type estimates,
we complete the proof of  Theorem \ref{thm2.1}.
\end{proof}

\section{The degree-counting formula and  existence results}\label{sec4}

This section is devoted to the proof of Theorems \ref{thm4}, \ref{thm2}, and \ref{thm5}.
 It is worth noting that
due to  Theorem \ref{thm2.1},
homotopy invariance of Leray-Schauder degree and the properties of ``$\mathrm{Index}$'',
 we only need to prove Theorem \ref{thm4}
 for $K\in \mathscr{A}$ being a Morse function. Once this is achieved,
 we also prove that
the $\mathrm{Index}$ as in Definition \ref{defn1.2} is well defined on $\mathscr{A}.$
Therefore, we always assume that $K\in \mathscr{A}$ is a Morse function in this section.

\subsection{On the case of subcritical equations}
Let $\sigma=1+m/2,$ $m\in \mathbb{N}_{+},$ and $n=2\sigma+2.$
In this subsection,  we consider the following  subcritical equation:
\be\label{subequ}
P_{\sigma}v=c(n,\sigma)K v^{n-1-\tau}
\quad \text{ on }\, \mathbb{S}^n,
\ee
where $c(n,\sigma)=\Gamma(n-1),$ $K\in C^{2}(\mathbb{S}^n),$ and $\tau>0.$

We will soon prove that when $K\in \mathscr{A},$   the solutions to
\eqref{subequ} either stay bounded and converge  to the solutions to critical equations \eqref{1.1} in
$C^{2\sigma}$ norm or become unbounded and blow up at finite points as $\tau \rightarrow 0^{+}$.

Denote the  $H^{\sigma}(\mathbb{S}^n)$ inner product and norm by
$$
\langle u, v\rangle=\int_{\mathbb{S}^{n}}(P_{\sigma} u) v,
\quad\|u\|_{\sigma}=\sqrt{\langle u, u\rangle}.
$$
The  Euler-Lagrange functional associated  with \eqref{subequ} is
\be\label{fun}
I_{\tau}(u)=\frac{1}{2}\int_{\mathbb{S}^n}(P_{\sigma} u) u-\frac{
\Gamma(n-1)}{n-\tau}\int_{\mathbb{S}^n}K|u|^{n-\tau},\quad \forall\, u\in H^{\sigma}(\mathbb{S}^n).
\ee

\begin{definition}\label{1.89}
Let $K\in C^{2}(\mathbb{S}^n),$ $\mathscr{K}^{-}$ be as in \eqref{1.2} and $k\in \mathbb{N}_{+}.$
Let $\overline{P}_{1}, \cdots, \overline{P}_{k} \in \mathscr{K}^{-}$ be the critical points
of $K$ with $\mu(M(\overline{P}_{1}, \cdots, \overline{P}_{k}))>0,$
and  $\varepsilon_{0}>0$ be sufficiently small. Define
$$
\begin{aligned}
\Omega_{\varepsilon_{0}}=&\Omega_{\varepsilon_{0}}
(\overline{P}_{1}, \cdots, \overline{P}_{k}) \\
=&\{(\alpha, t, P) \in \mathbb{R}_{+}^{k} \times
\mathbb{R}_{+}^{k} \times(\mathbb{S}^{n})^{k}:|\alpha_{i}-(K(P_{i}))^{-{1}/{2\sigma}}|<\varepsilon_{0},\\
&\quad t_{i}>1/\varepsilon_{0},\, |P_{i}-\overline{P}_{i}|<\varepsilon_{0},\, 1 \leq i \leq k\}.
\end{aligned}
$$
\end{definition}

For $P\in \mathbb{S}^n$ and $t>0,$
\be\label{delta}
\delta_{P, t}(x)=\frac{t}
{1+\frac{t^{2}-1}{2}(1-\cos\, d(x, P))},\quad x\in \mathbb{S}^{n}
\ee
 is the family of the solutions for
\be\label{1.51}
P_{\sigma} v=\Gamma(n-1) v^{n-1}, \quad v>0 \quad\text { on }\, \mathbb{S}^{n}.
\ee

We have the following lemma based on the ideas provided by Bahri in \cite{Ba}:
\begin{lemma}\label{lem8}
 Let $\varepsilon_{0}$ be sufficiently small and $\Omega_{\varepsilon_{0}}=\Omega_{\varepsilon_{0}}(\overline{P}_{1}, \cdots, \overline{P}_{k})$ be as in Definition \ref{1.89}. For any $u\in H^{\sigma}(\mathbb{S}^n)$
satisfying
$$
\Big\|u-\sum_{i=1}^k\wdt{\alpha}_{i}\delta_{\wdt{P}_{i},\wdt{t}_{i}}
\Big\|_{\sigma}<\frac{\varepsilon_{0}}{2}
$$
 for some $(\wdt{\alpha},\wdt{t},\wdt{P})\in \Omega_{\varepsilon_0/2},$
then there exists a unique $(\alpha,t,P)\in \Omega_{\varepsilon_{0}}$ such that
$$
u=\sum_{i=1}^{k}\alp_{i}\delta_{P_{i},t_{i}}+v,
$$
with $v$
satisfies
 \be\label{1.12}
\langle v, \delta_{P_{i}, t_{i}}\rangle=\big\langle v, \frac{\partial \delta_{P_{i}, t_{i}}}{\partial P_{i}^{(\ell)}}\big\rangle=\big\langle v, \frac{\partial \delta_{P_{i}, t_{i}}}{\partial t_{i}}\big\rangle=0,
\ee
where $\frac{\partial}{\partial P_{i}^{(\ell)}}$ denotes the corresponding derivatives.
\end{lemma}

In what follows, we say that $v \in E_{P,t}$ if $v$ satisfies \eqref{1.12}
and we work in some orthonormal basis near $\{\overline{P}_{1},\cdots,\overline{P}_{k}\}.$

\begin{definition}\label{sta}
Let $\tau,$ $\varepsilon_{0},$ $\nu_{0}>0$ be sufficiently small, $A>0$ be sufficiently large,
and $\Omega_{\varepsilon_{0}/2}=\Omega_{\varepsilon_{0}/2}(\overline{P}_{1},\cdots,\overline{P}_{k})$
be as in Definition \ref{1.89}. Define
\be\label{stau}
\begin{aligned}
& \Sigma_{\tau}(\overline{P}_{1}, \cdots, \overline{P}_{k}) \\
=&\{(\alpha, t, P, v)
\in \Omega_{\varepsilon_{0}/2}\times H^{\sigma}(\mathbb{S}^{n}):\\
&\quad|P_{i}-\overline{P}_{i}|<\tau^{1/2}|\log \tau|,\,
A^{-1} \tau^{-1 / 2}<t_{i}<A \tau^{-1 / 2},\, v \in E_{P,t},\,\|v\|_{\sigma}<\nu_{0}\}.
\end{aligned}
\ee
Without confusion we use the same notation for
$$
\Sigma_{\tau}(\overline{P}_{1},\cdots,\overline{P}_{k})=\Big\{u=\sum_{i=1}^{k} \alpha_{i} \delta_{P_{i}, t_{i}}+v:(\alpha, t, P, v)
 \in \Sigma_{\tau}\Big\} \subset H^{\sigma}(\mathbb{S}^{n}).
$$
\end{definition}

Combined with Theorem \ref{thm1},  we can obtain
 the necessary conditions on blowing up solutions to \eqref{subequ} when $K\in \mathscr{A}$
 as $\tau$ tends to $0^{+}.$
\begin{proposition}\label{prop3}
Let $\sigma=1+m/2,$ $m\in \mathbb{N}_{+},$ and $n=2\sigma+2.$
Let $K \in \mathscr{A}$ be a  Morse function and $\mathscr{K}^{-}$ be as in \eqref{1.2}.
 Then for any $\alpha \in (0,1),$ there exists some
positive constants $\varepsilon_{0},\nu_{0} \ll 1,$ and  $A, R \gg 1$ depending only on $K,$ such that
 when $\tau>0$ is sufficiently small,  for all u satisfying $u \in H^{\sigma}
(\mathbb{S}^{n}),$ $u>0,$ $I_{\tau}^{\prime}(u)=0,$ we have
$$
u \in \mathscr{O}_{R} \cup\{\cup_{k \geq 1}
\cup_{\overline{P}_{1}, \cdots, \overline{P}_{k} \in \mathscr{K}^{-},\,
\mu(M(\overline{P}_{1}, \cdots, \overline{P}_{k}))>0}
 \Sigma_{\tau}(\overline{P}_{1}, \cdots, \overline{P}_{k})\},
$$
where $I_{\tau}'(u)$ is as in \eqref{subequ}, $\mathscr{O}_{R}$ is as in \eqref{1.82}
and $\Sigma_{\tau}(\overline{P}_{1},\cdots,\overline{P}_{k})$ is as in \eqref{stau}.
\end{proposition}
\begin{proof}
For any $\tau>0$  sufficiently small, let $u_{\tau}\in H^{\sigma}(\mathbb{S}^n),$ $u_{\tau}>0$
be a critiacl point of $I_{\tau}(u).$ If $u_{\tau}$ is uniformly bounded,
then by the Schauder type estimates we know that there exists a $R>0$ such that $u_{\tau}\in \mathscr{O}_{R}. $ The proof is now completed.
If not,   there exists $\tau_{i}\rightarrow 0$ such that $\max_{\sn}u_{\tau_{i}}\rightarrow \infty.$
It follows from Theorem \ref{thm1} and $K\in \mathscr{A}$ that
there exists a constant $\delta^{*}>0$ such that
$\{u_{\tau_{i}}\}$ has only isolated simple blow up points
$q^{(1)},\cdots, q^{(k)}\in \mathscr{K}^{-},$ with $|q^{(j)}-q^{(\ell)}|\geq \delta^{*}
,$ $\forall\, j\ne \ell,$ and $\mu (M(q^{(1)},\cdots, q^{(k)}))>0.$
Then Proposition \ref{prop3} can be deduced from
Propositions \ref{prop9},  \ref{prop4}, \ref{prop7}, and Lemma \ref{lem8}.
\end{proof}

Now we are going to show that if $K\in \mathscr{A}$ is a Morse function,
 one can construct solutions highly concentrating at arbitrary
 points $q^{(1)}, \cdots, q^{(k)} \in \mathscr{K}^{-}$ provided $\mu(M(q^{(1)}, \cdots, q^{(k)}))>0.$

\begin{theorem}\label{thm3}
Let $\sigma=1+m/2,$ $m\in \mathbb{N}_{+},$ and $n=2\sigma+2.$
Let $ K \in \mathscr{A}$ be a Morse function and $\mathscr{K}^{-}$ be as in \eqref{1.2}.
Let $\tau, \varepsilon_{0},\nu_{0}>0$ be sufficiently  small,
$A>0$ be sufficiently  large and $k\in \mathbb{N}_{+}.$
Then for any $\overline{P}_1,\cdots,\overline{P}_{k}\in \mathscr{K}^{-}$
satisfying  $\mu(M(\overline{P}_{1},\cdots,\overline{P}_{k}))>0,$ we have
\be\label{1.75}
\deg_{H^{\sigma}}
(u-P_{\sigma}^{-1}(c(n,\sigma)K|u|^{2\sigma-\tau}u),
\Sigma_{\tau}(\overline{P}_{1},\cdots,\overline{P}_{k}), 0)
=(-1)^{k+\sum_{j=1}^{k}i(\overline{P}_{j})},
\ee
where $\deg_{H^{\sigma}}$ denotes the Leray-Schauder degree in $H^{\sigma}(\mathbb{S}^n),$
and $i(\overline{P}_{j})$ is the Morse index of $K$ at $\overline{P}_{j}.$
\end{theorem}

In order to prove Theorem \ref{thm3},
we need the following Lemmas \ref{lem4.1}, \ref{lem4.2} and
Propositions \ref{prop1}, \ref{lem5}, \ref{lem6}, \ref{lem7},
whose proofs mainly uses the estimates
 in the appendix.

\begin{lemma} \label{lem4.1}
Under the hypotheses  of  Theorem \ref{thm3},
in addition that $\Sigma_{\tau}=
\Sigma_{\tau}(\overline{P}_{1},\cdots,\overline{P}_{k})$ is as in Definition \ref{sta} for the given
$\tau, \varepsilon_{0},\nu_{0}, A,$ and $\overline{P}_{1},\cdots,\overline{P}_{k}\in \mathscr{K}^{-}.$
Then for any $(\alpha, t, P, v) \in \Sigma_{\tau},$
we have:
$$
\begin{aligned}
&I_{\tau}\Big(\sum_{i=1}^{k}\alpha_{i}\delta_{P_i,t_i}+v\Big)\\
=&\frac{\Gamma(n-1)}{2}\Big(
\sum_{i=1}^{k} \alpha_{i}^{2}\int_{\mathbb{S}^{n}}\delta_{P_{i},t_{i}}^{n}
+\sum_{i\ne j} \alpha_{i} \alpha_{j}
\int_{\mathbb{S}^{n}}\delta_{P_{i}, t_{i}}^{n-1} \delta_{P_{j}, t_{j}}
\Big)\\
\quad&
-\frac{\Gamma(n-1)}{n-\tau}\int_{\mathbb{S}^{n}}
K\Big(\sum_{i=1}^{k}\alpha_{i}\delta_{P_{i},t_{i}}\Big)^{n-\tau}
+f_{\tau}(v)+ Q_{\tau}(v,v)+V(\tau, \alpha, t, P, v),
\end{aligned}
$$
where
\be\label{1.24}
f_{\tau}(v):=-\Gamma(n-1)\int_{\mathbb{S}^{n}}
K\Big(\sum_{i=1}^{k}\alpha_{i}\delta_{P_{i},t_{i}}\Big)^{n-1-\tau}v,
\ee
\be\label{2.24}
Q_{\tau}(v,v):=\frac{1}{2}\int_{\sn}(P_{\sigma}v)v
-(n-1-\tau)\frac{\Gamma(n-1)}{2}\int_{\mathbb{S}^{n}}
K\Big(\sum_{i=1}^{k}\alpha_{i}\delta_{P_{i},t_{i}}\Big)^{2\sigma-\tau}v^2,
\ee
and there exists a constant $C>0$ depends only on $K, \nu_{0},$ and $A$ such that
$$
|V(\tau, \alpha, t, P, v)| \leq C\|v\|_{\sigma}^{3}.
$$
\end{lemma}

\begin{proof}
By \eqref{fun} and \eqref{1.12}, we have
\be\label{1.26}
 \begin{aligned}
 &I_{\tau}\Big(\sum_{i=1}^{k}\alpha_{i}\delta_{P_i,t_i}+v\Big)\\
=&\frac{\Gamma(n-1)}{2}
\Big(
\sum_{i=1}^{k}\alpha_{i}^2
\int_{\mathbb{S}^n}\delta_{P_i,t_i}^{n}
+\sum_{j\ne i}\alpha_{i}\alpha_{j}
\int_{\mathbb{S}^n}\delta_{P_i,t_i}^{n-1}\delta_{P_j,t_j}
\Big)
+\frac{1}{2}
\int_{\mathbb{S}^n}(P_{\sigma} v) v
\\
&-\frac{\Gamma(n-1)}{n-\tau}
\int_{\mathbb{S}^n}K\Big|\sum_{i=1}^{k}\alpha_i\delta_{P_i,t_i}+v\Big|^{n-\tau}.
\end{aligned}
 \ee
Then, it follows from  Lemma \ref{cplema1} and \eqref{p1p1} that
Lemma \ref{lem4.1} holds.
\end{proof}

\begin{lemma}\label{lem4.2}
Under the hypotheses of Lemma \ref{lem4.1}, in addition that
$E_{P,t}$ is as in \eqref{1.12}.
Then for any $(\alpha, t, P, v) \in \Sigma_{\tau},$
 there exists some function $V_{v}$ and a constant $C>0$
 depending only on $K, \nu_{0},$ and $A$ such that
$$
I_{\tau}'\Big(\sum_{i=1}^{k}\alpha_{i}\delta_{P_{i},t_{i}}+v\Big)\varphi
=f_{\tau}(\varphi)+2Q_{\tau}(v,\varphi)+
\langle V_{v}(\tau,\alpha,t,P,v),\varphi\rangle,
$$
and
$$
\|V_{v}(\tau,\alpha,t,P,v)\|_{\sigma}\leq
   \|v\|_{\sigma}^{2},
$$
where $f_{\tau}(v)$ is as in \eqref{1.24} and $Q_{\tau}(v,\varphi)$ is as in \eqref{2.24}.
\end{lemma}
\begin{proof}
For any $\varphi \in E_{P,t},$  by using \eqref{1.26}, Lemma \ref{cplema1}, and \eqref{1.12}, we have
\begin{align*}
&I_{\tau}'\Big(\sum_{i=1}^{k}\alpha_{i}\delta_{P_i,t_i}+v\Big)\varphi\\
=& \int_{\sn}P_{\sigma}(v)\varphi
-\Gamma(n-1)\int_{\sn}K\Big|\sum_{i=1}^{k}\alpha_{i}
\delta_{P_{i},t_{i}}+v\Big|^{2\sigma-\tau}
\Big(\sum_{i=1}^{k}\alpha_{i}\delta_{P_{i},t_{i}}+v\Big)
\varphi\\
=&\int_{\sn}P_{\sigma}(v)\varphi
-\Gamma(n-1)\int_{\sn}K\Big( \sum_{i=1}^{k} \alpha_{i}\delta_{P_{i},t_{i}}\Big)^{n-1-\tau}
\varphi\\
\quad &-\Gamma(n-1)(n-1-\tau)\int_{\sn}K\Big(\sum_{i=1}^{k}
\alpha_{i}\delta_{P_{i},t_{i}}\Big)^{2\sigma-\tau}v\varphi\\
&+\langle V_{v}(\tau,\alpha,t,P,v),\varphi \rangle.
\end{align*}
Then, the estimates of $V_{v}(\tau,\alpha,t,P,v)$ can be
can be obtained by Sobolev imbedding and \eqref{p1p1}.
\end{proof}

\begin{proposition}\label{prop1}
Under the hypotheses of the Theorem \ref{thm3},
in addition that $\Sigma_{\tau}(\overline{P}_{1},\cdots,
\overline{P}_{k})$ is as in \eqref{1.75}
and $E_{P,t}$ is as in \eqref{1.12} for the given $(\alpha, t,P).$
 Then there exists a unique minimizer $\overline{v}=\overline{v}_{\tau}(\alpha, t, P)
  \in E_{P,t}$ of  $I_{\tau}(\sum_{i=1}^{k} \alpha_{i} \delta_{P_{i}, t_{i}}+v)$
   with respect to $\{v \in E_{P,t}:\|v\|_{\sigma}<\nu_{0}\}.$ Furthermore,
   there exists a constant $C$ independent of $\tau$ such that
\be\label{2.25}
\|\overline{v}\|_{\sigma}\leq
C\sum_{i=1}^{k}|\nabla K(P_i)|\tau^{1/2}+C\tau|\log \tau|\leq C\tau|\log \tau|.
\ee
\end{proposition}

\begin{proof}
From Lemma \ref{lem4.2}, we have, for all $\varphi\in E_{P,t},$
\be\label{1.66}
I_{\tau}'\Big(\sum_{i=1}^{k}\alpha_{i}\delta_{P_{i},t_{i}}+v\Big)\varphi
=f_{\tau}(\varphi)+2Q_{\tau}(v,\varphi)+
\langle V_{v}(\tau,\alpha,t,P,v),\varphi\rangle,
\ee
where
$$
f_{\tau}(\varphi):=-
\Gamma(n-1)
\int_{\mathbb{S}^{n}}K\Big(\sum_{i=1}^{k}\alpha_{i}\delta_{P_{i},t_{i}}\Big)^{n-1-\tau}\varphi,
$$
and
$$
Q_{\tau}(v,\varphi):=\frac{1}{2}\int_{\sn}(P_{\sigma}v)\varphi
-(n-1-\tau)\frac{\Gamma(n-1)}{2}\int_{\mathbb{S}^{n}}
K\Big(\sum_{i=1}^{k}\alpha_{i}\delta_{P_{i},t_{i}}\Big)^{2\sigma-\tau}v\varphi.
$$
 It is obviously that $f_{\tau}$ is a continuous linear functional over $E_{P,t}$,
 there exists   a unique $\widetilde{f}_{\tau}\in E_{P,t}$ such that
\be\label{1.47}
 f_{\tau}(\varphi)=\langle \widetilde{f}_{\tau}, \varphi\rangle,\quad\forall\,\varphi\in E_{P,t}.
\ee
By the same method of proving the coercivity of the quadratic form $Q_{\tau}$ in \cite{ACN,clz},
 it follows  that there exists a constant $\delta_{0}>0$ (independent of $\tau$) such that
\be\label{1.68}
Q_{\tau}(v,v)\geq \frac{\delta_{0}}{2}\|v\|_{\sigma}^{2},\quad \forall
\, (\alpha,t,P,v)\in \Sigma_{\tau},
\ee
thus,  there  exists a unique symmetric continuous and coercive operator $\widetilde{Q}_{\tau}$ from
$E_{P,t}$ onto itself such that,
\be\label{1.67}
Q_{\tau}(v, \varphi)=\langle \widetilde{Q}_{\tau}v, \varphi\rangle,\quad \forall\, \varphi\in E_{P,t}.
\ee
Using these notations, \eqref{1.66}, \eqref{1.47}, and \eqref{1.67}, we have
\be\label{1.96}
I_{\tau}'\Big(\sum_{i=1}^{k}\alpha_{i}\delta_{P_{i},t_{i}}+v\Big)
=\widetilde{f}_{\tau}+2\widetilde{Q}_{\tau}v+V_{v}(\tau,\alpha,t,P,v).
\ee

There is an equivalence between the existence of minimizer $\overline{v}_{\tau}$ and
\be\label{1.46}
\widetilde{f}_{\tau}+2\widetilde{Q}_{\tau}v+V_{v}(\tau,\alpha,t,P,v)=0,\quad v\in E_{P,t}.
\ee
 As in \cite{Liu_cpaa,rey}, by the implicit function theorem, there exist  a
 $C^{1}$-map $\overline{v}: (\alpha,t,P)\mapsto E_{P,t}$
 satisfying  \eqref{1.46} and
\be\label{1.48}
\|\overline{v}\|_{\sigma}\leq C\|\widetilde{f}_{\tau}\|_{\sigma}.
\ee
Therefore, in order to prove \eqref{2.25}, we only need to estimate $\|\widetilde{f}_{\tau}\|_{\sigma}.$

 Applying  Lemma \ref{lem3}, \eqref{1-tau}, \eqref{a.9}, \eqref{stau},
 and \eqref{p-p1}, we can obtain
$$
\begin{aligned}
f_{\tau}(v)=&-\Gamma(n-1)\int_{\mathbb{S}^{n}}K
\Big(\sum_{i=1}^{k} (\alpha_{i}\delta_{P_{i},t_{i}})^{n-1-\tau}\Big)v
+O\Big( \sum_{i\ne j}\int_{\sn} \delta_{P_{i},t_{i}}^{n-2-\tau}\delta_{P_{j},t_{j}}|v|\Big)\\
=&-\Gamma(n-1)\int_{\mathbb{S}^{n}}(K-K(P_{i}))\sum_{i=1}^{k}
\alpha_{i}^{n-1-\tau}\delta_{P_{i},t_{i}}^{n-1}v\\
&+
O\Big(\sum_{i=1}^{k}\int_{\sn}|
\delta_{P_{i},t_{i}}^{n-1-\tau}-\delta_{P_{i},t_{i}}^{n-1}||v|\Big)+O\Big(
\sum_{i\ne j}\|\delta_{P_i,t_i}^{n-2-\tau}\delta_{P_j,t_j}\|_{L^{n/(n-1)}(\sn)}\|v\|_{\sigma}
\Big)\\
=&O\Big(\sum_{i=1}^{k}
|\nabla_{g_{0}}K(P_{i})|\int_{\sn}|P-P_{i}|\delta_{P_{i},t_{i}}^{n-1} |v|\Big)
+O\Big(\sum_{i=1}^k\int_{\mathbb{S}^n}|P-P_{i}|^2\delta_{P_i,t_{i}}^{n-1}|v|\Big)\\
&+O(\tau |\log \tau|\|v\|_{\sigma}),
\end{aligned}
$$
where $|P-P_{i}|$ represents the distance between
two points $P$ and $P_{i}$  after through a stereographic  projection with $P_{i}$ as the south pole of $\mathbb{S}^{n}.$

From \eqref{stau} and \eqref{p-p1}, we have, for all $(\alpha,t,P,v)\in  \Sigma_{\tau}(\overline{P}_{1},
\cdots,\overline{P}_{k})$,
\be\label{2.2}
\begin{aligned}
|f_{\tau}(v)|\leq &C\Big\{
\tau^{1/2}
\sum_{i=1}^{k}|\nabla K(P_i)|+\tau+\tau|\log \tau|
\Big\}\|v\|_{\sigma}\\
\leq &C\tau|\log \tau|\|v\|_{\sigma},
\end{aligned}
\ee
this,  combining \eqref{1.47} and \eqref{1.48}, we obtain \eqref{2.25}.
\end{proof}

\begin{proposition}\label{lem5}
Under the hypotheses of Theorem \ref{thm3},
then for any $(\alpha, t,P, v)\in \Sigma_{\tau}(\overline{P}_{1},\cdots,\overline{P}_{k}),$  we have
$$
\frac{\partial}{\partial \alpha_{i}} I_{\tau}\Big(\sum_{j=1}^{k} \alpha_{j} \delta_{P_{j}, t_{j}}+v\Big)
=-2\sigma\|\delta_{P_{i},t_{i}}\|_{\sigma}^{2}\beta_i+V_{\alpha_i}(\tau,\alpha,t,P,v),
$$
where $\beta=(\beta_{1},\cdots,\beta_{k}),$
$\beta_{i}:=\alpha_{i}-K(P_{i})^{-{1}/{2\sigma}},$  $i=1,\cdots,k,$
and
$$
V_{\alpha_i}(\tau,\alpha,t,P,v)=O(|\beta|^2)+O(\tau|\log\tau|)
+O(\|v\|_{\sigma}^{2}).
$$
Furthermore, let $\overline{v}$ be as in Proposition \ref{prop1}, then we have
$$
\frac{\partial}{\partial \alpha_{i}} I_{\tau}\Big(\sum_{j=1}^{k} \alpha_{j} \delta_{P_{j}, t_{j}}+\overline{v}\Big)
=-2\sigma\|\delta_{P_{i},t_{i}}\|_{\sigma}^{2}\beta_i+O(|\beta|^2+\tau|\log \tau|).
$$
\end{proposition}

\begin{proof}
Using Lemma \ref{cplema1},  \eqref{a.10}, \eqref{1-tau},
and Lemma \ref{lem3} we have
$$
\begin{aligned}
&\frac{\partial}{\partial \alpha_{i}} I_{\tau}\Big(
\sum_{j=1}^{k} \alpha_{j} \delta_{P_{j}, t_{j}}+v\Big)
\nonumber \\
=&\Gamma(n-1)\Big(\alpha_{i}\int_{\sn}\delta_{P_{i}, t_{i}}^{n}
+\sum_{j\ne i} \alpha_{j}
\int_{\mathbb{S}^{n}}\delta_{P_{i}, t_{i}}^{n-1} \delta_{P_{j}, t_{j}}
\Big)\\
&-\Gamma(n-1)\int_{\mathbb{S}^n}K
\Big|\sum_{i=1}^{k}\alpha_i\delta_{P_i,t_i}+v\Big|^{n-2-\tau}
\Big(\sum_{j=1}^{k}\alpha_j\delta_{P_j,t_j}+v\Big)\delta_{P_{i},t_{i}}\\
=&\Gamma(n-1)
\Big(\alpha_{i}\int_{\sn}\delta_{P_{i}, t_{i}}^{n}
-\int_{\sn} K\Big|\sum_{j=1}^{k} \alpha_{j}\delta_{P_{j},t_{j}}\Big|^{n-1-\tau}
\delta_{P_{i},t_{i}}\Big)\\
& -\Gamma(n-1)(n-\tau-1)\int_{\sn}
K\Big|\sum_{j=1}^{k} \alpha_{j}\delta_{P_{j},t_{j}}\Big|^{n-2-\tau}v\delta_{P_{i},t_{i}}+O(\tau)+O(\|v\|_{\sigma}^{2})\\
=&\Gamma(n-1)
\Big(\alpha_{i}\int_{\sn}\delta_{P_{i}, t_{i}}^{n}
-\int_{\sn}K\Big(\sum_{j=1}^{k} (\alpha_{j}\delta_{P_{j},t_{j}})^{n-1-\tau}\Big)\delta_{P_{i},t_{i}}\Big)\\
&-\Gamma(n-1)\int_{\sn}K\Big(\sum_{j=1}^{k}(\alpha_{j}\delta_{P_{j},t_{j}})^{n-2-\tau}\Big)v
\delta_{P_{i},t_{i}}+O(\tau)+O(\|v\|_{\sigma}^2).
\end{aligned}
$$
It follows from \eqref{p-p1} and \eqref{stau} that
\be\label{1.77}
\begin{aligned}
\int_{\mathbb{S}^n}K\alpha_{i}^{n-1}
\delta_{P_i,t_i}^{n-\tau}=&
\int_{\mathbb{S}^n}K(P_i)\alpha_{i}^{n-1}\delta_{P_i,t_i}^{n-\tau}
-\int_{\mathbb{S}^n}({K}(P)-K(P_i))\alpha_{i}^{n-1}\delta_{P_i,t_i}^{n-\tau}\\
=&
\int_{\mathbb{S}^n}K(P_i)\alpha_{i}^{n-1}\delta_{P_i,t_i}^{n-\tau}
+O(\tau).
\end{aligned}
\ee
Similarly, by \eqref{1.12}, \eqref{a.9}, \eqref{stau}, and
 \eqref{p-p1}, we have
\begin{align}\label{1.78}
&\int_{\mathbb{S}^n}
K\alpha_{i}^{n-2}\delta_{P_i,t_i}^{n-1-\tau}v\notag\\
=&
\int_{\mathbb{S}^n}K(P_i) \alpha_{i}^{n-2}\delta_{P_i,t_i}^{n-1}v
+\int_{\mathbb{S}^n}(K(P)-K(P_i))\alpha_{i}^{n-2}\delta_{P_i,t_i}^{n-1}v
+O( \tau|\log \tau|\|v\|_{{\sigma}})\notag\\
 =&O(\tau|\log\tau|)+O(\|v\|_{{\sigma}}^{2}).
\end{align}
By using the fact $|\alpha_i^{n-1-\tau}-\alpha_{i}^{n-1}|=O(\tau),$ \eqref{1.77}, \eqref{1.78},
 \eqref{a.3}, and \eqref{a.71} that
\begin{align*}
&\frac{\partial}{\partial \alpha_{i}} I_{\tau}\Big(
\sum_{j=1}^{k} \alpha_{j} \delta_{P_{j}, t_{j}}+v\Big) \\
=&\Gamma(n-1)\Big(\alpha_{i}\int_{\sn}\delta_{P_{i}, t_{i}}^{n}
-K(P_{i})
\int_{\mathbb{S}^{n}} \alpha_{i}^{n-1}\delta_{P_i,t_i}^{n-\tau}
-\int_{\mathbb{S}^n}K\alpha_i^{n-2}\delta_{P_i,t_i}^{n-1-\tau}v
\Big)\\
&+O(\tau|\log\tau|)+O(\|v\|_{\sigma}^{2})
\\
=&-2\sigma\|\delta_{P_{i},t_{i}}\|_{\sigma}^{2}\beta_{i}
+O(|\beta|^2)
+O(\tau|\log\tau|)+O(\|v\|_{\sigma}^{2}).
\end{align*}
Hence
$$
\frac{\partial}{\partial \alpha_{i}} I_{\tau}
\Big(\sum_{j=1}^{k} \alpha_{j} \delta_{P_{j}, t_{j}}+v\Big)
=-2\sigma\|\delta_{P_{i},t_{i}}\|_{\sigma}^{2}\beta_{i}+V_{\alpha_i}(\tau,\alpha,t,P,v),
$$
where
$$
V_{\alpha_i}(\tau,\alpha,t,P,v)=O(|\beta|^2)
+O(\tau|\log\tau|)
+O(\|v\|_{\sigma}^{2}).
$$
Combining with \eqref{2.2}, we obtain
\be\label{1.18}
\begin{aligned}
\frac{\partial}{\partial \alpha_{i}} I_{\tau}\Big(\sum_{j=1}^{k} \alpha_{j} \delta_{P_{j}, t_{j}}+\overline{v}\Big)
=&-2\sigma\|\delta_{P_{i},t_{i}}\|_{\sigma}^{2}\beta_{i}+V_{\alpha_i}(\tau,\alpha,t,P,\overline{v})\\
=&-2\sigma\|\delta_{P_{i},t_{i}}\|_{\sigma}^{2}\beta_{i}+O(|\beta|^2+\tau|\log \tau|).
\end{aligned}
\ee
Proposition \ref{lem5} follows from the above.
\end{proof}

\begin{proposition}\label{lem6}
Under the hypotheses of  Proposition \ref{lem5},
then for any $(\alpha, t,P, v)\in \Sigma_{\tau}(\overline{P}_{1},\cdots,\overline{P}_{k}),$  we have
$$
\begin{aligned}
\frac{\partial}{\partial t_{i}}I_{\tau}\Big(\sum_{j=1}^{k} \alpha_{j}
\delta_{P_{j}, t_{j}}+v\Big)
=&\Theta_{1}\frac{1}{K(P_{i})^{{1}/{\sigma}}}\frac{\tau}{t_{i}}
+\Theta_{2}\frac{\Delta_{g_{0}}K(P_i)}{K(P_{i})^{{n}/{2\sigma}}}\frac{1}{t_{i}^{3}}\\
&+\Theta_3\sum_{j\ne i} \frac{G_{P_{i}}(P_{j})}{(K(P_{i})K(P_{j}))^{{1}/{2\sigma}}}
\frac{1}{t_i^2t_j}
+V_{t_i}(\tau,\alpha, t,P,v),
\end{aligned}
$$
where $\Theta_{1},$ $\Theta_{2},$ $\Theta_{3}$ are positive constants,
$G_{P_{i}}(P_{j})$ is as in \eqref{1.8}, and
$$
V_{t_i}(\tau,\alpha, t,P,v)=O(\tau\|v\|_{\sigma})+O(\tau^{1/2}\|v\|_{\sigma}^{2})+O(|\beta|\tau^{3/2})+o(\tau^{3/2}).
$$
\end{proposition}

\begin{proof}
By \eqref{1.26},  Lemma \ref{cplema1}, H\"{o}lder inequality, and Sobolev embedding,
we have
\begin{align}\label{1.31}
\begin{aligned}
&\frac{\partial}{\partial t_{i}}I_{\tau}\Big(\sum_{j=1}^{k}
 \alpha_{j} \delta_{P_{j}, t_{j}}+v\Big)\\
=&\Gamma(n-1)
\Big(\sum_{j\ne i}\alpha_{i}\alpha_{j}\frac{\partial }{\partial t_i}
\int_{\mathbb{S}^n}\delta_{P_i, t_i}^{n-1}\delta_{P_j,t_j}
-\int_{\mathbb{S}^n}K\Big( \sum_{j=1}^{k}\alpha_j\delta_{P_j,t_j}\Big)^{n-1-\tau}\alpha_{i}
\frac{\partial\delta_{P_i, t_i}}{\partial t_{i}}
\Big)\\
&
-\Gamma(n-1)(n-1-\tau)\int_{\mathbb{S}^n}
K\Big(\sum_{j=1}^{k}\alpha_{j}\delta_{P_j,t_j}\Big)^{n-2-\tau}v
\alpha_{i}\frac{\partial \delta_{P_i,t_i}}{\partial t_i}\\
 &
+O\Big(\|v\|_{\sigma}^{2}\Big\|\frac{\partial \delta_{P_i,t_i}}{\partial t_i}\Big\|_{\sigma}\Big).
\end{aligned}
\end{align}

From \eqref{1.12},  we can obtain
\begin{align}\label{1.55}
\begin{aligned}
\int_{\mathbb{S}^n}
\delta_{P_i,t_i}^{n-2}
\frac{\partial \delta_{P_i,t_i}}{\partial t_i} v
=&\frac{2}{n+2\sigma}\int_{\mathbb{S}^n}v\frac{\partial}{\partial t_i}
(\delta_{P_i,t_i}^{n-1})\\
=&\frac{2}{(n+2\sigma)c(n,\sigma)}\frac{\partial}{\partial t_i}\langle v,\delta_{P_i,t_i}\rangle\\
=&\frac{2}{(n+2\sigma)c(n,\sigma)}\Big\langle v, \frac{\partial \delta_{P_i,t_i}}{\partial t_i}\Big\rangle\\
=&0.
\end{aligned}
\end{align}
It follows from \eqref{1.55}, \eqref{stau},  \eqref{a.9}, \eqref{part11}, and \eqref{n.1} that
\begin{align*}
&\Big|\int_{\mathbb{S}^n} K\delta_{P_i,t_i}^{n-2-\tau}
\frac{\partial \delta_{P_i,t_i}}{\partial t_i} v\Big|\\
=&
\Big| \int_{\mathbb{S}^n}(K-K(P_i))\delta_{P_i,t_i}^{n-2}
\frac{\partial \delta_{P_i,t_i}}{\partial t_i}v
+\int_{\mathbb{S}^n}K(\delta_{P_i,t_i}^{n-2-\tau}-\delta_{P_i,t_i}^{n-2})
\frac{\partial \delta_{P_i,t_i}}{\partial t_i} v
 \Big|\\
 \leq&
C (\tau^{1/2}|\log \tau|)
 \int_{\mathbb{S}^n}|P-P_i|\delta_{P_i,t_i}^{n-2}
 \frac{\partial \delta_{P_i,t_i}}{\partial t_i}v\\
& +O\Big(\|\delta_{P_i,t_i}^{n-2-\tau}-\delta_{P_i,t_i}^{n-2}\|_{L^{n/(n-2)}(\mathbb{S}^n)}
 \Big\|\frac{\partial \delta_{P_i,t_i}}{\partial t_i}\Big\|_{\sigma}
 \|v\|_{\sigma}\Big)\\
 \leq&
 (\tau^{1/2}|\log \tau|)
 O\Big(\big\||\cdot-P_i|\delta_{P_i,t_i}^{n-2}\big\|_{L^{n/(n-2)}(\mathbb{S}^n)}
 \Big\|\frac{\partial \delta_{P_i,t_i}}{\partial t_i}\Big\|_{\sigma}
 \|v\|_{\sigma}\Big)+O\big(\tau^{3/2}|\log\tau|\|v\|_{\sigma}\big)\\
 \leq& C\tau^{3/2}|\log \tau|\|v\|_{\sigma},
\end{align*}
this, Lemma \ref{lem3},  \eqref{a.11},  and \eqref{n.2} yields,
\begin{align}\label{1.27}
&\Big|
\int_{\mathbb{S}^n}K\Big(\sum_{j=1}^{k}\alpha_{j}\delta_{P_j,t_j}\Big)^{n-2-\tau}
\frac{\partial \delta_{P_i,t_i}}{\partial t_i}v
\Big|\notag\\
\leq &
\Big|\int_{\mathbb{S}^n}K\alpha_{i}^{n-2-\tau}
\delta_{P_i,t_i}^{n-2-\tau}\frac{\partial \delta_{P_i,t_i}}{\partial t_i}v\Big|
+C\sum_{j\ne i} \int_{\mathbb{S}^n} \delta_{P_{i},t_{i}}\delta_{P_j,t_j}^{n-3-\tau}
\Big|\frac{\partial \delta_{P_i,t_i}}{\partial t_i}\Big||v|\notag\\
&+C\sum_{j\ne i}\int_{\sn}\delta_{P_{j},t_{j}}^{n-2-\tau}
\Big|\frac{\partial \delta_{P_i,t_i}}{\partial t_i}\Big||v|\notag\\
\leq& C(\tau^{3/2}|\log \tau|\|v\|_{\sigma}+
\tau^{3/2}\|v\|_{\sigma})\notag\\
\leq& C(\tau^{3/2}|\log\tau|\|v\|_{\sigma}).
\end{align}

Using \eqref{1.31}, \eqref{1.27}, Lemma \ref{lem3}, and  \eqref{n.6}, we obtain
\begin{align}\label{4.2}
&\frac{\partial}{\partial t_{i}}I_{\tau}
\Big(\sum_{j=1}^{k} \alpha_{j} \delta_{P_{j}, t_{j}}+v\Big)\notag\\
=&\Gamma(n-1)
\Big(
\sum_{j\ne i}\alpha_{i}\alpha_{j}
\frac{\partial}{\partial t_i}\int_{\mathbb{S}^n}\delta_{P_j,t_j}
\delta_{P_i, t_i}^{n-1}
 -\int_{\mathbb{S}^n}K(\alpha_i\delta_{P_i,t_i})^{n-1-\tau}
    \alpha_{i}\frac{\partial\delta_{P_i, t_i}}{\partial t_{i}}
    \Big)\notag\\
&
  -\Gamma(n-1)\int_{\sn}K\Big(\sum_{j\ne i} \alpha_{j}\delta_{P_{j},t_{j}}\Big)^{n-1-\tau}
  \alpha_{i}\frac{\partial \delta_{P_{i},t_{i}}}{\partial t_{i}}\notag\\
&
-\Gamma(n-1)(n-1-\tau)\int_{\sn}
K(\alpha_{i}\delta_{P_{i},t_{i}})^{n-2-\tau}
\Big(\sum_{j\ne i}\alpha_{j}\delta_{P_{j},t_{j}}\Big)
\alpha_{i}\frac{\partial \delta_{P_{i},t_{i}}}{\partial t_{i}}\notag\\
&+o(\tau^{3/2})+O(\tau^{3/2}|\log \tau|\|v\|_{\sigma})
+O(\tau^{1/2}\|v\|_{\sigma}^{2})\notag\\
=&
\Gamma(n-1)
\Big(
\sum_{j\ne i}\alpha_{i}\alpha_{j}\frac{\partial}{\partial t_i}
\int_{\mathbb{S}^n}\delta_{P_j,t_j}\delta_{P_i,t_i}^{n-1}
    -\int_{\mathbb{S}^n}K\alpha_{i}^{n-\tau}\delta_{P_i,t_i}^{n-1-\tau}
\frac{\partial \delta_{P_i,t_i}}{\partial t_i}
\Big)\notag\\
&
-\Gamma(n-1)\int_{\sn}\alpha_{i}K
\Big(\sum_{j\ne i}\alpha_{j}\delta_{P_{j},t_{j}}\Big)^{n-1-\tau}
\frac{\partial \delta_{P_{i},t_{i}}}{\partial t_{i}}\notag\\
&
-\Gamma(n-1)(n-1)
\int_{\sn}\alpha_{i}^{n-1}
K\Big(\sum_{j\ne i}\alpha_{j}\delta_{P_{j},t_{j}}\Big)
\delta_{P_{i},t_{i}}^{n-2-\tau}
\frac{\partial \delta_{P_{i},t_{i}}}{\partial t_{i}}\notag\\
&+o(\tau^{3/2})+O(\tau^{3/2}|\log \tau|\|v\|_{\sigma})
+O(\tau^{1/2}\|v\|_{\sigma}^{2}).
\end{align}
It follows from Lemma \ref{lem3}  that
\be\label{4.1}
\Big(\sum_{j\ne i}\alpha_{j}\delta_{P_{j},t_{j}}\Big)^{n-1-\tau}=
\sum_{j\ne i}(\alpha_j\delta_{P_{j},t_{j}})^{n-1-\tau}
+O\Big(\sum_{\substack{j\ne i,\ell \ne i, j\ne \ell}} \delta_{P_{j},t_{j}}^{n-2-\tau}\delta_{P_{\ell},t_{\ell}}\Big).
\ee
By \eqref{4.2}, \eqref{4.1}, \eqref{part11}, \eqref{1-tau}, and \eqref{n.6}, we can obtain
\begin{align*}
&\frac{\partial}{\partial t_{i}}I_{\tau}\Big(\sum_{j=1}^{k}
 \alpha_{j} \delta_{P_{j}, t_{j}}+v\Big)\notag\\
 =&\Gamma(n-1)
\Big(
\sum_{j\ne i}\alpha_{i}\alpha_{j}\frac{\partial}{\partial t_i}
\int_{\mathbb{S}^n}\delta_{P_j,t_j}\delta_{P_i,t_i}^{n-1}
    -\int_{\mathbb{S}^n}K\alpha_{i}^{n}\delta_{P_i,t_i}^{n-1-\tau}
\frac{\partial \delta_{P_i,t_i}}{\partial t_i}
\Big)
 \\
&
-\Gamma(n-1)\int_{\sn}\alpha_{i}K
\sum_{j\ne i}(\alpha_{j}\delta_{P_{j},t_{j}})^{n-1-\tau}
\frac{\partial \delta_{P_{i},t_{i}}}{\partial t_{i}}\\
&
-\Gamma(n-1)\alpha_{i}^{n-1}
\sum_{j\ne i}\int_{\sn}
K\alpha_{j}\delta_{P_{j},t_{j}}
\frac{\partial }{\partial t_{i}}(\delta_{P_{i},t_{i}})^{n-1-\tau}\\
&
+o(\tau^{3/2})+O(\tau^{3/2}|\log \tau|\|v\|_{\sigma})
+O(\tau^{1/2}\|v\|_{\sigma}^{2}).
\end{align*}
By \eqref{n.3}, we have
\begin{align}\label{4.4}
&\int_{\sn}K\delta_{P_{j},t_{j}}
\frac{\partial }{\partial t_{i}}(\delta_{P_{i},t_{i}})^{n-1-\tau}\notag\\
=&\frac{\partial }{\partial t_{i}}
\int_{\sn}K\delta_{P_{j},t_{j}}\delta_{P_{i},t_{i}}^{n-1-\tau}\notag\\
=&K(P_{i})\frac{\partial }{\partial t_{i}}
\int_{\sn}\delta_{P_{j},t_{j}}\delta_{P_{i},t_{i}}^{n-1-\tau}
+\frac{\partial }{\partial t_{i}}
\int_{\sn}(K-K(P_{i}))\delta_{P_{j},t_{j}}\delta_{P_{i},t_{i}}^{n-1-\tau}\notag\\
=&K(P_{i})\frac{\partial }{\partial t_{i}}
\int_{\sn}\delta_{P_{j},t_{j}}\delta_{P_{i},t_{i}}^{n-1-\tau}
+O(\tau^2),
\end{align}
and by \eqref{n.4},
\begin{align}\label{4.5}
&\int_{\sn}K\delta_{P_{j},t_{j}}^{n-1-\tau}
\frac{\partial \delta_{P_{i},t_{i}}}{\partial t_{i}}\notag\\
=&\frac{\partial }{\partial t_{i}}\int_{\sn} K
\delta_{P_{i},t_{i}}\delta_{P_{j},t_{j}}^{n-1-\tau}\notag\\
=&K(P_{j})\frac{\partial }{\partial t_{i}}\int_{\sn}
\delta_{P_{i},t_{i}}\delta_{P_{j},t_{j}}^{n-1-\tau}
+\frac{\partial }{\partial t_{i}}\int_{\sn}(K-K(P_{j}))
\delta_{P_{i},t_{i}}\delta_{P_{j},t_{j}}^{n-1-\tau}\notag\\
=&K(P_{j})\frac{\partial }{\partial t_{i}}\int_{\sn}
\delta_{P_{i},t_{i}}\delta_{P_{j},t_{j}}^{n-1-\tau}+O(\tau^{2}).
\end{align}
Thus,  from \eqref{4.4}, \eqref{4.5}, \eqref{a.1}, \eqref{part3-tau}, and \eqref{a.5},  we get
\begin{align*}
&\frac{\partial}{\partial t_{i}}I_{\tau}\Big(\sum_{j=1}^{k}
 \alpha_{j} \delta_{P_{j}, t_{j}}+v\Big)\notag\\
 =&
 \Gamma(n-1)
\Big(
\sum_{j\ne i}\alpha_{i}\alpha_{j}\frac{\partial}{\partial t_i}
\int_{\mathbb{S}^n}\delta_{P_j,t_j}\delta_{P_i,t_i}^{n-1}
    -\int_{\mathbb{S}^n}K\alpha_{i}^{n}\delta_{P_i,t_i}^{n-1-\tau}
\frac{\partial \delta_{P_i,t_i}}{\partial t_i}
\Big)
 \\
&
 -\Gamma(n-1)
 K(P_{i})\alpha_{i}^{n-1}\sum_{j\ne i}\alpha_{j}\frac{\partial }{\partial t_{i}}
\int_{\sn}\delta_{P_{j},t_{j}}\delta_{P_{i},t_{i}}^{n-1}\\
&
-\Gamma(n-1)
\alpha_{i}\sum_{j\ne i} K(P_{j})\alpha_{j}^{n-1} \frac{\partial }{\partial t_{i}}\int_{\sn}
\delta_{P_{i},t_{i}}\delta_{P_{j},t_{j}}^{n-1}\\
&
 +o(\tau^{3/2})+O(\tau^{3/2}|\log \tau|\|v\|_{\sigma})
+O(\tau^{1/2}\|v\|_{\sigma}^{2})\\
=&\Gamma(n-1)
\Big(
\sum_{j\ne i}\alpha_{i}\alpha_{j}\frac{\partial}{\partial t_i}
\int_{\mathbb{S}^n}\delta_{P_j,t_j}\delta_{P_i,t_i}^{n-1}
-\frac{1}{n-\tau}\int_{\mathbb{S}^n}K(P_{i})\alpha_{i}^{n}
\frac{\partial \delta_{P_i,t_i}^{n-\tau}}{\partial t_i}
\Big)
\\
& -\Gamma(n-1)\frac{2\Delta_{g_0} K(P_{i})}{n(n-\tau)}\int_{\sn}
|P-P_{i}|^2
\alpha_{i}^{n}\frac{\partial \delta_{P_{i},t_{i}}^{n-\tau}}{\partial t_{i}}\\
 &
-\Gamma(n-1)\sum_{j\ne i}\{\alpha_{i}^{n-1}\alpha_{j}K(P_{i})
+\alpha_{i}\alpha_{j}^{n-1}K(P_{j})\}
\frac{\partial}{\partial t_{i}}
\int_{\sn}\delta_{P_{j},t_{j}}\delta_{P_{i},t_{i}}^{n-1}\\
&
 +o(\tau^{3/2})+O(\tau^{3/2}|\log \tau|\|v\|_{\sigma})
+O(\tau^{1/2}\|v\|_{\sigma}^{2})\\
=&\Gamma(n-1)\sum_{j\ne i}
\{\alpha_{i}\alpha_{j}-\alpha_{i}^{n-1}\alpha_{j}K(P_{i})
-\alpha_{i}\alpha_{j}^{n-1}K(P_{j})\}
\frac{\partial}{\partial t_{i}}
\int_{\sn}\delta_{P_{j},t_{j}}\delta_{P_{i},t_{i}}^{n-1}\\
&
-\frac{\Gamma(n-1)}{n-\tau}\alpha_{i}^{n}K(P_{i})
\frac{\partial}{\partial t_{i}}\int_{\sn}\delta_{P_{i},t_{i}}^{n-\tau}\\
&
-\frac{2\Gamma(n-1)}{n(n-\tau)}\Delta_{g_0} K(P_{i})\alpha_{i}^{n}
\frac{\partial }{\partial t_{i}}\int_{\sn}
|P-P_{i}|^2\delta_{P_{i},t_{i}}^{n-\tau}\\
&
+o(\tau^{3/2})+O(\tau^{3/2}|\log \tau|\|v\|_{\sigma})
+O(\tau^{1/2}\|v\|_{\sigma}^{2})\\
=&-\Gamma(n-1)\sum_{j\ne i}\frac{1}{(K(P_{i})K(P_{j}))^{{1}/{2\sigma}}}
\frac{\partial }{\partial t_{i}}\int_{\sn}\delta_{P_{j},t_{j}}\delta_{P_{i},t_{i}}^{n-1}\\
 &- \frac{\Gamma(n-1)}{n}\frac{1}{K(P_{i})^{{1}/{\sigma}}}
\frac{\partial}{\partial t_{i}}\int_{\sn}\delta_{P_{i},t_{i}}^{n-\tau}\\
&-\frac{2\Gamma(n-1)}{n^2}\frac{\Delta_{g_{0}}K(P_{i})}{K(P_{i})^{{n}/{2\sigma}}}
\frac{\partial }{\partial t_{i}} \int_{\sn}|P-P_{i}|^2\delta_{P_{i},t_{i}}^{n-\tau}\\
&
+o(\tau^{3/2})+O(\tau^{3/2}|\log \tau|\|v\|_{\sigma})
+O(\tau^{1/2}\|v\|_{\sigma}^{2})+O(|\beta|\tau^{3/2}),
\end{align*}
where $|P-P_{i}|$ represents the distance between  two points $P$ and $P_{i}$  after through a stereographic  projection with $P_{i}$ as the south pole of $\mathbb{S}^{n}.$

It follows that
\begin{align}\label{1.32}
&\frac{\partial }{\partial t_{i}}I_{\tau}\Big(\sum_{j=1}^{k}
\alpha_{j}\delta_{P_{j},t_{j}}+v\Big)\notag\\
=&\Theta_{1}\frac{1}{K(P_{i})^{{1}/{\sigma}}}\frac{\tau}{t_{i}}
+\Theta_{2}\frac{\Delta_{g_{0}}K(P_i)}{K(P_{i})^{{n}/{2\sigma}}}\frac{1}{t_{i}^{3}}\notag\\
&\quad+\Theta_3\sum_{j\ne i} \frac{G_{P_{i}}(P_{j})}{(K(P_{i})K(P_{j}))^{{1}/{2\sigma}}}
\frac{1}{t_i^2t_j}+V_{t_{i}}(\tau,\alpha,t,P,v),
\end{align}
where
\begin{align*}
&\Theta_{1}
=2^{n-2}\Gamma(n-1)|\mathbb{S}^{n-1}|\frac{n-2}{n(n-1)}\ub(\frac{n}{2},\frac{n}{2}-1),\\
&\Theta_{2}
=2^{n}\Gamma(n-1)
|\mathbb{S}^{n-1}|\frac{1}{n(n-1)}\ub(\frac{n}{2},\frac{n}{2}-1),\\
&\Theta_{3}
=2^{n}\Gamma(n-1)|\mathbb{S}^{n-1}|\ub(\frac{n}{2},\frac{n}{2}-1),
\end{align*}
and
$$
V_{t_i}(\tau,\alpha,t,P,v)=o(\tau^{3/2})+O(\tau\|v\|_{\sigma})
+O(\tau^{1/2}\|v\|_{\sigma}^{2})+O(|\beta|\tau^{3/2}).
$$
Proposition \ref{lem6} follows from the above.
\end{proof}

\begin{proposition}\label{lem7}
Under the hypotheses  of  Proposition \ref{lem5},
then for any $(\alpha, t,P, v)\in \Sigma_{\tau}(\overline{P}_{1},\cdots,\overline{P}_{k}),$
we have
$$
\frac{\partial}{\partial P_{i}}
I_{\tau}\Big(\sum_{j=1}^{k} \alpha_{j} \delta_{P_{j}, t_{j}}+v\Big)
=-\Theta_4\nabla_{g_0}{K}(P_i)
+V_{P_i}(\tau,\alpha,t,P,v),
$$
where $\Theta_{4}\geq\nu_{1}>0$ is a constant,
and
$$
V_{P_i}(\tau,\alpha,t,P,v)=
  O(\tau^{1/2})+O( \|v\|_{\sigma})+O( \tau^{-1/2}\|v\|_{\sigma}^{2}).
$$
\end{proposition}

\begin{proof}
Using \eqref{1.26} and Lemma \ref{cplema1}, we have
\begin{align}\label{1.33}
&\frac{\partial}{\partial P_{i}}
I_{\tau}\Big(\sum_{j=1}^{k} \alpha_{j} \delta_{P_{j}, t_{j}}+v\Big)\notag\\
=&\Gamma(n-1)\sum_{j\ne i}
\alp_i\alp_j\frac{\partial}{\partial P_i}\int_{\mathbb{S}^n}
\delta_{P_{j},t_{j}}^{n-1}\delta_{P_i,t_i}\notag\\
&-\Gamma(n-1) \int_{\mathbb{S}^n}K\Big|\sum_{j=1}^{k}
\alp_j\delta_{P_j,t_j}+v\Big|^{n-2-\tau}\Big( \sum_{j=1}^{k}\alp_j\delta_{P_j,t_j}+v\Big)
\alp_i\frac{\partial\delta_{P_i, t_i}}{\partial P_{i}}\notag\\
=&\Gamma(n-1)\sum_{j\ne i}
\alp_i\alp_j\frac{\partial}{\partial P_i}\int_{\mathbb{S}^n}
\delta_{P_{j},t_{j}}^{n-1}\delta_{P_i,t_i}\notag\\
&-\Gamma(n-1)\int_{\mathbb{S}^n}K\Big(\sum_{j=1}^{k}\alp_{j}\delta_{P_j,t_j}\Big)^{n-1-\tau}
\alp_{i}\frac{\partial \delta_{P_i,t_i}}{\partial P_i}\notag\\
&-\Gamma(n-1)(n-1-\tau)\int_{\mathbb{S}^n}
K\Big(\sum_{j=1}^{k}\alpha_{j}\delta_{P_{j},t_{j}}\Big)^{n-2-\tau}v
\alpha_{i}\frac{\partial \delta_{P_i,t_i}}{\partial P_i}\notag\\
&+O\Big( \|v\|_{\sigma}^{2}
\Big\|\frac{\partial \delta_{P_i,t_i}}{\partial P_i}\Big\|_{\sigma}\Big).
\end{align}

It follows from Lemma \ref{lem3} that
\begin{align*}
\Big( \sum_{j=1}^{k} \alpha_{j}\delta_{P_j,t_j}\Big)^{n-2-\tau}
=&\Big(\alpha_{i}\delta_{P_i,t_i} +\sum_{j\ne i}\alpha_{j}\delta_{P_j,t_j}\Big)^{n-2-\tau}\\
=&(\alpha_{i}\delta_{P_{i},t_{i}})^{n-2-\tau}
+O\Big( \sum_{j\ne i}\delta_{P_{i},t_{i}}^{n-3-\tau}
\delta_{P_j,t_j}+\sum_{j\ne i}\delta_{P_j,t_j}^{n-2-\tau}\Big).
\end{align*}
By \eqref{n.5}, \eqref{P1P1}, \eqref{1-tau}, \eqref{n.1}, and \eqref{a.9}, we have
\begin{align*}
&\int_{\sn} K\Big(\sum_{j=1}^{k}\alpha_{j}\delta_{P_j,t_{j}}\Big)^{n-2-\tau}
v\alpha_{i}\frac{\partial \delta_{P_i,t_i}}{\partial P_i}\\
=&\int_{\sn}K(\alpha_{j}\delta_{P_j,t_{j}})^{n-2-\tau}
v\alpha_{i}\frac{\partial \delta_{P_i,t_i}}{\partial P_i}
+O\Big( \sum_{j\ne i}\int_{\sn} \delta_{P_{i},t_i}^{n-3-\tau}\delta_{P_j,t_j}
\Big| \frac{\partial \delta_{P_i,t_i}}{\partial P_{i}}\Big||v|\Big)\\
&+O\Big(\sum_{j\ne i}
\int_{\sn} \delta_{P_j,t_j}^{n-2-\tau}
\Big| \frac{\partial \delta_{P_i,t_i}}{\partial P_{i}}\Big||v|
\Big)\\
=&
\int_{\sn}K(P_{i})(\alpha_{j}\delta_{P_j,t_{j}})^{n-2-\tau}
v\alpha_{i}\frac{\partial \delta_{P_i,t_i}}{\partial P_i}
+O\Big(\int_{\sn} |P-P_{i}|\delta_{P_{i},t_{i}}^{n-2-\tau}
\Big|\frac{\partial \delta_{P_i,t_i}}{\partial P_i}\Big||v|\Big)
\\
&+O(\tau^{1/2}\|v\|_{\sigma})+O(\tau\|v\|_{\sigma})\\
=&O\Big(\int_{\sn}|\delta_{P_{i},t_{i}}^{n-2-\tau}-\delta_{P_{i},t_{i}}^{n-2}|
\Big|\frac{\partial \delta_{P_i,t_i}}{\partial P_i}\Big| |v|\Big)+O(\|v\|_{\sigma})\\
=&O(\|v\|_{\sigma}).
\end{align*}
From Lemma \ref{lem3}, we have
$$
\begin{aligned}
&\Big(\sum_{j=1}^{k}\alp_{j}\delta_{P_j,t_j}\Big)^{n-1-\tau}\\
=&\Big(\alpha_{i}\delta_{P_{i},t_{i}}+\sum_{j\ne i}\alpha_{j}\delta_{P_{j},t_{j}}\Big)^{n-1-\tau}\\
=&(\alpha_{i}\delta_{P_{i},t_{i}})^{n-1-\tau}
+\Big(\sum_{j\ne i}\alpha_{j}\delta_{P_{j},t_j}\Big)^{n-1-\tau}
+(n-1-\tau)\alpha_{i}\delta_{P_{i},t_{i}}^{n-2-\tau}\Big(\sum_{j\ne i}\alpha_{j}\delta_{P_{j},t_{j}}\Big)\\
&+O\Big(\sum_{j\ne i}\delta_{P_{i},t_{i}}^{n-3-\tau}\delta_{P_{j}t_{j}}^{2} \Big),
\end{aligned}
$$
then, by using \eqref{P1P1} and \eqref{n.7}, \eqref{n.8}, \eqref{n.9}, we can obtain
\begin{align*}
&\frac{\partial}{\partial P_{i}}
I_{\tau}\Big(\sum_{j=1}^{k} \alpha_{j} \delta_{P_{j}, t_{j}}+v\Big)\\
=&\Gamma(n-1)\sum_{j\ne i}
\alp_i\alp_j\frac{\partial}{\partial P_i}\int_{\mathbb{S}^n}
\delta_{P_{j},t_{j}}^{n-1}\delta_{P_i,t_i}\\
&
-\Gamma(n-1)\alpha_{i}\int_{\mathbb{S}^n}K(\alpha_{i}\delta_{P_i,t_i})^{n-1-\tau}
\frac{\partial \delta_{P_i,t_i}}{\partial P_i}\\
&
-\Gamma(n-1)\alpha_{i}(n-1-\tau)\int_{\sn}K\Big( \sum_{j\ne i} \alpha_{j}\delta_{P_{j},t_{j}}\Big)
(\alpha_{i}\delta_{P_i,t_i})^{n-2-\tau}\frac{\partial\delta_{P_{i},t_{i}}}{\partial P_i}\\
&
-\Gamma(n-1)\alpha_{i}\int_{\sn}\Big(\sum_{j\ne i}\alpha_{j}\delta_{P_{j},t_j}\Big)^{n-1-\tau}
\frac{\partial\delta_{P_{i},t_{i}}}{\partial P_i}\\
&+O(\|v\|_{\sigma})+O(\tau^{-1/2}\|v\|_{\sigma}^2)+O(\tau^{3/2})\\
=&-\Gamma(n-1)\alpha_{i}^{n}\int_{\mathbb{S}^n}K\delta_{P_i,t_i}^{n-1-\tau}
\frac{\partial \delta_{P_i,t_i}}{\partial P_i}+O(\tau^{1/2})+O(\|v\|_{\sigma})+O(\tau^{-1/2}\|v\|_{\sigma}^2)\\
=&-\Theta_{4}(\tau,\alpha, t,P,v)\nabla K(P_{i})
+V_{P_{i}}(\tau,\alpha,t,P,v),
\end{align*}
where
\begin{align*}
\Theta_4(\tau,\alpha, t,P,v)\geq \nu_{1}>0 \quad\text{ with }\, \nu_1 \text{ independent  of }\, \tau,
\end{align*}
and
\begin{align}\label{1.34}
V_{P_i}(\tau,\alpha,t,P,v)=O(\tau^{1/2})
+O(\|v\|_{\sigma})
+O( \tau^{-1/2}\|v\|_{\sigma}^{2}).
\end{align}

We now prove that the existence of $\nu_{1}.$
Let $P_{i}$ be the south pole and make a stereographic projection $F$ to the
equatorial plane of $\mathbb{S}^n$ with $y=(y_{(1)},\cdots ,y_{(n)})$
as the stereographic projection coordinates,
let $\widetilde{K}=K(F(y))$
and  $|J_{F}|:=(2/(1+|y|^{2}))^{n}.$
Then  we have $F(0)=P_{i}$ and
\begin{align*}
&\int_{\mathbb{S}^n}K\delta_{P_i,t_i}^{n-1-\tau}\frac{\partial \delta_{P_{i},t_{i}}}{\partial P_{i}}\\
=&\int_{\mathbb{R}^{n}}
\omega_{y_i,t_i}^{n-1}(\nabla\wdt{K}(0)\cdot y+O(|y|^2))
g_{\tau}(y)
\frac{\partial \omega_{y_{i},t_{i}}}{\partial y_{i}}\\
=&:\mathcal{L}=(\mathcal{L}^{(1)},\cdots,\mathcal{L}^{(n)}),
\end{align*}
where $\omega_{y_i,t_i}(y)=\frac{2t_i}{1+t_{i}^2|y|^2},$
and
$
g_{\tau}(y):=(\omega_{y_1,0}^{-1}|J_{F}|^{1/n})^{\tau}.
$
For $\ell =1,\cdots, n,$ we have
\begin{align*}
\mathcal{L}^{(\ell)}=&\int_{\mathbb{R}^{n}}
\omega_{y_i,t_i}^{n-1}(\nabla\wdt{K}(0)\cdot y+O(|y|^2))
g_{\tau}(y)\frac{\partial \omega_{y_{i},t_{i}}}{\partial y_{i}}
\\
=&\int_{\mathbb{R}^{n}}t_{i}y_{(\ell)}\omega_{y_i,t_i}^{n+1}
(\nabla \wdt{K}(0)+O(|y|^2))g_{\tau}(y)\\
=&\frac{1}{n}\frac{\partial \wdt{K}}{\partial y_{(\ell)}}(0)
\int_{\mathbb{R}^n}t_{i}|y|^2
\omega_{y_i,t_i}^{n+1}g_{\tau}(y)+O(\tau^{1/2}),
\end{align*}
thus,
$$
\mathcal{L}=\nabla_{g_{0}}K(P_{i})\frac{2}{n}
\int_{\mathbb{S}^n}t_{i}|y|^2
\omega_{y_i,t_i}^{n+1}g_{\tau}(y)+O(\tau^{1/2}).
$$
It follows from $t_{i}^{-\tau}\leq g_{\tau}(y)\leq t_{i}^{\tau}$ that
$$
\int_{\mathbb{S}^n}t_{i}|y|^2
\omega_{y_i,t_i}^{n+1}g_{\tau}(y)\geq t_{i}^{-\tau}
\int_{\mathbb{S}^n}t_{i}|y|^2
\omega_{y_i,t_i}^{n+1}\rightarrow
\int_{\mathbb{R}^{n}} \frac{2^{n+1}}{(1+|x|^2)^{n}}
$$
as $\tau\rightarrow 0.$ This ensures the existence of $\nu_{1}.$
 We have proved Proposition \ref{lem7}.
\end{proof}

By using Propositions \ref{prop1}, \ref{lem5}, \ref{lem6}, \ref{lem7},
and constructing a family of homotopy Id+compact operators,
 we will obtain the degree-counting formula of the solutions to the subcritical equation \eqref{subequ} on
$\Sigma_{\tau}(\overline {P}_{1},\cdots,\overline{P}_{k}).$

\begin{proof}[Proof of Theorem \ref{thm3}]

The $\mathscr{K}^{-}$ be as in \eqref{1.2} for the given $K$ and
 $\Sigma_{\tau}(\overline{P}_{1},\cdots,\overline{P}_{k})$ be as in \eqref{stau}
for the given $\overline{P}_{1},\cdots,\overline{P}_{k}\in \mathscr{K}^{-}.$

For $u=\sum_{i=1}^{k}\alp_{i}\delta_{P_i,t_i}+v\in \Sigma_{\tau}(\overline{P}_{1},
\cdots,\overline{P}_{k}),$  we have
$$
T_{u}H^{\sigma}(\mathbb{S}^n)=E_{P,t}\bigoplus
 \mathrm{span}\{\delta_{P_i,t_i},\frac{\partial \delta_{P_i,t_i}}{\partial t_i},
 \frac{\partial \delta_{P_i,t_i}}{\partial P_i} \}.
$$
Since $I'_{\tau}(u)\in T_{u}H^{\sigma}(\mathbb{S}^n),$
there exist $\xi \in E_{P,t},$ $\eta\in \mathrm{span}
\big\{\delta_{P_i,t_i},\frac{\partial \delta_{P_i,t_i}}{\partial t_i},
 \frac{\partial \delta_{P_i,t_i}}{\partial P_i} \big\} $ such that
$$
I'_{\tau}(u)=\xi+\eta.
$$
By Lemma \ref{lem4.2}, we have
\begin{align}\label{1.95}
\langle \xi, \varphi\rangle=I'_{\tau}(u)\varphi=f_{\tau}(\varphi)+2Q_{\tau}(v,\varphi)+
\langle V_{v}(\tau,\alpha,t,P,v),\varphi\rangle,\quad \forall\, \varphi \in E_{P,t},
\end{align}
where $\|V_{v}(\tau,\alpha,t,P,v)\|_{\sigma}\leq C\|v\|_{\sigma}^{2}.$
Replacing  $\varphi$  by $v$ in \eqref{1.95} and using \eqref{1.68}, we have
\be\label{4.9}
\|\xi\|_{\sigma}\geq \delta_{0}\|v\|_{\sigma}-
\|{f}_{\tau}\|-O(\|v\|_{\sigma}^{2})
\geq \frac{\delta_{0}}{2}\|v\|_{\sigma}-\|f_{\tau}\|,
\ee
where $\delta_{0}$ is as in \eqref{1.68}.
Let $\beta=(\beta_{1},\cdots,\beta_{k}),$
$\beta_{i}=\alpha_{i}-K(P_{i})^{-{1}/{2\sigma}}$ be as in Proposition \ref{lem5}, we define
$$
\widehat{\Sigma}_{\tau}=
\Big\{ u=\sum_{i=1}^{k}\alpha_{i}\delta_{P_i,t_i}+v
\in \Sigma_{\tau}(\overline{P}_{1},\cdots,\overline{P}_{k}):\|v\|_{\sigma}<\tau|\log\tau|^3,\,
|\beta|<\tau|\log\tau|^2\Big\}.
$$
It follows from Proposition \ref{prop1} and \eqref{1.18} that
  $$
  I'_{\tau}(u)\ne 0, \quad\forall\, u\in \Sigma_{\tau}(\overline{P}_{1},\cdots,\overline{P}_{k})
 \backslash \widehat{\Sigma}_{\tau}.
 $$

For any $u=\sum_{i=1}^{k}\alpha_{i}\delta_{P_i,t_i}+v\in \widehat{\Sigma}_{\tau},$
by \eqref{1.26} and Proposition \ref{lem5}, we have
$$
\begin{aligned}
\langle  \eta, \delta_{P_i,t_i}\rangle
 =&I_{\tau}'(u)\delta_{P_i,t_i}\\
 =&
 \Gamma(n-1)\Big( \alpha_{i}\int_{\sn}\delta_{P_{i}, t_{i}}^n
+\sum_{j=1}^{k} \alpha_{j}
\int_{\mathbb{S}^{n}}\delta_{P_{i}, t_{i}}^{n-1} \delta_{P_{j}, t_{j}}
\Big)\\
&
-\Gamma(n-1)\int_{\mathbb{S}^n}K
\Big|\sum_{i=1}^{k}\alpha_i\delta_{P_i,t_i}+v\Big|^{n-2-\tau}
\Big(\sum_{j=1}^{k}\alpha_j\delta_{P_j,t_j}+v\Big)\delta_{P_{i},t_{i}}\\
=&
\frac{\partial }{\partial \alpha_{i}}I_{\tau}\Big(\sum_{j=1}^{k}\alpha_{j}\delta_{P_j,t_j}+v\Big)\\
=&-2\sigma\|\delta_{P_{i},t_{i}}\|_{\sigma}^{2}\beta_i+V_{\alpha_i}(\tau,\alpha,t,P,v),
\end{aligned}
$$
 and
\be\label{4.6}
V_{\alpha_{i}}(\tau,\alp,t,P,v)= O(\tau|\log \tau|).
\ee

It follows from \eqref{1.31} and \eqref{1.32} that
$$
\begin{aligned}
\Big\langle \eta, \frac{\partial \delta_{P_i,t_i}}{\partial t_i}\Big\rangle
=&I'_{\tau}(u)\,\frac{\partial \delta_{P_i,t_i}}{\partial t_i}\\
=&\frac{1}{\alpha_{i}}
\frac{\partial }{\partial t_i}
I_{\tau}\Big(\sum_{j=1}^{k}\alpha_{j}\delta_{P_j,t_j}+v\Big)\\
=&\frac{1}{\alpha_{i}}
\Big\{
\Theta_{1}\frac{1}{K(P_{i})^{{1}/{\sigma}}}\frac{\tau}{t_{i}}
+\Theta_{2}\frac{\Delta_{g_{0}}K(P_i)}{K(P_{i})^{{n}/{2\sigma}}}\frac{1}{t_{i}^{3}}
 \Big.\\
&\phantom{=\;\;}\Big.
+\Theta_3\sum_{j\ne i} \frac{G_{P_{i}}(P_{j})}{(K(P_{i})K(P_{j}))^{{1}/{2\sigma}}}
\frac{1}{t_i^2t_j}+V_{t_{i}}(\tau,\alpha,t,P,v)
\Big\},
\end{aligned}
$$
where
\be\label{4.7}
|V_{t_i}(\tau, \alpha, t,P,v)|=o(\tau^{3/2}).
\ee

By  \eqref{1.33} and \eqref{1.34}, we obtain
$$
\begin{aligned}
\Big\langle \eta, \frac{\partial \delta_{P_i,t_i}}{\partial P_{i}}\Big\rangle
=&I_{\tau}'(u)\frac{\partial \delta_{P_i,t_i}}{\partial P_{i}}\\
=&\frac{1}{\alpha_{i}}
\frac{\partial }{\partial P_i}
I_{\tau}\Big(\sum_{i=1}^{k} \alpha_{i}\delta_{P_i,t_i}+v\Big)\\
=&\frac{1}{\alpha_{i}}
\{
-\Theta_{4}\nabla_{g_0}K(P_{i})
+V_{P_i}(\tau, \alpha, t,P, v)
\},
\end{aligned}
$$
with $V_{P_i}$ satisfying
\be\label{4.8}
|V_{P_i}(\tau, \alpha, t, P,v)|=O (\tau^{1/2}).
\ee

Using the estimates  stated above, we define a family of
operators on $\widehat{\Sigma}_{\tau}$ as follows:
for any $u=\sum_{i=1}^{k}\alpha_{i}\delta_{P_i,t_i}+v\in \widehat{\Sigma}_{\tau}$,
$$
X_{\theta}(u):=\xi_{\theta}(u)+\eta_{\theta}(u),\quad 0\leq\theta\leq 1,
$$
where, for any $\varphi\in E_{P,t},$
\be\label{xi}
\langle
\xi_{\theta},\varphi
\rangle:=
\theta f_{\tau}(\varphi)+(1-\theta)\langle v,\phi\rangle
+2\theta Q_{\tau}(\varphi, v)
+\theta\langle V_{v}(\tau,\alpha, t,P,v),\varphi \rangle,
\ee
and
\begin{align}\label{eta}
\begin{aligned}
\langle \eta_{\theta}, \delta_{P_{i},t_i}  \rangle
:=&-2\sigma\|\delta_{P_{i},t_{i}}\|_{\sigma}^{2}
\Big\{ \alpha_{i}-\frac{\theta}{K(P_{i})^{{1}/{2\sigma}}}
-\frac{1-\theta}{K(\overline{P}_{i})^{{1}/{2\sigma}}} \Big\}\\
&+\theta V_{\alpha_{i}}(\tau, \alpha,t,P,v),\\
\Big\langle \eta_{\theta}, \frac{\partial \delta_{P_i,t_i}}{\partial t_i}  \Big\rangle
:=&\Big\{
\frac{\theta}{\alpha_{i}}+(1-\theta)
\Big\}
\Big\{
\frac{\Theta_{1}}{K(P_i(\theta))^{{1}/{\sigma}}}\frac{\tau}{t_{i}}
+\frac{\Theta_{2}\Delta_{g_0}K(P_{i}(\theta))}{K(P_i(\theta))^{{n}/{2\sigma}}}
\frac{1}{t_i^3} \Big.\\
&
+\sum_{j\ne i}\frac{\Theta_{3}G_{P_{i}(\theta)}(P_{j}(\theta))}{(K(P_i(\theta))K(P_j(\theta)))^{{1}/{2\sigma}}}
\frac{1}{t_i^2t_j}
\Big\}+\frac{\theta}{\alpha_{i}}V_{t_i}(\tau,\alpha, t,P,v),\\
\Big\langle
\eta_{\theta}, \frac{\partial \delta_{P_{i}, t_{i}}}{\partial P_{i}}
\Big\rangle
:=&-\Big\{
(1-\theta)+\frac{\theta}{\alpha_{i}}\Theta_{4}
\Big\}
\nabla_{g_{0}}K(P_{i})+\frac{\theta}{\alpha_{i}}
V_{P_{i}}(\tau, \alpha, t,P, v),
\end{aligned}
\end{align}
where $P_{i}(\theta)$ is the short geodesic trajectory on $\mathbb{S}^n$ with $
P_{i}(0)=\overline{P}_{i},$ $P_{i}(1)=P_{i}.$

Obviously, $X_{1}=I_{\tau}'(u)=\xi+\eta.$
 It is well known from \eqref{fun} that $I_{\tau}'(u)$ is of
  the form Id+compact on $H^{\sigma}(\mathbb{S}^{n})$.
  From Sobolev compact imbedding theorem, the explicit
  forms of $V_{v}, V_{\alpha_{i}}, V_{t_{i}},
 V_{P_{i}},$ $A^{-2}<t_{i}^{2}\tau<A^{2},$ \eqref{4.6}, \eqref{4.7}, \eqref{4.8},
 and $\Omega_{\varepsilon_{0}/2}$ in the definition
 of $\widehat{\Sigma}_{\tau}$ is a finite dimensional submanifold of $H^{\sigma}(\mathbb{S}^n),$
we can conclude  that $X_{\theta}$ $(0\leq \theta\leq 1)$ is the form Id+compact.
Furthermore, we have $X_{\theta}\ne 0 $  on $\partial \widehat{\Sigma}_{\tau},$
$\forall\, 0\leq \theta\leq 1.$
In fact, for a given $u=\sum_{i=1}^{k} \alpha_{i}
\delta_{P_{i},t_{i}}+v\in \partial \widehat{\Sigma}_{\tau},$
 we obtain $\xi\ne0$ by using \eqref{4.9} and \eqref{2.2}.
When $\theta=0,$ $\xi_{0}=v\ne 0.$
It follows from  \eqref{xi} that $\xi_{\theta}\ne 0,$ $\forall \, 0<\theta<1.$
 By the homotopy invariance of the Leray-Schauder degree, we have
\be\label{1.36}
 \deg_{H^{\sigma }}(X_{1},\widehat{\Sigma}_{\tau}, 0)=
 \deg_{H^{\sigma}}(X_{0},\widehat{\Sigma}_{\tau},0).
 \ee
It is  easily seen from    \eqref{xi} and \eqref{eta}  that for any
$u=\sum_{i=1}^{k}\alpha_{i}\delta_{P_{i},t_{i}}+v\in \widehat{\Sigma}_{\tau},$
$$
X_{0}(u)=\xi_{0}(u)+\eta_{0}(u),
$$
where $\xi_{0}\in E_{P,t}$, $\eta_{0}\in \mathrm{span}\{\delta_{P_{i},t_i},\frac{\partial \delta_{P_{i},t_{i}}}{\partial t_{i}},\frac{\partial \delta_{P_i,t_{i}}}{\partial P_{i}}\}$ satisfy
\be\label{x0}
\begin{aligned}
\langle\xi_{0},\varphi\rangle
=&
\langle v,\varphi\rangle,
\\
\langle \eta_{0}, \delta_{P_{i},t_i}  \rangle
=&
-\beta_i2\sigma\|\delta_{P_{i},t_{i}}\|_{\sigma}^{2}
(\alpha_{i}-K(\overline{P}_{i})^{-{1}/{2\sigma}}),\\
 \Big\langle \eta_{0}, \frac{\partial \delta_{P_i,t_i}}{\partial t_i}  \Big\rangle
=&
\frac{\Theta_{1}}{K(\overline{P}_{i})^{{1}/{\sigma}}}\frac{\tau}{t_{i}}
+\frac{\Theta_{2}\Delta_{g_0}
K(\overline{P}_{i})}{K(\overline{P}_i)^{{n}/{2\sigma}}}\frac{1}{t_i^3}
+\sum_{j\ne i}\frac{\Theta_{3}G_{\overline{P}_{i}}
(\overline{P}_{j})}{(K(\overline{P}_i)K(\overline{P}_j))^{{1}/{2\sigma}}}
\frac{1}{t_i^2t_j}
,\\
\Big\langle
\eta_{0}, \frac{\partial \delta_{P_{i}, t_{i}}}{\partial P_{i}}
\Big\rangle
=&-\nabla_{g_{0}}K(P_{i}).
\end{aligned}
\ee
 Recalling the definition of $M(\overline{P}_{1},\cdots,\overline{P}_{k}).$
 From the above, we can easily get
$$
X_{0}(u)=0\quad \text{ on }\, \widehat{\Sigma}_{\tau},
$$
if and only if
\be\label{2.0003}
\begin{aligned}
&\alpha_{i}=K(\overline{P}_{i})^{-{1}/{2\sigma}}, \quad P_{i}=\overline{P}_{i},\quad v=0,\\
&\frac{\sigma}{4K(\overline{P}_{i})^{{1}/{\sigma}}}\frac{\tau}{t_{i}}
-\sum_{j=1}^{k}M_{ij}(\overline{P}_{1},\cdots,\overline{P}_{k}) \frac{1}{t_{i}^2t_{j}}=0.
\end{aligned}
\ee
For any $(s_1,\cdots, s_{k})\in \mathbb{R}^{k},$ $s_{i}>0,$ $i=1,\cdots, k,$ we define
$$
F(s_{1},\cdots,s_{k}):=-\frac{\sigma\tau}{4}
\sum_{j=1}^{k}\frac{1}{K(\overline{P}_{j})^{{1}/{\sigma}}}\log s_{j}
+\frac{1}{2}\sum_{i,j=1}^{k}M_{ij}(\overline{P}_{1},\cdots,\overline{P}_{k})
s_{i}s_{j},
$$
and for  $t_{i}=s_{i}^{-1},$
$$
\widehat{F}(t_1,\cdots,t_{k}):=F(s_{1},\cdots,s_{k}).
$$

The derivative with respect to $t_{i}$ is
$$
\frac{\partial \widehat{F}}{\partial t_{i}}
(t_{1},\cdots,t_{k})=\frac{\sigma\tau}{4K(\overline{P}_{i})^{{1}/{\sigma}}}
\frac{\tau}{t_{i}}-\sum_{j=1}^{k}M_{ij}(\overline{P}_{1},\cdots,\overline{P}_{k})\frac{1}{t_{i}^{2}t_{j}},
$$
combining this and \eqref{x0}, we have
$$
\Big\langle
\eta_{0}, \frac{\partial \delta_{P_i,t_i}}{\partial t_{i}}
\Big\rangle=\frac{\partial \widehat{F}}{\partial t_{i}}(t_{1},\cdots, t_{k}).
$$

It is obvious that $\nabla\widehat{F}(t_{1},\cdots, t_{k})=0$ if and only if $\nabla F(s_{1},\cdots,s_{k})=0$.
Since $\mu(M(\overline{P_{1}},\cdots,\overline{P}_{k}))>0,$
a trivial verification shows that $F(s_{1},\cdots,s_{k})$ is a strictly convex function, and having a unique critical
point in the first quadrant.
It follows that $\widehat{F}(t_1,\cdots,t_k)$ has unique critical point in the first quadrant with Morse index
zero. Hence $X_{0}$ has precisely one non-degenerate zero in $\widehat{\Sigma}_{\tau}$.
Furthermore, by \eqref{2.0003} we can easily obtain
  \be\label{1.35}
  \deg_{H^{\sigma}}(X_{0},\widehat{\Sigma}_{\tau},0)
  =(-1)^{k+\sum_{i=1}^{k}i(\overline{P}_{i})}.
  \ee
Combining \eqref{1.35} and \eqref{1.36},  we complete the
proof of Theorem \ref{thm3}.
\end{proof}

Recall the  definition of $\mathscr{O}_{R}$ in \eqref{1.82}. For $\delta>0$ suitably small, define
\be\label{1.85}
\mathscr{O}_{R,\delta}:=\{u\in H^{\sigma}(\mathbb{S}^{n}):
\inf_{\omega\in \mathscr{O}_{R}}\|u-\omega\|_{\sigma}<\delta\}.
\ee

\begin{proposition}\label{prop444}
Let $\sigma=1+m/2,$ $m\in \mathbb{N}_{+},$ and $n=2\sigma+2.$ Let $K\in \mathscr{A}$  be a Morse function and
$0<\tau_{0}\leq \tau \leq4/(n-2\sigma)-\tau_{0}.$
Then there exists some constants $C_{0}>0,$ $\delta_{0}>0$
depending only on $\tau_{0}$ and $K,$
such that
\be\label{0.21}
\{ u\in H^{\sigma}(\mathbb{S}^n):u>0~~  a.e.,\,I'_{\tau}(u)=0
\}\subset \mathscr{O}_{C_{0}, \delta_{0}}.
\ee
 Furthermore, we have $I_{\tau}'(u)\ne 0$ on $\partial \mathscr{O}_{C_0,\delta_{0}}$ and
\be\label{1.37}
\deg_{H^{\sigma} }(u-P_{\sigma}^{-1}(\Gamma(n-1)K|u|^{\frac{4\sigma}{n-2\sigma}-\tau}u),
\mathscr{O}_{C_{0},\delta_{0}}, 0)=-1.
\ee
\end{proposition}

\begin{proof}
From  Proposition \ref{prop3},  we know that for $\tau>0$ small there exists some
   suitable value of $\nu_{0}, A, R$ such that
  $u$ satisfying $u\in H^{\sigma}(\mathbb{S}^n)$, $u>0$, a.e., $I'_{\tau}(u)=0$ are
 either in $\mathscr{O}_{R}$ or in
 some $\Sigma_{\tau}(q^{(1)},\cdots, q^{(k)}).$
 Combining \eqref{stau}, \eqref{1.12},  \eqref{a.2}, and \eqref{p1p1},
    we conclude that there exists some
positive constants
$C_{0}$ and $\delta_{0}$ such that
  \eqref{0.21} holds.

For $K^{*}(x)=x_{(n+1)}+2,$ $x=(x_{(1)},\cdots,x_{(n+1)})
\in\mathbb{S}^n\subset\mathbb{R}^{n+1}$ and $t\in (0,1),$ we consider $K_{t}=tK+(1-t)K^{*}.$
By the homotopy invariance of the Leray-Schauder degree, we only need to establish
\eqref{1.37} for $K^*$ and $\tau$ very small.
It is easy to see that $K^{*}\in \mathscr{A}$ is a Morse function.
The proof of \eqref{1.37} is straightforward by  the Kazdan-Warner condition, Theorem \ref{thm3},
and a homotopy argument.
\end{proof}

\subsection{Proof of  Theorems \ref{thm4}, \ref{thm2} and \ref{thm5}}
Using Theorem \ref{thm3} and Proposition \ref{prop444}, we next prove Theorem \ref{thm4}.
\begin{proof}[Proof of Theorem \ref{thm4}]
The existence of $R_{0}$ can be easily obtained from Theorem \ref{thm2.1}.
For all $R\geq R_{0},$ using Theorem \ref{thm2.1}, Proposition \ref{prop3},  and
By the homotopy invariance of the Leray-Schauder degree, we have
\be\label{2.0004}
\begin{aligned}
 &\deg_{C^{2\sigma,\alpha}}(u-P_{\sigma}^{-1}(\Gamma(n-1)Ku^{n-1}), \mathscr{O}_{R}, 0)
\\
=&\deg_{C^{2\sigma,\alpha}}(u-P_{\sigma}^{-1}(\Gamma(n-1)K|u|^{2\sigma-\tau}u), \mathscr{O}_{R}, 0)
\end{aligned}
\ee
for $\tau>0$ sufficiently small.

Let $C_{0}\gg R,$ $0<\delta_{1}\ll \delta_{0},$
and  $\tau_{0}$ be given by  Proposition \ref{prop444}.
Using \eqref{1.37}, Proposition \ref{prop3}, \eqref{1.75}, \eqref{index}, and
the excision property of the degree, we have
\be\label{2.0005}
\deg_{H^{\sigma}}(u-P_{\sigma}^{-1}(\Gamma(n-1)K|u|^{2\sigma-\tau}u), \mathscr{O}_{R,\delta_{1}}, 0)
=\mathrm{Index} (K).
\ee
As in the proof of Proposition \ref{prop444}, one can check that
there  are no critical points of $I_{\tau}$ in $\overline{\mathscr{O}_{R,\delta_{1}}}\backslash \mathscr{O}_{R}.$
Using the same proof idea as Li \cite[Theorem B.2]{LYYJ}
and \cite[Theorems 2.4 and 2.5]{jlxm}, we can easily get
\be\label{2.0006}
\begin{aligned}
&\deg_{C^{2\sigma,\alpha}}(u-P_{\sigma}^{-1}(\Gamma(n-1)
K|u|^{2\sigma-\tau}u),\mathscr{O}_{R},0)\\
=&\deg_{H^{\sigma}}(u-P_{\sigma}^{-1}(\Gamma(n-1) K|u|^{2\sigma-\tau}u),\mathscr{O}_{R,\delta_{1}},0).
\end{aligned}
\ee
It follows from \eqref{2.0004}--\eqref{2.0006} that for $R\geq R_{0},$ \eqref{1.38} is proved.
Theorem \ref{thm4} follows from the above.
\end{proof}

Using Theorem \ref{thm4} and perturbing the prescribing function near its critical point,
we can know  exactly where the blow up occur when $K\notin \mathscr{A}.$

\begin{proof}[Proof of the Theorem \ref{thm2}]
Since  the Morse functions in
$C^{2}(\mathbb{S}^{n})^{*}\backslash\mathscr{A}=\partial \mathscr{A}$
are dense in $\partial \mathscr{A},$ without loss of generality we consider
the case that $K\in \partial\mathscr{A}$ is a Morse function.
 First recall the definition  of $\mathscr{K}$ and $\mathscr{K}^{+},$
 we can assume here  $\mathscr{K}\backslash\mathscr{K}^{+}=
 \{ q^{(1)}, \cdots, q^{(m)}\},$ $m\in \mathbb{N}_{+}.$
 From the definition of $\mathscr{A}$
 and $K\in \partial \mathscr{A},$ we know that
there exists $1\leq i_{1}<\cdots<i_{k}\leq m,$ $k\geq 1,$
 such that
 \be\label{2.0009}
 \mu(M( q^{(i_1)}, \cdots, q^{(i_{k})}))=0.
 \ee

 {\bf Case 1}: There is only one such $\{q^{(i_1)},\cdots,q^{(i_k)}\}$ satisfying \eqref{2.0009}.
Using the same $C^{2}$ perturbation method as in  Li \cite{LYY,LYYN},
we can obtain a smooth, one-parameter
  family of Morse functions $\{K_{t}\}$ $(-1\leq t\leq 1)$ with the following
   properties:
\begin{itemize}
  \item [(a)]$K_{t}$ $(-1\leq t\leq 1)$ are identically the same as $K$ except in some small
 balls around $q^{(i_1)},\cdots, q^{(i_k)}$ and $K_{0}=K.$
 $K_{t}$  have the same critical points with the same Morse index  for any $-1\leq t\leq 1.$
  \item [(b)] $\mu(M(K_{t}; q^{(j_1)},\cdots, q^{(j_s)}))$ have the same sign for $-1<t<1$
for any $1\leq j_{1}<\cdots<j_{s}\leq m,$ $(j_{1},\cdots, j_{s})\ne (i_1,\cdots, i_{k}).$
Furthermore,
$$
\mu(M(K_{t}; q^{(i_1)},\cdots, q^{(i_{k})}))
\begin{cases}
<0  , & \text{ if }\, -1<t<0, \\
=0  , &\text{ if }\, t=0,\\
>0,   & \text{ if }\, 0<t<1.
\end{cases}
$$
\end{itemize}
It is easily seen that  $K_{t}\in \mathscr{A}$ when $t\ne 0.$  From the definition of
$\mathrm{Index},$  we have
$$
\mathrm{Index}(K_1)
=\mathrm{Index}(K_{-1})+(-1)^{k-1+\sum_{j=1}^{k}i(q^{(i_{j})})},
$$
evidently, $\mathrm{Index}(K_{1})\ne \mathrm{Index}(K_{-1}).$

By the homotopy invariance of
the Leray-Schauder degree and Theorem \ref{thm4},
there exists $t_i$ and $v_{i}\in \mathscr{M}_{K_{t_i}},$
such that
$$
\lim_{i\rightarrow \infty} \|v_{i}\|_{C^{2\sigma, \alpha}(\mathbb{S}^n)}
=\infty\quad \text{ or }\quad \lim_{i\rightarrow \infty} (\min_{\mathbb{S}^n} v_{i})=0.
$$

In fact, we can prove that if $\lim_{i\to \infty }(\min_{\mathbb{S}^{n}}v_{i})=0$,
then $\lim_{i\to \infty } \|v\|_{C^{2\sigma, \alpha}}=\infty.$
If not, it means that $\{v_{i}\}$ has no blow up point,
 then $v_{i}\equiv 0 $ on $\mathbb{S}^{n}$ can be obtained from
 $\lim_{i\to \infty }(\min_{\mathbb{S}^{n}}v_{i})=0$
 and Hanarck inequality. This leads to contradictions and
we deduce that \eqref{1.72}holds.

It follows from $K_{t}\in \mathscr{A}$ $(t\ne 0)$ and Theorem \ref{thm2.1}
that $t_{i}\rightarrow 0,$ namely, $K_{t_{i}}\rightarrow K.$
Then by Theorem \ref{thm1}, we can know that
 $\{v_{i}\}$ blows up exactly at $k$ points $q^{(i_{1})},\cdots,q^{(i_{k})}.$

{\bf Case 2}:
If  $\{q^{(i_{1})},\cdots,q^{(i_{k})}\}$
satisfying \eqref{2.0009} is not unique,
 we can perturb as described above  the function $K$  near its some critical
points to change the Hessian matrix of $K$ at these points,
such that there exists a sequence of Morse functions ${K_{\ell}}$ satisfying:
 $K_{\ell}\rightarrow K,$  ${K_{\ell}}$ are identically the same as $K$ except in some small
 balls and have the same critical points with the
same Morse index; there is only one such $(i_{1},\cdots, i_{k})$
such that \eqref{2.0009} is true for any $\ell.$
From Case 1,  we know that there exists a sequence of $K_{i}\to K$ in $C^{2}(\mathbb{S}^n),$
$v_{i}\in \mathscr{M}_{K_{i}}$ such that $\{v_{i}\}$ blows up at precisely the $k$ points
$q^{(i_{1})},\cdots,q^{(i_{k})}.$
We have thus proved Theorem \ref{thm2}.
\end{proof}

Using Theorem \ref{thm1} and the proof method of Theorem \ref{thm2}, we show Theorem \ref{thm5} holds.
\begin{proof}[Proof of Theorem \ref{thm5}]
By using Theorem \ref{thm1} we can prove the Part (i) of Theorem \ref{thm5}.
The Part (ii) of Theorem \ref{thm5} is similar to the proof of Theorem \ref{thm2}, we omit it here.
\end{proof}

\appendix

\section{Appendix}\label{sec2}

In this section, we review some results about the local analysis and blow
up profiles for nonlinear integral equations obtained in Jin-Li-Xiong \cite{jlxm}.
For any $x\in \mathbb{R}^{n}$ and $r>0,$
the symbol ${B}_{r}(x)$ denotes the ball in $\mathbb{R}^{n}$
with radius $r$ and center $x$, and $B_{r}:=B_{r}(0).$
\subsection{H\"older estimates and Schauder type estimates}
Consider nonnegative solutions of the integral equation
\begin{align}\label{2.41}
u(x)=\int_{\mathbb{R}^{n}} \frac{V(y) u(y)}{|x-y|^{n-2 \sigma}} \mathrm{d} y
\quad \text { a.e \, in } B_{3},
\end{align}
where $0<\sigma<n / 2$.

The  H\"older estimates for solutions to \eqref{2.41} is following:
\begin{proposition}\label{thm22}
For $n \geq 1$, $0<\sigma<n / 2,$ $r>n /(n-2 \sigma)$ and $p>n / 2 \sigma$, let $0 \leq V \in L^{p}(B_{3}),$
$0 \leq u \in L^{r}(B_{3})$ and $0 \leq V u \in L_{l o c}^{1}(\mathbb{R}^{n})$.
If u satisfies \eqref{2.41}, then $u \in C^{\alpha}(B_{1})$,
$$
\|u\|_{C^{\alpha}(B_{1})} \leq C\|u\|_{L^{r}(B_{3})},
$$
and u satisfies the Harnack inequality
$$
\max _{\bar{B}_{1}} u \leq C \min _{\bar{B}_{1}} u,
$$
where $C>0$ and $\alpha \in(0,1)$ depend only on $n,$ $\sigma,$ $p$, and an upper bound
of $\|V\|_{L^{p}(B_{3})}$.
\end{proposition}

The Schauder type estimates for solutions $u$ to \eqref{2.41} is following:
\begin{proposition}
 In addition to the assumptions in Proposition \ref{thm22},
 we assume that $V \in$ $C^{\alpha}(B_{3})$ for some $\alpha>0$
 but not an integer, then $u \in C^{2 \sigma+\alpha^{\prime}}(B_{1})$ and
$$
\|u\|_{C^{2 \sigma+\alpha^{\prime}}(B_{1})} \leq C\|u\|_{L^{r}(B_{3})},
$$
where $\alpha^{\prime}=\alpha$
if $2 \sigma+\alpha \notin \mathbb{N}_{+}$, otherwise $\alpha^{\prime}$ can be any positive constant less than $\alpha$. Here $C>0$ depends only on $n, \sigma, \alpha$ and an upper bound of $\|V\|_{C^{\alpha}(B_{3})}$.
\end{proposition}

\subsection{Blow up profiles for nonlinear integral equations}

\begin{proposition}[Pohozaev type identity]\label{prop1.1}
Let $u \geq 0$ in $\mathbb{R}^{n}$, and $u \in C(\overline{B}_{R})$ be a solution of
$$
u(x)=\int_{B_{R}} \frac{K(y) u(y)^{p}}{|x-y|^{n-2 \sigma}} \,\ud y+h_{R}(x),
$$
where $1<p \leq \frac{n+2 \sigma}{n-2 \sigma},$ and $h_{R}(x) \in C^{1}(B_{R}),$ $\nabla h_{R} \in L^{1}(B_{R}).$ Then
\begin{align*}
&\Big(\frac{n-2 \sigma}{2}-\frac{n}{p+1}\Big) \int_{B_{R}} K(x) u(x)^{p+1} \,\ud x
-\frac{1}{p+1} \int_{B_{R}} x \nabla K(x) u(x)^{p+1} \,\ud x \\
=& \frac{n-2 \sigma}{2} \int_{B_{R}} K(x) u(x)^{p} h_{R}(x) \,\ud x
+\int_{B_{R}} x \nabla h_{R}(x) K(x) u(x)^{p} \,\ud x \\
&-\frac{R}{p+1} \int_{\partial B_{R}} K(x) u(x)^{p+1} \, \ud s.
\end{align*}
\end{proposition}

\begin{proposition}\label{prop2}
Suppose that $0 \leq u_{i} \in L_{{loc}}^{\infty}(\mathbb{R}^{n})$ satisfies \eqref{2.1}
with $K_{i}$ satisfying \eqref{2.2111}. Suppose that $x_{i} \rightarrow 0$ is an isolated blow up point of $\{u_{i}\}$, i.e., for some positive constants $A_{3}$ and $\bar{r}$ independent of $i$,
$$
|x-x_{i}|^{2 \sigma /(p_{i}-1)}u_{i}(x) \leq A_{3}\quad \text { for all }\, x \in B_{\bar{r}} \subset \Omega.
$$
Then for any $0<r<\bar{r}/3$, we have the following Harnack inequality
$$
\sup _{B_{2r}(x_{i}) \backslash \overline{B_{r / 2}(x_{i})}} u_{i}
\leq C \inf_{B_{2 r}(x_{i})\backslash \overline{B_{r/2}(x_{i})}} u_{i},
$$
where $C$ is a positive constant depending only on $\sup_{i}\|K_{i}\|_{L^{\infty}(B_{\bar{r}}(x_{i}))},
n, \sigma, \bar{r}$ and $A_{3}.$
\end{proposition}

\begin{proposition}\label{prop3.1}
Assume the hypotheses  in Proposition \ref{prop2}.
Then for every $R_{i} \rightarrow \infty$, $\varepsilon_{i} \rightarrow 0^{+},$
 we have, after passing to a subsequence (still denoted as $\{u_{i}\},$ $\{x_{i}\},$ etc.), that
$$
\|m_{i}^{-1} u_{i}(m_{i}^{-(p_{i}-1) /2\sigma}\cdot+x_{i})
-(1+k_{i}|\cdot|^{2})^{(2 \sigma-n) / 2}\|_{C^{2}(B_{2 R_{i}}(0))} \leq \varepsilon_{i},
$$
$$
r_{i}:=R_{i} m_{i}^{-(p_{i}-1) / 2 \sigma} \rightarrow 0  \quad \text { as }\, i \rightarrow \infty,
$$
where $ m_{i}:=u_{i}(x_{i})$ and $ k_{i}:=({K_{i}(x_{i}) \pi^{n/2}\Gamma(\sigma)}/{\Gamma(\frac{n}{2}+\sigma)})^{1/\sigma}.$
\end{proposition}

\begin{proposition}\label{prop9}
Under the  hypotheses  of Proposition \ref{prop3.1},
there exists a positive constant $C=C(n, \sigma, A_{1}, A_{2}, A_{3})$ such that,
$$u_{i}(x) \geq C^{-1} m_{i}(1+k_{i} m_{i}^{(p_{i}-1) / \sigma}|x-x_{i}|^{2})^{(2 \sigma-n) / 2}\quad
\text{for all} \quad |x-x_{i}| \leq 1.$$
In particular, for any $e \in \mathbb{R}^{n},|e|=1$, we have
$$
u_{i}(x_{i}+e) \geq C^{-1} m_{i}^{-1+((n-2 \sigma) / 2 \sigma) \tau_{i}}
$$
where $\tau_{i}=(n+2 \sigma) /(n-2 \sigma)-p_{i}$.
\end{proposition}

\begin{proposition}\label{prop4}
Under the hypotheses of Proposition \ref{prop2} with $\bar{r}=2,$
and in addition that $x_{i} \rightarrow 0$ is also an isolated simple blow up point with constant $\rho,$
we have
$$
\tau_{i}=O(u_{i}(x_{i})^{-c_{1}+o(1)})\quad \text{and}\quad
u_{i}(x_{i})^{\tau_{i}}=1+o(1),
$$
where $c_{1}=\min \{2,2 /(n-2 \sigma)\}$.  Moreover,
$$
u_{i}(x) \leq C u_{i}^{-1}(x_{i})|x-x_{i}|^{2 \sigma-n} \quad \text { for all }\,|x-x_{i}| \leq 1.
$$
\end{proposition}

\begin{proposition}\label{prop7}
Under the hypotheses of Proposition \ref{prop4},
let
\begin{align*}
T_{i}(x):=&
u_{i}(x_{i})
\int_{B_{1}\left(x_{i}\right)} \frac{K_{i}(y) u_{i}(y)^{p_{i}}}{|x-y|^{n-2 \sigma}} \mathrm{d} y
+
u_{i}(x_{i}) \int_{\mathbb{R}^{n} \backslash B_{1}(x_{i})} \frac{K_{i}(y) u_{i}(y)^{p_{i}}}{|x-y|^{n-2 \sigma}} \mathrm{d} y\\
=&:T_{i}^{\prime}(x)+T_{i}^{\prime \prime}(x).
\end{align*}
Then, after passing a subsequence,
$$
T_{i}^{\prime}(x) \rightarrow a|x|^{2 \sigma-n}
 \quad \text { in } \, C_{loc}^{2}(B_{1} \backslash\{0\})
$$
and
$$
T_{i}^{\prime \prime}(x) \rightarrow h(x) \quad \text { in } \, C_{l o c}^{2}(B_{1})
$$
for some $h(x) \in C^{2}(B_{2})$, where
$$
a=\Big(\frac{\pi^{n / 2} \Gamma(\sigma)}
{\Gamma(\frac{n}{2}+\sigma)}\Big)^{-\frac{n}{2 \sigma}} \int_{\mathbb{R}^{n}}\Big(\frac{1}{1+|y|^{2}}\Big)^{\frac{n+2 \sigma}{2}} \mathrm{~d} y \lim _{i \rightarrow \infty} K_{i}(0)^{\frac{2 \sigma-n}{2 \sigma}}.
$$
Consequently, we have
$$
u_{i}(x_{i}) u_{i}(x) \rightarrow a|x|^{2 \sigma-n}
+h(x) \quad\text { in } \, C_{l o c}^{2}(B_{1} \backslash\{0\}).
$$
\end{proposition}

\begin{proposition}\label{prop5}
Under the hypotheses of Proposition \ref{prop4},  we have
$$
\int_{|x-x_{i}| \leq r_{i}}|x-x_{i}|^{s} u_{i}(x)^{p_{i}+1}\,\ud x=
\begin{cases}
O(u_{i}(x_{i})^{-2 s /(n-2 \sigma)}), & -n<s<n, \\
O(u_{i}(x_{i})^{-2 n /(n-2 \sigma)} \log u_{i}(x_{i})), & s=n, \\
o(u_{i}(x_{i})^{-2 n /(n-2 \sigma)}), & s>n,
\end{cases}
$$
and
$$
\int_{r_{i}<|x-x_{i}| \leq 1}|x-x_{i}|^{s} u_{i}(x)^{p_{i}+1}\,\ud x=
\begin{cases}
o(u_{i}(x_{i})^{-2 s /(n-2 \sigma)}), & -n<s<n, \\
O(u_{i}(x_{i})^{-2 n /(n-2 \sigma)} \log u_{i}(x_{i})), & s=n, \\
O(u_{i}(x_{i})^{-2 n /(n-2 \sigma)}), & s>n ,
\end{cases}
$$
where $r_{i}$ is as in Proposition \ref{prop3.1}.
\end{proposition}

\begin{proposition}\label{prop8}
Let $\sigma=1+m/2,$ $m\in\mathbb{N}_{+}$, $n=2\sigma+2$,
and $K_{i}\to K$ in $C^{2}(B_{3})$.
Let $p_{i}\leq \frac{n+2\sigma}{n-2\sigma}=n-1$, $p_{i}\to n-1,$
and $\tau_{i}=n-1-p_{i}.$
 Let $u_{i}(x)$ satisfy
$$
u_{i}(x)=\int_{\mathbb{R}^{n}}
 \frac{K_{i}(y) H(y)^{\tau_{i}}(y) u_{i}(y)^{p_{i}}}{|x-y|^{2}} \mathrm{d} y \quad \text { for } \quad x \in B_{3},
$$
where $H(y)=2/(1+|y|^2).$
Let $x_{i} \rightarrow 0$ is an isolated simple blow up point
of $\{u_{i}\}$ with constant $A_{3}$ and $\rho$, i.e.,
$
|x-x_{i}|^{(p_{i}-1) / 2 \sigma} u_{i}(x) \leq A_{3},
$
and $r^{\frac{p_{i}-1}{2 \sigma}} \bar{u}_{i}(r)$
has precisely one critical point in $(0, \rho)$ for large $i$,
where $\bar{u}_{i}(r)=$ $f_{\partial B_{r}(x_{i})} u_{i} \mathrm{~d}s$.

Then there exists some constants $C_1$, $C_2$ depending only on $n$, $A_3$,
$\|K\|_{C^{2}(B_{3})}$, $\rho$, such that
$$
|\nabla K_{i}(x_i)|\leq C_{1}u_{i}(x_{i})^{-1},\quad
\tau_{i}\leq C_{2} u_{i}(x_{i})^{-2}.
$$
\end{proposition}

\section{Appendix}\label{sec5}
In this appendix, we provide some estimates that can be verified by
elementary calculations  which have been used in the proof of Theorem \ref{thm4}.

For $P\in \mathbb{S}^n$ and $t>0,$ let
$$
\delta_{P,t}(x)=\frac{t}
{1+\frac{t^{2}-1}{2}(1-\cos\, d(x, P))}
,\quad x\in \mathbb{S}^{n},
$$
where $d(\cdot\, , \,\cdot )$ is the distance induced by the standard metric of $\mathbb{S}^{n}.$
Let $P$ be the south pole of $\mathbb{S}^{n}$ and make a stereographic
projection with respect to the equatorial plane, we then have
$$
\delta_{P, t}(y)=\frac{t(1+|y|^{2})}{1+t^{2}|y|^{2}},
\quad \forall\, y \in \mathbb{R}^{n}.
$$

\begin{lemma}\label{cplema1}
Let  $2\leq\alpha \leq \beta,$ there exists a positive constant
$C$ depending only on $\beta$ such that, for any $a \geq 0,$ $b \in \mathbb{R},$
$$
\Big|
| a+b|^{\alpha-1}(a+b)-a^{\alpha}-\alpha a^{\alpha-1} b-\frac{\alpha(\alpha-1)}{2} a^{\alpha-2} b^{2}
 \Big|
  \leq C(|b|^{\alpha}+a^{\gamma}|b|^{\alpha-\gamma}),
$$
where $\gamma=\max\{0,\alpha-3\}.$
\end{lemma}

\begin{lemma}\label{lem3}
For any $2\leq \alpha\leq 3$ and  any $a,b\geq 0,$
there exists some universal
 constant $C>0$ such that for any $a,b\geq 0,$ we have
$$
|(a+b)^{\alpha}-a^{\alpha}-b^{\alpha}-\alpha a^{\alpha-1}b|\leq C
 a^{\alpha-2}b^{2},
$$
$$
|(a+b)^{\alpha}-a^{\alpha}-b^{\alpha}|\leq C|a^{\alpha-1}b+ab^{\alpha-1}|.
$$
For any $1\leq \alpha\leq 2,$ there exists some universal
 constant $C>0$ such that for any $a,b\geq 0,$ we have
$$
 |(a+b)^{\alpha}-a^{\alpha}|\leq C(a^{\alpha-1}b+b^{\alpha}).
$$
\end{lemma}

\begin{lemma}
We have
$$
\int_{\mathbb{R}^{n}}\frac{1}{(1+|x|^2)^n}=
\frac{(n-2)|\mathbb{S}^{n-1}|}{4(n-1)}
\ub(\frac{n}{2},\frac{n}{2}-1),
$$
$$
\int_{\mathbb{R}^n}\frac{|x|^2}{(1+|x|^2)^n}
=\frac{n|\mathbb{S}^{n-1}|}{4(n-1)}\ub(\frac{n}{2},\frac{n}{2}-1),
$$
$$
\int_{\mathbb{R}^n} \frac{|x|^2-1}{(1+|x|^2)^n}
=\frac{|\mathbb{S}^{n-1}|}{2(n-1)}\ub(\frac{n}{2},\frac{n}{2}-1),
$$
where $\ub(\frac{n}{2},\frac{n}{2}-1)$ is the Beta function.
\end{lemma}

\begin{lemma}\label{lema.1}
 Let $\varepsilon_{0}, \tau>0$ be suitably small and $A>0$ be suitably large.
 Let $A^{-1} \tau^{-1 / 2}<t_{1}, t_{2}<A \tau^{-1 / 2},$
 $P_{1},P_{2}\in \mathbb{S}^n,$ $|P_{1}-P_{2}|\geq \varepsilon_{0},$
 $\delta_{P_{i},t_{i}}$ be as in
 \eqref{delta} and $G_{P_1}(P_2)$ be as in \eqref{1.8},
 where $|P_{1}-P_{2}|$ represents the distance between  two points $P_{1}$
 and $P_{2}$  after through a stereographic  projection$.$
Then, we have,
\begin{align}\label{a.2}
\int_{\mathbb{S}^n}\delta_{P_1, t_1}^2\delta_{P_2,t_2}
=2^{n+1}\Big( \int_{\mathbb{R}^n} \frac{1}{(1+|x|^2)^{n-1}}\Big)\frac{G_{P_{1}}(P_2)}{t_{1}t_{2}}
+O(\tau^{2}),
\end{align}
\begin{align}\label{a.3}
\int_{\mathbb{S}^n}\delta_{P_1,t_1}^{n-1-\tau}\delta_{P_2,t_2}=O(\tau),
\end{align}
\begin{align}\label{a.1}
\frac{\partial}{\partial t_{1}}
\int_{\mathbb{S}^n}\delta_{P_1,t_{1}}^{n-1}\delta_{P_2, t_{2}}=
-(n-1)2^{n+1}\frac{G_{P_{1}}(P_2)}{t_{1}^{2}t_{2}}
\Big(\int_{\mathbb{R}^{n}}\frac{|x|^2-1}{(1+|x|^2)^n}\Big)
+O(\tau^2),
\end{align}
\begin{align}\label{part3-tau}
\frac{\partial}{\partial t_1}\int_{\mathbb{S}^n}
\delta^{n-\tau}_{P_1,t_1}=-\frac{\tau}{t_1}\int_{\mathbb{R}^n}
\frac{2^n}{(1+|x|^2)^{n}}+O(\tau^{\frac{5}{2}}|\log \tau|),
\end{align}
\begin{align}\label{a.5}
\frac{\partial}{\partial t_1}
\int_{\mathbb{S}^n}
|P-P_1|^2\delta_{P_1,t_1}^{n-\tau}=
-\frac{2^{n+1}}{t_1^3}
\int_{\mathbb{R}^n}
\frac{|x|^{2}}{(1+|x|^2)^n}+O(\tau^{\frac{5}{2}}|\log \tau|).
\end{align}
\end{lemma}

\begin{lemma}
Under the hypotheses of Lemma \ref{lema.1}, in addition that $\Theta_{5},\Theta_{6}$ are
 positive constants independent of $\tau.$
Then, we have,
\begin{align}\label{p1p1}
\langle \delta_{P_{1}, t_{1}}, \delta_{P_{1}, t_{1}}\rangle=
2^{n-1}|\mathbb{S}^{n-1}|\mathrm{B}(\frac{n}{2},\frac{n}{2}),
\end{align}
\begin{align}\label{a.10}
\langle \delta_{P_{1}, t_{1}}, \delta_{P_{2}, t_{2}}\rangle
=O(\tau),
\end{align}
\begin{align}\label{part11}
\Big\langle
\frac{\partial\delta_{P_1, t_1}}{\partial t_{1}},
\frac{\partial\delta_{P_1, t_1}}{\partial t_{1}}
\Big\rangle=\Theta_5 t_1^{-2}=O(\tau),
\end{align}
\begin{align}\label{P1P1}
\Big\langle \frac{\partial\delta_{P_1,t_1}}
{\partial P_1^{(\ell)}},
\frac{\partial{\delta_{P_1,t_1}}}
{\partial P_1^{(\ell)}}
\Big\rangle=\Theta_{6}t_1^2,
\quad
\Big\langle \frac{\partial\delta_{P_1,t_1}}
{\partial P_1^{(\ell)}},
\frac{\partial{\delta_{P_1,t_1}}}
{\partial P_1^{(m)}}
\Big\rangle=0,\, \forall\, \ell\neq m,
\end{align}
\begin{align}\label{1-tau}
\begin{aligned}
&\|\delta_{P_1,t_1}^{n-2-\tau}\delta_{P_2, t_2}\|_{L^{n/(n-1)}
(\mathbb{S}^n)}= O(\tau),\\
&\|\delta_{P_{1},t_1}^{n-3-\tau}\delta_{P_2,t_2}^2\|_{L^{n/(n-1)}(\mathbb{S}^n)}=O(\tau),
\end{aligned}
\end{align}
\begin{align}\label{a.9}
\begin{aligned}
&\|\delta_{P_1,t_1}^{n-1-\tau}-\delta_{P_1, t_1}^{n-1} \|_{L^{n/(n-1)}(\mathbb{S}^n)}=
O(\tau |\log \tau|),\\
&\|\delta_{P_1,t_1}^{n-2-\tau}-\delta_{P_1, t_1}^{n-2} \|_{L^{n/(n-2)}(\mathbb{S}^n)}
=O(\tau |\log\tau|),
\end{aligned}
\end{align}
\begin{align}\label{a.71}
\|\delta_{P_1,t_1}^{n-\tau}-\delta_{P_1,t_1}^n\|_{L^{1}(\mathbb{S}^n)}
=O(\tau|\log \tau|),
\end{align}
\begin{align}\label{p-p1}
\begin{aligned}
&\big\||\cdot-P_1|\delta_{P_1, t_1}^{n-1}\big\|_{L^{n/(n-1)}(\mathbb{S}^n)}=O( \tau^{1/2}),\\
&
\big\||\cdot-P_1|^2 \delta^{n-1}_{P_{1}, t_1}\big\|_{L^{n/(n-1)}(\mathbb{S}^n)}
=O(\tau),
\end{aligned}
\end{align}
\begin{align}\label{n.1}
\begin{aligned}
&\big\| |\cdot-P_{1}|\delta_{P_{1},t_{1}}^{n-2}|\big\|_{L^{n/(n-2)}(\mathbb{S}^n)}=O(\tau^{1/2}),
\\
&\big\| |\cdot-P_{1}|\delta_{P_{1},t_{1}}^{n-2-\tau}|\big\|_{L^{n/(n-2)}(\mathbb{S}^n)}=O(\tau^{1/2}),
\end{aligned}
\end{align}
\begin{align}\label{n.6}
\Big\|\delta_{P_{1},t_{1}}^{n-3-\tau}\delta_{P_{2},t_{2}}^2
\frac{\partial \delta_{P_{1},t_{1}}}{\partial t_{1}}
\Big\|_{L^{1}(\sn)}=o(\tau^{3/2}),
\end{align}
\begin{align}\label{a.11}
\Big\|
\delta_{P_2,t_2}^{n-2-\tau}
\frac{\partial \delta_{P_1,t_1}}{\partial t_1}\Big\|_{L^{n/(n-1)}(\mathbb{S}^n)}
= O(\tau^{3/2}),
\end{align}
\be\label{n.2}
\Big\|\delta_{P_{1},t_{1}}^{n-3-\tau}\delta_{P_{2},t_{2}}\frac{\partial \delta_{P_{1},t_{1}}}
{\partial t_{1}}\Big\|_{L^{n/(n-1)}(\mathbb{S}^n)}
=O(\tau^{3/2}).
\ee
\end{lemma}

\begin{lemma}
In addition to the hypotheses of Lemma \ref{lema.1}, we assume that $ K\in C^{1}(\mathbb{S}^n)$. Then
\begin{align}\label{n.3}
\frac{\partial}{\partial t_{1}}\int_{\sn}(K-K(P_{1}))\delta_{P_{2},t_{2}}
\delta_{P_{1},t_{1}}^{n-1-\tau}=O(\tau^2),
\end{align}
\begin{align}\label{n.4}
\frac{\partial }{\partial t_{1}}\int_{\sn}(K-K(P_{2}))
\delta_{P_{1},t_{1}}\delta_{P_{2},t_{2}}^{n-1-\tau}=O(\tau^2).
\end{align}
\end{lemma}

\begin{lemma}
Let $\varepsilon_0, \tau, A$ be as in Lemma \ref{lema.1}, $P_{1},P_2,P_3 \in \mathbb{S}^n$
satisfy $|P_i-P_j|\geq \varepsilon_0,$  $i\neq j,$
 and $A^{-1}\tau^{-1/2}<t_1,t_2,t_3\leq A\tau^{-1/2}.$ Then,  we have,
 \begin{align}\label{n.61}
\Big\| \delta_{P_{2},t_2}^{n-2-\tau}\delta_{P_{3},t_3}
\frac{\partial \delta_{P_{1},t_1}}{\partial t_{1}}\Big\|_{L^{1}(\mathbb{S}^{n})}
=o(\tau^{3/2}),
\end{align}
\begin{align}\label{n.8}
\int_{\sn} \delta_{P_1,t_1}^{n-2-\tau}\delta_{P_{2},t_2}
\Big|\frac{\partial\delta_{P_1, t_1}}{\partial P_1}\Big|=O(\tau^{1/2}),
\end{align}
\begin{align}\label{n.5}
\Big\|\delta_{P_{1},t_{1}}^{n-3-\tau}\delta_{P_{2},t_{2}}
\Big|\frac{\partial \delta_{P_{1},t_{1}}}{\partial P_{1}
}\Big|\Big\|_{L^{n/(n-1)}(\sn)}=O(\tau^{1/2}),
\end{align}
\be\label{n.7}
\int_{\mathbb{S}^n}
\delta_{P_{1},t_1}^{n-3-\tau}\delta_{P_2,t_2}^{2}
\Big|\frac{\partial \delta_{P_{1},t_1}}{\partial P_{1}}\Big|=O(\tau^{3/2}),
\ee
\be\label{n.9}
\Big|\frac{\partial}{\partial P_{1}}
\int_{\sn}\delta_{P_{2},t_2}^{n-1-\tau}\delta_{P_1,t_1}\Big|=O(\tau).
\ee
\end{lemma}

\bibliographystyle{plain}

\def\cprime{$'$}

\end{document}